\documentclass{amsart}

\usepackage[T1]{fontenc}
\usepackage{lmodern}
\usepackage{amsmath,amssymb,mathtools}
\usepackage{bm}
\usepackage{enumitem}
\usepackage{booktabs}
\usepackage{graphicx}
\usepackage{float}
\usepackage{microtype}
\usepackage[letterpaper,textwidth=6.5in,textheight=8.95in,centering]{geometry}
\usepackage[hidelinks,pdfusetitle]{hyperref}

\newtheorem{theorem}{Theorem}[section]
\newtheorem{proposition}[theorem]{Proposition}
\newtheorem{lemma}[theorem]{Lemma}
\newtheorem{corollary}[theorem]{Corollary}
\theoremstyle{definition}
\newtheorem{definition}[theorem]{Definition}
\newtheorem{assumption}[theorem]{Assumption}
\numberwithin{equation}{section}
\theoremstyle{remark}
\newtheorem{remark}[theorem]{Remark}

\newcommand{\R}{\mathbb{R}}

\newcommand{\Sph}{\mathbb{S}}
\newcommand{\dt}{\Delta t}
\newcommand{\Isp}{I^{\mathrm{sp}}}
\newcommand{\norm}[1]{\left\lVert #1\right\rVert}
\newcommand{\abs}[1]{\left\lvert #1\right\rvert}
\newcommand{\tr}{\operatorname{Tr}}

\title[Error Bounds for Boundary-Stopped Cubature Schemes]
{Error Bounds for Boundary-Stopped Positive Cubature Schemes for Degenerate
	Hamilton--Jacobi--Bellman Equations}
\author{Haoran Xu}
\address{School of Mathematical Sciences, Soochow University,
	Suzhou 215006, China}
\email{20244007005@stu.suda.edu.cn}
\author{Xingye Yue}
\address{Center for Financial Engineering, School of Mathematical Sciences,
	Soochow University, Suzhou 215006, China}
\email{xyyue@suda.edu.cn}
\thanks{Corresponding author: Xingye Yue.}
\date{}
\hypersetup{
	pdftitle={Error Bounds for Boundary-Stopped Positive Cubature Schemes for Degenerate Hamilton--Jacobi--Bellman Equations},
	pdfauthor={Haoran Xu and Xingye Yue},
	pdfkeywords={Hamilton--Jacobi--Bellman equation, monotone scheme, first-contact stopping, positive cubature, Dirichlet boundary condition, viscosity solution}
}

\begin{document}
	\begin{abstract}
We derive error bounds for a boundary-stopped positive cubature scheme for
degenerate parabolic Hamilton--Jacobi--Bellman equations with Dirichlet data
on bounded domains.
The scheme balances antipodal branches at their first boundary contacts and
uses a control-dependent effective time.  Its stopped-barrier remainder is
\(O(\tau\Delta t^{\gamma/2})\), and the total mass of interpolated branches
is at most \(2\tau/\Delta t\), where \(\tau\) is the effective time.
After division by the effective time, these estimates give a consistency
bound that is uniform as \(\tau\to0\).  Under a common strict boundary barrier and the stated mesh
conditions, interior consistency, coefficient shaking, switching, and
localized comparison yield
\[
\|(u-u_h)^+\|_\infty\le C(\Delta t^{1/4}+h\Delta t^{-1/2}),
\qquad
\|(u_h-u)^+\|_\infty\le C(\Delta t^{1/10}+h^{1/2}\Delta t^{-1/4}).
\]
Thus \(\Delta t\asymp h^{10/7}\) gives an \(O(h^{1/7})\) maximum-norm
bound.  The analysis applies to any fixed centrally symmetric positive
degree-two cubature and permits rank-deficient diffusion.  Numerical examples examine parallel heat loss to a cold wall and
control-dependent, rank-deficient diffusion.
\end{abstract}

\maketitle
	
	\section{Introduction}
\label{sec:introduction}

Hamilton--Jacobi--Bellman (HJB) equations describe value functions in
stochastic optimal control and arise in applications such as option
pricing under uncertain model parameters \cite{KushnerDupuis2001,MaForsyth2017}.
Their diffusion matrices may be anisotropic or rank deficient, and their
solutions need not be smooth. The viscosity solution framework provides
comparison and well-posedness for such equations
\cite{CrandallIshiiLions1992}. At the discrete level, monotonicity
preserves order between approximate solutions. Together with stability,
consistency, and comparison for the limiting problem, it provides a
general route to convergence
\cite{BarlesSouganidis1991,CrandallLions1996}.

Constructing a monotone approximation becomes more difficult when mixed
second derivatives are present. On a prescribed compact stencil,
nonnegative coefficients impose restrictions on the diffusion tensor
and the mesh \cite{MotzkinWasow1952,Reisinger2018}.
Related compatibility conditions arise in finite volume methods for
anisotropic diffusion
\cite{NordbottenAavatsmarkEigestad2007,ShengYuan2008};
Droniou \cite{Droniou2014} reviews the roles of coercivity and discrete
maximum principles in this setting. For divergence-form problems,
Droniou and Le Potier \cite{DroniouLePotier2011} construct schemes
satisfying local maximum and minimum principles, while Canc\`es,
Cathala, and Le Potier \cite{CancesCathalaLePotier2013} enforce these
bounds through nonlinear corrections of finite volume schemes.
Another approach adapts the stencil directions to the diffusion tensor,
as in the sparse nonnegative construction of Fehrenbach and Mirebeau
\cite{FehrenbachMirebeau2014}. These methods illustrate how the choice
of stencil and the form of the discrete operator affect preservation
of the continuous bounds.

For HJB equations, positive wide stencils have been developed through
generalized finite differences \cite{BonnansZidani2003} and
semi-Lagrangian approximations
\cite{CamilliFalcone1995,DebrabantJakobsen2013}.
The latter evaluate the solution along diffusion directions and use
positive interpolation at points outside the spatial grid. This allows
the diffusion directions to vary with the control without requiring
diagonal dominance in the grid coordinates. Ma and Forsyth
\cite{MaForsyth2017} use a combination of compact and rotated wide
stencils for the two-factor uncertain-volatility equation.
The probabilistic interpretation connects these constructions to
locally consistent Markov-chain approximations
\cite{KushnerDupuis2001}. Probabilistic schemes have also been developed
for broader classes of fully nonlinear parabolic equations, with
structural conditions ensuring monotonicity
\cite{FahimTouziWarin2011,GuoZhangZhuo2015}.

The present work continues our study of expectation-based
positivity-preserving discretizations
\cite{RenXuYue2024,XuRenYue2026,XuLiYue2026}.
For linear parabolic equations, Xu, Ren, and Yue
\cite{XuRenYue2026} combine the Feynman--Kac representation with
positive wide-stencil interpolation and incorporate Dirichlet,
Neumann, and periodic data through stopping, reflection, and wrapping.
Xu, Li, and Yue \cite{XuLiYue2026} develop the corresponding approach
for linear elliptic problems. These two works establish error estimates
under smoothness assumptions on the solution. Here we study a controlled
parabolic equation: both the diffusion directions and the branch
durations depend on the control, and the value function is treated as
a viscosity solution. Our objective is to retain the positive
expectation-based update and derive maximum-norm error bounds in this
nonlinear setting.

Boundary treatment is essential for this extension. A wide branch can
reach the Dirichlet boundary before the nominal time step ends, so its
duration and its weight must reflect the contact geometry.
Milstein and Tretyakov \cite{MilsteinTretyakov2001} develop
probabilistic layer methods for nonlinear parabolic Dirichlet problems.
Reisinger and Rotaetxe Arto \cite{ReisingerRotaetxe2017} analyze
boundary truncation for semi-Lagrangian HJB schemes and its effects on
consistency and stability. More recently, Carlini, Picarelli, and Silva
\cite{CarliniPicarelliSilva2026} stop the diffusion-column branches at
first contact, balance the two signs within each column, and balance
the columns. With positive interpolation, they prove consistency and
convergence under their Dirichlet hypotheses. We use this first-contact
and pair-balancing mechanism for a general fixed centrally symmetric
positive degree-two cubature. Its antipodal directions
\(\pm\eta_j\in\mathbb R^p\) and positive pair masses \(\rho_j\)
satisfy \(\sum_j\rho_j=1\) and
\(\sum_j\rho_j\eta_j\eta_j^\top=I_p\), where \(p\) is the
dimension of the driving noise and \(I_p\) is the identity matrix.
The analysis includes unequal pair masses and rank-deficient diffusion.

Quantitative estimates require smooth comparison functions even when
the viscosity solution itself has limited regularity. Krylov's
coefficient-shaking method \cite{Krylov2000} constructs suitable
approximations by perturbing the coefficients before mollification.
Barles and Jakobsen \cite{BarlesJakobsen2007} combine this method with
switching systems to obtain both one-sided error bounds for general
monotone schemes. Debrabant and Jakobsen \cite{DebrabantJakobsen2013}
apply this approach to semi-Lagrangian diffusion schemes in the whole
space. Error estimates on bounded domains are developed by Dong and
Krylov \cite{DongKrylov2007} and, using a common strict boundary barrier
and switching-system approximations, by Picarelli, Reisinger, and
Rotaetxe Arto \cite{PicarelliReisingerRotaetxe2020}.
We use the latter framework to control the error near the boundary and
to construct the interior comparison functions.

For the stopped scheme considered here, the required consistency
estimates must be uniform in the effective time \(\tau\), the weighted
mean of the branch durations. This quantity depends on the control and
can be much smaller than the nominal step \(\Delta t\). We show that
the stopped-barrier remainder is \(O(\tau\Delta t^{\gamma/2})\),
where \(\gamma\in(0,1]\) is the H\"older exponent in the boundary
regularity assumptions. We also prove that the total weight of branches
requiring interpolation is at most \(2\tau/\Delta t\). For a mesh of
diameter \(h\), their interpolation error is therefore
\(O(\tau h^2/\Delta t)\). Both errors retain the factor \(\tau\),
so normalization gives the uniform residual bound
\(C(\Delta t^{\gamma/2}+h^2/\Delta t)\).

These bounds yield discrete barriers in a strip near the boundary.
Farther inside the domain, the branches are unstopped and central
symmetry gives a sharper consistency estimate. We combine that estimate
with scalar and switching comparison functions, and then apply
localized discrete comparison to obtain the two one-sided bounds in
Theorem~\ref{thm:bc-rate}. Their different exponents come from the
\(O(\varepsilon)\) scalar approximation and the
\(O(\varepsilon^{1/3})\) switching approximation, where
\(\varepsilon\) denotes the smoothing scale. In particular,
\(\Delta t\asymp h^{10/7}\) gives an \(O(h^{1/7})\)
maximum-norm bound under the stated coefficient, boundary, and mesh
assumptions. The numerical examples examine heat loss to a cold wall, error
and cost under mesh refinement, and controlled diffusion with and without drift.

Section~\ref{sec:bounded} specifies the equation, boundary conditions,
and mesh. Sections~\ref{sec:scheme} and~\ref{sec:BS} develop the stopped
weights, monotonicity, comparison, and consistency.
Section~\ref{sec:quantitative} proves the error bounds, and
Section~\ref{sec:numerics} presents the numerical experiments.
The appendices supply the Taylor estimates and the extensions used
for the auxiliary equations.

\section{Preliminaries}
	\label{sec:bounded}
	
	\subsection{Notation}
	
	Let \(O\subset\mathbb R^N\), \(N\ge1\), be a bounded domain, and set
	\[
	Q_T=[0,T)\times O,
	\qquad
	Q_T^{\mathrm{cl}}=[0,T]\times\overline O.
	\]
	For space--time points \(z=(t,x)\) and \(z'=(s,y)\), define
	\begin{equation}
		\varrho(z,z')
		:=
		|x-y|+|t-s|^{1/2}.
		\label{eq:parabolic-distance}
	\end{equation}
	This is the distance used throughout the paper.
	
	If \(U\subset\mathbb R\times\mathbb R^N\) and \(w:U\to\mathbb R\), define
	\begin{align*}
		[w]_{x;1;U}
		&:=
		\sup_{\substack{(t,x),(t,y)\in U\\x\ne y}}
		\frac{|w(t,x)-w(t,y)|}{|x-y|},\\
		[w]_{t;1/2;U}
		&:=
		\sup_{\substack{(t,x),(s,x)\in U\\t\ne s}}
		\frac{|w(t,x)-w(s,x)|}{|t-s|^{1/2}},\\
		[w]_{\gamma;U}
		&:=
		\sup_{\substack{z,z'\in U\\z\ne z'}}
		\frac{|w(z)-w(z')|}{\varrho(z,z')^\gamma},
		\qquad
		0<\gamma\le1.
	\end{align*}
	Thus \([w]_{x;1;U}\) is the uniform spatial Lipschitz seminorm,
	whereas \([w]_{t;1/2;U}\) is the uniform \(1/2\)-H\"older seminorm
	in time.
	
	For \(w\) with the indicated derivatives, put
	\[
	\|w\|_{\gamma;U}:=\|w\|_{L^\infty(U)}+[w]_{\gamma;U}
	\]
	and define
	\begin{equation}
		\|w\|_{2,\gamma;U}
		:=
		\sum_{2q+|\alpha|\le2}
		\|\partial_t^qD_x^\alpha w\|_{\gamma;U}.
		\label{eq:parabolic-holder-norm}
	\end{equation}
	We write \(w\in C^{2,\gamma}(U)\) when this norm is finite and all
	derivatives appearing in \eqref{eq:parabolic-holder-norm} are continuous.
	In the condition \(2q+|\alpha|\le2\), one time derivative has weight two
	and one spatial derivative has weight one.  All H\"older spaces and
	seminorms below use the distance \(\varrho\) and the convention in
	\eqref{eq:parabolic-holder-norm}; no additional ``parabolic'' subscript
	will be used.
	
	Throughout the paper, \(|\cdot|\) denotes the absolute value for scalars,
	the Euclidean norm for vectors, and the Frobenius norm for matrices.
	The same formulas define the corresponding seminorms for vector- and
	matrix-valued functions using these norms.
	
	For \(\alpha>0\), let
	\begin{equation}
		O_\alpha
		:=
		\bigl\{
		x\in O:
		\operatorname{dist}(x,\partial O)>\alpha
		\bigr\}.
		\label{eq:contracted-domain}
	\end{equation}
	We write \(\mathbb S^N\) for the space of symmetric \(N\times N\) matrices,
	\(\langle\cdot,\cdot\rangle\) for the Euclidean inner product, and
	\(\tr\) for the trace.
	
	\subsection{Problem setting and assumptions}
	
	Let
	\[
	\sigma(t,x,a)\in\R^{N\times p},\quad
	\mu(t,x,a)\in\R^N,\quad
	r(t,x,a),f(t,x,a)\in\R,
	\]
	where $p\geq1$ is the noise dimension and \((A,d_A)\) is the control
	space.  No relation between $p$ and the state dimension $N$ is needed.
	We use
	\begin{equation}
		\label{eq:covariance-factor}
		A^a(t,x):=\sigma(t,x,a)\sigma(t,x,a)^\top\in\Sph^N
	\end{equation}
	for the diffusion covariance and call $\sigma^a$ a diffusion factor.
	Define the fixed-control residual
	\begin{equation*}
		F^a[\varphi](t,x)
		:=
		\partial_t\varphi+
		\frac12\tr(A^aD^2\varphi)
		+\langle\mu^a,D\varphi\rangle-r^a\varphi+f^a,
	\end{equation*}
	where all coefficients on the right are evaluated at $(t,x,a)$.  The
	bounded-domain problem considered in this paper is written directly as
	\begin{equation}
		\label{eq:bounded-HJB}
		\begin{cases}
			\displaystyle\inf_{a\in A}F^a[u](t,x)=0,
			&(t,x)\in Q_T,\\
			u(t,x)=u_\partial(t,x),
			&(t,x)\in[0,T)\times\partial O,\\
			u(T,x)=\Psi(x),
			&x\in\overline O.
		\end{cases}
	\end{equation}
	For comparison with the standard backward HJB notation, set
	\begin{equation}
		\label{eq:H}
		H(t,x,z,\xi,X)
		:=\sup_{a\in A}\mathcal L^a(t,x,z,\xi,X),
		\qquad
		\mathcal L^a(t,x,z,\xi,X)
		:=-\frac12\tr(A^aX)-\mu^a\!\cdot\xi+r^az-f^a.
	\end{equation}
	Since \(F^a[\varphi]=\partial_t\varphi-\mathcal L^a[\varphi]\),
	the differential equation in \eqref{eq:bounded-HJB} is equivalently
	\[
	-\partial_tu+H(t,x,u,Du,D^2u)=0.
	\]
	
	When $A=\{a\}$ is a singleton, \eqref{eq:bounded-HJB} reduces to
	\begin{equation}
		\label{eq:linear-parabolic}
		\partial_tu+\frac12\tr(\sigma\sigma^\top D^2u)
		+\langle\mu,Du\rangle-ru+f=0
		\quad\text{in }Q_T,
	\end{equation}
	with the same Dirichlet and terminal data.
	
	\begin{assumption}[Bounded-domain data]
		\label{ass:bounded}
		The following conditions hold.
		\begin{enumerate}[label=\textup{(A\arabic*)}]
			\item There is a constant \(K>0\) such that
			\[
			\norm{\Psi}_\infty+\norm{u_\partial}_\infty
			+\sup_{a\in A}\left(
			\norm{\sigma^a}_\infty+\norm{\mu^a}_\infty
			+\norm{r^a}_\infty+\norm{f^a}_\infty
			\right)
			\leq K.
			\]
			The terminal and lateral boundary data satisfy
			\begin{align*}
				|\Psi(x)-\Psi(y)|
				&\leq K|x-y|,
				&&x,y\in\overline O,\\
				|u_\partial(t,x)-u_\partial(s,y)|
				&\leq K\bigl(|x-y|+|t-s|^{1/2}\bigr),
				&&(t,x),(s,y)\in[0,T]\times\partial O.
			\end{align*}
			For each \(g\in\{\sigma,\mu,r,f\}\),
			\[
			|g(t,x,a)-g(s,y,a)|
			\leq K\bigl(|x-y|+|t-s|^{1/2}\bigr)
			\]
			for all
			\[
			(t,x),(s,y)\in[0,T]\times\overline O,
			\qquad a\in A.
			\]
			
			\item \((A,d_A)\) is a nonempty compact metric space.  For every fixed
			\((t,x)\in[0,T]\times\overline O\), the coefficients
			\(\sigma,\mu,r,f\) are continuous in \(a\).
		\end{enumerate}
	\end{assumption}
	
	\begin{remark}[Normalization of the zero-order coefficient]
Choose \(\lambda\ge\|r^-\|_\infty\), and put
\(q(t)=e^{-\lambda(T-t)}\).  The transformation
\[
\widetilde u=qu,\qquad \widetilde r=r+\lambda,\qquad
\widetilde f=qf,\qquad \widetilde u_\partial=qu_\partial
\]
gives \(\widetilde F^a[qu]=qF^a[u]\) for every control, with the same
terminal datum because \(q(T)=1\).  We henceforth apply the scheme to the
transformed equation and drop the tildes, so that
\begin{equation}
r(t,x,a)\ge0\qquad\text{on }[0,T]\times\overline O\times A.
\label{eq:nonnegative-r-normalization}
\end{equation}
The original approximation is recovered as \(q^{-1}\widetilde u_h\);
the bounded factor \(q^{-1}\) leaves all maximum-norm rates unchanged.
The lifting and barrier below transform as \(\widetilde G=qG\) and
\(\widetilde\zeta=q\zeta\).  Their boundary and terminal compatibility
conditions are preserved, and a homogeneous barrier residual at most
\(-\beta\) becomes at most \(-e^{-\lambda T}\beta\).
\end{remark}
	
	\subsection{Strong and weak Dirichlet viscosity solutions}
	
	In the definitions below, local maxima and minima and the corresponding
	\(C^{1,2}\) test functions are understood relative to
	\(Q_T^{\mathrm{cl}}\).  At \(t=0\), the time derivative is one-sided.
	
	\begin{definition}[Strong Dirichlet viscosity solution]
		\label{def:strong-boundary}
		A strong Dirichlet viscosity subsolution is an upper semicontinuous function
		\(v:Q_T^{\mathrm{cl}}\to\mathbb R\) satisfying the following two conditions.  First,
		whenever \((t,x)\in Q_T\) and \(v-\varphi\) has a local maximum at
		\((t,x)\), with \(v(t,x)=\varphi(t,x)\),
		\begin{equation}
			\inf_{a\in A}F^a[\varphi](t,x)\geq0,
			\label{eq:viscosity-interior}
		\end{equation}
		Second,
		\begin{equation}
			v(t,x)\leq u_\partial(t,x)
			\quad\text{on }[0,T)\times\partial O,
			\qquad
			v(T,x)\leq\Psi(x)
			\quad\text{on }\overline O.
			\label{eq:viscosity-strong-boundary}
		\end{equation}
		
		A strong Dirichlet viscosity supersolution is a lower semicontinuous function
		\(v:Q_T^{\mathrm{cl}}\to\mathbb R\) satisfying the following two conditions.  First,
		whenever \((t,x)\in Q_T\) and \(v-\varphi\) has a local minimum at
		\((t,x)\), with \(v(t,x)=\varphi(t,x)\),
		\[
		\inf_{a\in A}F^a[\varphi](t,x)\leq0,
		\]
		Second,
		\[
		v(t,x)\geq u_\partial(t,x)
		\quad\text{on }[0,T)\times\partial O,
		\qquad
		v(T,x)\geq\Psi(x)
		\quad\text{on }\overline O.
		\]
		
		A strong Dirichlet viscosity solution is a continuous function
		\(v:Q_T^{\mathrm{cl}}\to\mathbb R\) that is both a strong subsolution and a
		strong supersolution.
	\end{definition}
	
	\begin{definition}[Weak Dirichlet viscosity solution]
		\label{def:bounded-viscosity}
		A weak Dirichlet viscosity subsolution is an upper semicontinuous function
		\(v:Q_T^{\mathrm{cl}}\to\mathbb R\) satisfying the following three conditions.  First,
		whenever \((t,x)\in Q_T\) and \(v-\varphi\) has a local maximum at
		\((t,x)\), with \(v(t,x)=\varphi(t,x)\),
		\[
		\inf_{a\in A}F^a[\varphi](t,x)\geq0,
		\]
		Second,
		\begin{equation}
			v(T,x)\leq\Psi(x),
			\qquad x\in\overline O,
			\label{eq:viscosity-weak-terminal}
		\end{equation}
		Third, whenever
		\((t,x)\in[0,T)\times\partial O\) and \(v-\varphi\) has a local maximum
		at \((t,x)\), with \(v(t,x)=\varphi(t,x)\),
		\begin{equation}
			\max\left\{
			\inf_{a\in A}F^a[\varphi](t,x),
			u_\partial(t,x)-v(t,x)
			\right\}
			\geq0.
			\label{eq:viscosity-lateral}
		\end{equation}
		
		A weak Dirichlet viscosity supersolution is a lower semicontinuous function
		\(v:Q_T^{\mathrm{cl}}\to\mathbb R\) satisfying the following three conditions.  First,
		whenever \((t,x)\in Q_T\) and \(v-\varphi\) has a local minimum at
		\((t,x)\), with \(v(t,x)=\varphi(t,x)\),
		\[
		\inf_{a\in A}F^a[\varphi](t,x)\leq0,
		\]
		Second,
		\[
		v(T,x)\geq\Psi(x),
		\qquad x\in\overline O,
		\]
		Third, whenever
		\((t,x)\in[0,T)\times\partial O\) and \(v-\varphi\) has a local minimum
		at \((t,x)\), with \(v(t,x)=\varphi(t,x)\),
		\begin{equation}
			\min\left\{
			\inf_{a\in A}F^a[\varphi](t,x),
			u_\partial(t,x)-v(t,x)
			\right\}
			\leq0.
			\label{eq:viscosity-lateral-super}
		\end{equation}
		
		A weak Dirichlet viscosity solution is a continuous function
		\(v:Q_T^{\mathrm{cl}}\to\mathbb R\) that is both a weak subsolution and a
		weak supersolution.
	\end{definition}
	
	\begin{remark}
		The subsolution and supersolution conditions in
		Definitions~\ref{def:strong-boundary} and
		\ref{def:bounded-viscosity} follow the strong and weak Dirichlet notions
		in \cite[Definitions~2.2 and~2.3]{CarliniPicarelliSilva2026}, with the
		zeroth-order term \(r^av\) included.  Reversing the overall sign expresses
		the conditions in terms of \(\inf_{a\in A}F^a\).  Thus upper contacts of subsolutions carry the
		inequality \(\inf_aF^a\ge0\), while lower contacts of supersolutions carry
		\(\inf_aF^a\le0\).  We use the continuous-solution case of those
		definitions.  The two notions differ only on the lateral
		boundary.  The strong definition imposes the lateral boundary data
		pointwise, while the weak definition uses
		\eqref{eq:viscosity-lateral} and
		\eqref{eq:viscosity-lateral-super}.  Both definitions impose the
		terminal data pointwise.  Under the barrier assumptions in
		Section~\ref{sec:quantitative},
		Lemma~\ref{lem:relaxed-is-strong} shows that the two definitions are
		equivalent.  The binary \(\max/\min\) in the weak boundary complementarity
		conditions and in later switching obstacles are not extrema over the control
		set; all control-set extrema are denoted by \(\inf\) or \(\sup\).
	\end{remark}
	
	\subsection{Spatial discretization and interpolation}
	
	Let \(O_h\subset\mathbb R^N\) be a polyhedral approximation of \(O\), and
	let \(\mathcal T_h\) be a conforming simplicial triangulation of \(O_h\)
	with vertex set
	\[
	G_h=\{x_i:1\leq i\leq N_h\},
	\qquad
	h:=\max_{E\in\mathcal T_h}\operatorname{diam}E.
	\]
	Here \(d_H\) denotes Hausdorff distance.  The constants
	\(C_\partial,C_p\), and \(C_{\rm patch}\) below are independent of \(h\).
	Let \(\{\psi_i\}_{i=1}^{N_h}\) be the \(P_1\) nodal basis on
	\(\mathcal T_h\); thus \(\psi_i\ge0\) and \(\sum_i\psi_i=1\).
	We assume
	
	\begin{enumerate}[label=\textup{(G\arabic*)}]
		\item \(G_h\cap\partial O_h=G_h\cap\partial O\), and every remaining
		vertex lies in \(O\).
		
		\item Every boundary element has at most one face in \(\partial O_h\).
		
		\item Every \(\mathcal T_h\) is conforming and consists of nondegenerate
		simplices.  No uniform minimum-angle condition is imposed.
		
		\item Let \(\mathcal X_h\subset O\) be the set of off-grid branch endpoints
		at which the scheme requests interpolation.  The boundary approximation and
		the evaluation map
		\(p_h:\overline O\to\overline O_h\) satisfy
		\[
		d_H(\partial O,\partial O_h)\le C_\partial h^2,
		\qquad p_h(x_i)=x_i,
		\qquad \sup_{x\in\overline O}|p_h(x)-x|\le C_p h^2.
		\]
		For \(x\in\mathcal X_h\), let
		\[
		\mathcal V_h(x):=
		\{\ell:\psi_\ell(p_h(x))>0\}
		\]
		be the indices of the vertices used to interpolate at \(x\).  Every
		\(\ell\in\mathcal V_h(x)\) satisfies
		\[
		|x_\ell-x|\le C_{\rm patch}h,
		\qquad [x,x_\ell]\subset\overline O,
		\qquad [x,x_\ell)\subset O.
		\]
	\end{enumerate}
	
	For \(\Phi:G_h\to\mathbb R\), define
	\begin{equation}
		\Isp[\Phi](x)
		:=
		\sum_{i=1}^{N_h}\psi_i(p_h(x))\Phi(x_i),
		\qquad x\in\overline O.
		\label{eq:Isp}
	\end{equation}
	The interpolation is monotone, preserves constants, and, for every
	\(\varphi\in C^2(\overline O)\), satisfies
	\begin{equation}
		\sup_{x\in\mathcal X_h}
		\left|
		\Isp[\varphi|_{G_h}](x)-\varphi(x)
		\right|
		\le C_Ih^2\left(
		\norm{D\varphi}_{L^\infty(\overline O)}
		+\norm{D^2\varphi}_{L^\infty(\overline O)}
		\right),
		\label{eq:Isp-error}
	\end{equation}
	where \(C_I:=C_p+\tfrac12C_{\rm patch}^2\) is independent of
	\(h\) and \(\varphi\).  Indeed, for fixed \(x\in\mathcal X_h\), put
	\(\lambda_\ell=\psi_\ell(p_h(x))\).  Then
	\(\lambda_\ell\ge0\), \(\sum_\ell\lambda_\ell=1\), and
	\(\sum_\ell\lambda_\ell x_\ell=p_h(x)\).  Taylor expansion from \(x\)
	to each active vertex, along the segment in \textup{(G4)}, gives
	\[
	\left|\sum_\ell\lambda_\ell\varphi(x_\ell)-\varphi(x)\right|
	\le \norm{D\varphi}_\infty|p_h(x)-x|
	+\frac12\norm{D^2\varphi}_\infty
	\sum_\ell\lambda_\ell|x_\ell-x|^2,
	\]
	and \eqref{eq:Isp-error} follows from \textup{(G4)}.
	
	Choose \(N_{\dt}\in\mathbb N\) and set
	\[
	\dt=\frac{T}{N_{\dt}},
	\qquad
	t_n=n\dt,
	\qquad
	n=0,\ldots,N_{\dt}.
	\]
	We denote by \(u_{h,i}^n\) the numerical approximation of \(u(t_n,x_i)\)
	and write \(u_h^n:=(u_{h,i}^n)_{i=1}^{N_h}\).  The full space--time grid is
\begin{equation}
G_{h,\dt}:=\{(t_n,x_i):0\le n\le N_{\dt},\ x_i\in G_h\}.
\label{eq:space-time-grid}
\end{equation}
	\section{The stopped cubature scheme}
	\label{sec:scheme}
	
	At a fixed time--space node and for one trial control, the diffusion is first
	replaced by a finite symmetric random increment.  Each resulting branch is
	then followed until the end of the time step or its first contact with the
	Dirichlet boundary.  Stopping changes the two travel times in an antipodal
	pair, so the original cubature masses must be adjusted to preserve the first
	two moments.  After this discrete random step has been constructed, the
	continuation value, discount, and source term are added to form one
	fixed-control row.  The HJB update is obtained by taking the infimum of that
	complete row over the control set.
	
	\subsection{Cubature of the frozen diffusion}
	\label{sec:symmetric-degree-two-rules}
	
	\begin{definition}[Symmetric positive cubature of degree two]
		\label{def:symmetric-cubature}
		Let $p$ be the noise dimension.  A symmetric positive cubature of degree
		two is a labeled collection
		\begin{equation}
			\label{eq:cubature-data}
			\mathcal C
			=\{(\rho_j/2,\eta_j),(\rho_j/2,-\eta_j):1\leq j\leq M\},
		\end{equation}
		where $\eta_j\in\R^p\setminus\{0\}$, $\rho_j>0$, and
		\begin{equation}
			\label{eq:cubature-moments}
			\sum_{j=1}^M\rho_j=1,
			\qquad
			\sum_{j=1}^M\rho_j\eta_j\eta_j^\top=I_p.
		\end{equation}
		Equivalently, let \(J_{\mathcal C}\) be a random pair index and
\(\xi_{\mathcal C}\in\{-1,1\}\) its random sign, with
\begin{equation}
\mathbb P(J_{\mathcal C}=j,\xi_{\mathcal C}=\pm1)=\frac{\rho_j}{2},
\qquad Z_{\mathcal C}:=\xi_{\mathcal C}\eta_{J_{\mathcal C}}.
\label{eq:cubature-random-vector}
\end{equation}
Then

\begin{equation}
			\label{eq:cubature-random-moments}
			\mathbb E[Z_{\mathcal C}]=0,
			\qquad
			\mathbb E[Z_{\mathcal C}Z_{\mathcal C}^\top]=I_p.
		\end{equation}
	\end{definition}
	
	The cubature $\mathcal C$ is fixed throughout a refinement family.  Thus
	$M$, the masses $\rho_j$, and the nodes $\eta_j$ are independent of $h$,
	$\dt$, and $a$.  Write
	\[
	\mu_i^n(a)=\mu(t_n,x_i,a),
	\qquad
	\sigma_i^n(a)=\sigma(t_n,x_i,a),
	\qquad
	A_i^n(a)=\sigma_i^n(a)\sigma_i^n(a)^\top.
	\]
	For fixed $(n,i,a)$, the frozen diffusion over an elapsed time
	$q\in[0,\dt]$ has the Gaussian representation
	\[
	x_i+q\mu_i^n(a)+\sqrt q\,\sigma_i^n(a)Z,
	\qquad Z\sim N(0,I_p).
	\]
	Replacing $Z$ by $Z_{\mathcal C}$ gives the discrete random position
	\begin{equation}
		\label{eq:frozen-cubature-random-variable}
		\widehat X_{n,i}^{a,\mathcal C}(q)
		=x_i+q\mu_i^n(a)+\sqrt q\,\sigma_i^n(a)Z_{\mathcal C}.
	\end{equation}
	Its mean and centered covariance are
	\begin{equation}
		\label{eq:frozen-cubature-physical-moments}
		\begin{aligned}
			\mathbb E[\widehat X_{n,i}^{a,\mathcal C}(q)-x_i]
			&=q\mu_i^n(a),\\
			\mathbb E\!\left[
			(\widehat X_{n,i}^{a,\mathcal C}(q)-x_i-q\mu_i^n(a))
			(\widehat X_{n,i}^{a,\mathcal C}(q)-x_i-q\mu_i^n(a))^\top
			\right]
			&=qA_i^n(a).
		\end{aligned}
	\end{equation}
	Thus the finite random step reproduces the drift and covariance of the
	frozen diffusion.
	
	The support of \eqref{eq:frozen-cubature-random-variable} consists of the
	$2M$ paths
	\begin{equation}
		\label{eq:branch-path}
		\Gamma_{n,i}^{a,j,\pm}(q)
		=x_i+q\mu_i^n(a)
		\mathbin{\pm}\sqrt q\,\sigma_i^n(a)\eta_j,
		\qquad 0\leq q\leq\dt.
	\end{equation}
	Their full-step endpoints are
	\begin{equation}
		\label{eq:trial}
		X_{n,i}^{*,a,j,\pm}
		=\Gamma_{n,i}^{a,j,\pm}(\dt)
		=x_i+\dt\mu_i^n(a)
		\mathbin{\pm}\sqrt\dt\,\sigma_i^n(a)\eta_j.
	\end{equation}
	Before stopping, the two branches in pair $j$ have masses $\rho_j/2$.
	
	Three examples used below are as follows.
	
	\begin{enumerate}[label=\textup{(\alph*)},leftmargin=*]
		\item The column rule is
		\begin{equation}
			\label{eq:column-cubature}
			M=p,
			\qquad
			\rho_j=\frac1p,
			\qquad
			\eta_j=\sqrt p\,e_j,
			\qquad j=1,\ldots,p.
		\end{equation}
		It has $2p$ branches $\pm\sqrt p\,e_j$, each of mass $1/(2p)$.
		
		\item The full Rademacher rule is uniform on $\{-1,1\}^p$.  Choosing one
		representative from each antipodal pair gives
		\begin{equation}
			\label{eq:rademacher-cubature}
			M=2^{p-1},
			\qquad
			\rho_j=2^{1-p},
			\qquad
			\bigcup_{j=1}^M\{\eta_j,-\eta_j\}=\{-1,1\}^p.
		\end{equation}
		The rule therefore has $2^p$ branches, each of mass $2^{-p}$.
		
		\item For $p=2$, let
		\[
		\bar\eta_j=
		\begin{pmatrix}
			\cos(2\pi(j-1)/3)\\
			\sin(2\pi(j-1)/3)
		\end{pmatrix},
		\qquad j=1,2,3.
		\]
		The equiangular tight-frame rule is
		\begin{equation}
			\label{eq:tight-frame-cubature}
			M=3,
			\qquad
			\rho_j=\frac13,
			\qquad
			\eta_j=\sqrt2\,\bar\eta_j.
		\end{equation}
		Since $\sum_{j=1}^3\bar\eta_j\bar\eta_j^\top=\frac32I_2$, this rule
		satisfies \eqref{eq:cubature-moments} and has six branches of mass $1/6$.
	\end{enumerate}
	
	\begin{remark}[Dependence on the diffusion factor]
		\label{rem:factorization}
		The PDE coefficient is the covariance
		$A^a=\sigma^a(\sigma^a)^\top$, but the discrete branches use the
		individual directions $\sigma^a\eta_j$.  Hence replacing $\sigma^a$ by
		$\sigma^aQ$ for an orthogonal matrix $Q$ leaves the PDE unchanged but may
		change the discrete operator if the nodes $\eta_j$ are held fixed.  A
		simultaneous change
		\[
		\widetilde\sigma^a=\sigma^aQ,
		\qquad
		\widetilde\eta_j=Q^\top\eta_j
		\]
		leaves every physical direction unchanged, since
		$\widetilde\sigma^a\widetilde\eta_j=\sigma^a\eta_j$.
	\end{remark}
	
	\subsection{First-contact stopping and boundary weights}
	
	Fix an interior node $x_i$, a time level $t_n$, and a control $a$.  For each
	branch, let
	\[
	\mathcal H_{n,i}^{a,j,\pm}
	=\{q\in(0,\dt]:\Gamma_{n,i}^{a,j,\pm}(q)\notin O\}.
	\]
	If this set is nonempty, define
	\begin{equation}
		\label{eq:branch-duration}
		d_{n,i}^{a,j,\pm}=\inf\mathcal H_{n,i}^{a,j,\pm},
		\qquad
		\chi_{n,i}^{a,j,\pm}=1.
	\end{equation}
	Otherwise set
	\[
	d_{n,i}^{a,j,\pm}=\dt,
	\qquad
	\chi_{n,i}^{a,j,\pm}=0.
	\]
	The path is continuous and starts in $O$, so every duration is positive.  If
	$\chi_{n,i}^{a,j,\pm}=1$, then
	$\Gamma_{n,i}^{a,j,\pm}(d_{n,i}^{a,j,\pm})\in\partial O$.  Contact at
	$q=\dt$ is classified as stopping; the indicator $\chi$, rather than the
	duration alone, records whether prescribed data must be used.  If such a
	contact occurs at terminal time, the terminal datum $\Psi$ is used; contacts
	strictly before $T$ use the lateral datum $u_\partial$.
	
	Stopping generally gives different durations in the two signs of a pair.
	Suppress $(n,i,a)$ temporarily and put
	\[
	s_j^\pm=\sqrt{d_j^\pm}.
	\]
	We first choose the conditional weights inside pair $j$.  Preservation of
	mass and cancellation of the signed noise increment require
	\[
	\gamma_j^++\gamma_j^-=1,
	\qquad
	\gamma_j^+s_j^+=\gamma_j^-s_j^-.
	\]
	The unique solution is
	\begin{equation}
		\label{eq:weight-definitions}
		\gamma_j^+=\frac{s_j^-}{s_j^++s_j^-},
		\qquad
		\gamma_j^-=\frac{s_j^+}{s_j^++s_j^-},
		\qquad
		\theta_j=s_j^+s_j^-.
	\end{equation}
	For these weights,
	\[
	\gamma_j^+(s_j^+)^2+\gamma_j^-(s_j^-)^2=\theta_j.
	\]
	
	Let $\pi_j$ be the total mass assigned to pair $j$.  To retain the relative
	second-moment contribution $\rho_j\eta_j\eta_j^\top$ of the original
	cubature, require
	\[
	\pi_j\theta_j=\tau\rho_j,
	\qquad j=1,\ldots,M,
	\]
	where $\tau$ is the elapsed time represented by the complete stopped row.
	Together with $\sum_j\pi_j=1$, these equations give
	\begin{equation}
		\label{eq:universal-weights}
		\tau=\left(\sum_{j=1}^M\frac{\rho_j}{\theta_j}\right)^{-1},
		\qquad
		\pi_j=\frac{\tau\rho_j}{\theta_j},
		\qquad
		\omega_j^\pm=\pi_j\gamma_j^\pm.
	\end{equation}
	Thus $\omega_j^\pm$ is the final mass of the two stopped branches in pair
	$j$, while $\tau$ is the effective time of the row.
When displaying the dependence on the time level, node, and control, 
we write the effective time as \(\tau_{n,i}(a)\); for a fixed row, 
we abbreviate it to \(\tau\).
	
	\begin{proposition}[Moments of the stopped rule]
		\label{prop:weights}
		At every interior node, the weights in
		\eqref{eq:weight-definitions}--\eqref{eq:universal-weights} are strictly
		positive and satisfy
		\begin{equation}
			\label{eq:weight-moments}
			\begin{aligned}
				\sum_{j=1}^M(\omega_j^++\omega_j^-)&=1,\\
				\sum_{j=1}^M
				(\omega_j^+s_j^+-\omega_j^-s_j^-)\eta_j&=0,\\
				\sum_{j=1}^M
				\bigl(\omega_j^+(s_j^+)^2+\omega_j^-(s_j^-)^2\bigr)
				\eta_j\eta_j^\top&=\tau I_p,\\
				\sum_{j=1}^M
				\bigl(\omega_j^+(s_j^+)^2+\omega_j^-(s_j^-)^2\bigr)&=\tau.
			\end{aligned}
		\end{equation}
		Moreover,
		\begin{equation}
			\label{eq:adaptive-weight-bounds}
			0<\omega_j^+\leq\frac{s_j^-}{s_j^++s_j^-}<1,
			\qquad
			0<\omega_j^-\leq\frac{s_j^+}{s_j^++s_j^-}<1.
		\end{equation}
		If no branch is stopped, then
		$\tau=\dt$, $\pi_j=\rho_j$, and $\omega_j^\pm=\rho_j/2$.
	\end{proposition}
	
	\begin{proof}
		The definitions give
		\begin{equation}
			\label{eq:pair-identities}
			\gamma_j^++\gamma_j^-=1,
			\qquad
			\gamma_j^+s_j^+-\gamma_j^-s_j^-=0,
			\qquad
			\gamma_j^+(s_j^+)^2+\gamma_j^-(s_j^-)^2=\theta_j.
		\end{equation}
		Since $\omega_j^\pm=\pi_j\gamma_j^\pm$,
		\[
		\sum_j(\omega_j^++\omega_j^-)
		=\sum_j\pi_j
		=\tau\sum_j\frac{\rho_j}{\theta_j}=1.
		\]
		The second identity in \eqref{eq:pair-identities}, multiplied by
		$\pi_j\eta_j$ and summed over $j$, gives the zero first moment.  The third
		identity and $\pi_j\theta_j=\tau\rho_j$ give
		\[
		\sum_j
		\bigl(\omega_j^+(s_j^+)^2+\omega_j^-(s_j^-)^2\bigr)
		\eta_j\eta_j^\top
		=\tau\sum_j\rho_j\eta_j\eta_j^\top
		=\tau I_p.
		\]
		The same calculation without $\eta_j\eta_j^\top$ gives the last identity
		in \eqref{eq:weight-moments}.
		
		For each $j$,
		\[
		\frac1\tau=\sum_{k=1}^M\frac{\rho_k}{\theta_k}
		\geq\frac{\rho_j}{\theta_j},
		\]
		so $\pi_j=\tau\rho_j/\theta_j\leq1$.  Therefore
		$\omega_j^\pm=\pi_j\gamma_j^\pm\leq\gamma_j^\pm$, which proves
		\eqref{eq:adaptive-weight-bounds}.
		
		If no branch is stopped, then
		$s_j^+=s_j^-=\sqrt\dt$, $\theta_j=\dt$, and
		$\gamma_j^\pm=1/2$.  Substitution into
		\eqref{eq:universal-weights} gives
		$\tau=\dt$, $\pi_j=\rho_j$, and $\omega_j^\pm=\rho_j/2$.
	\end{proof}
	
	For the column rule, the preceding construction recovers the $2p$-branch
	weights of Carlini--Picarelli--Silva
	\cite{CarliniPicarelliSilva2026}.
	
	\begin{corollary}[The $2p$ column weights]
		\label{cor:column-rule}
		Let $\mathcal C$ be the column rule \eqref{eq:column-cubature} and write
		\[
		d_j^\pm=\lambda_j^\pm\dt,
		\qquad 0<\lambda_j^\pm\leq1,
		\qquad
		\Lambda=\left(
		\sum_{k=1}^p\frac1{\sqrt{\lambda_k^+\lambda_k^-}}
		\right)^{-1}.
		\]
		Then the $2p$ paths are
		\[
		x_i+q\mu_i^n(a)
		\mathbin{\pm}\sqrt{pq}\,\sigma_i^{n,j}(a),
		\qquad j=1,\ldots,p,
		\]
		and their stopped weights are
		\begin{equation}
			\label{eq:column-stopped-weights}
			\begin{aligned}
				\tau&=p\dt\Lambda,
				&\pi_j&=\frac{\Lambda}{\sqrt{\lambda_j^+\lambda_j^-}},\\
				\omega_j^+
				&=\frac{\Lambda}
				{\sqrt{\lambda_j^+}
					(\sqrt{\lambda_j^+}+\sqrt{\lambda_j^-})},
				&
				\omega_j^-
				&=\frac{\Lambda}
				{\sqrt{\lambda_j^-}
					(\sqrt{\lambda_j^+}+\sqrt{\lambda_j^-})}.
			\end{aligned}
		\end{equation}
	\end{corollary}
	
	\begin{proof}
		For the column rule, $\rho_j=1/p$ and
		$\eta_j=\sqrt p\,e_j$.  Hence
		$s_j^\pm=\sqrt{\lambda_j^\pm\dt}$ and
		$\theta_j=\dt\sqrt{\lambda_j^+\lambda_j^-}$.  Substitution into
		\eqref{eq:universal-weights} yields $\tau$ and $\pi_j$ in
		\eqref{eq:column-stopped-weights}; multiplying $\pi_j$ by
		$\gamma_j^\pm$ yields the two branch weights.
	\end{proof}
	
	\begin{corollary}[Full Rademacher weights]
		\label{cor:rademacher-rule}
		For \eqref{eq:rademacher-cubature}, let
		\[
		\Lambda_R
		=\left(\sum_{j=1}^{2^{p-1}}\frac1{s_j^+s_j^-}\right)^{-1}.
		\]
		Then
		\[
		\tau=2^{p-1}\Lambda_R,
		\qquad
		\omega_j^+
		=\frac{\Lambda_R}{s_j^+(s_j^++s_j^-)},
		\qquad
		\omega_j^-
		=\frac{\Lambda_R}{s_j^-(s_j^++s_j^-)}.
		\]
	\end{corollary}
	
	\begin{proof}
		Insert $\rho_j=2^{1-p}$ into \eqref{eq:universal-weights}.
	\end{proof}
	
	\begin{remark}[Zero diffusion directions]
		If $\sigma_i^n(a)\eta_j=0$, the two branches in pair $j$ coincide with
		the frozen drift path.  Their durations and weights remain well defined,
		and their contribution to the physical covariance is zero.
	\end{remark}
	
	\subsection{Bounds for the effective time}
	
	The effective time is a weighted harmonic mean of the pairwise quantities
	$\theta_j$.  The next estimate shows that it cannot collapse at a fixed
	interior node as the control varies.
	
	\begin{lemma}[Control-uniform effective time]
		\label{lem:effective-time-lower-bound}
		Assume $0<\dt\leq1$ and let $x_i\in O$.  Set
		\[
		d_i=\operatorname{dist}(x_i,\partial O),
		\qquad
		B_{\mathcal C}=1+
		\sup_{\substack{(t,x,a)\in[0,T]\times\overline O\times A\\
				1\leq j\leq M}}
		\bigl(|\mu(t,x,a)|+|\sigma(t,x,a)\eta_j|\bigr).
		\]
		Then, uniformly in $a\in A$,
		\begin{equation}
			\label{eq:effective-time-lower-bound}
			0<m_{n,i}:=\min\left\{\dt,\frac{d_i^2}{B_{\mathcal C}^2}\right\}
			\leq\tau_{n,i}(a)\leq\dt.
		\end{equation}
	\end{lemma}
	
	\begin{proof}
		Let $s_j^\pm=\sqrt{d_{n,i}^{a,j,\pm}}$.  If a branch is not stopped,
		then $s_j^\pm=\sqrt\dt$.  If it is stopped, its endpoint lies on
		$\partial O$, and hence
		\begin{align*}
			d_i
			&\leq
			\left|(s_j^\pm)^2\mu_i^n(a)
			\mathbin{\pm}s_j^\pm\sigma_i^n(a)\eta_j\right|\\
			&\leq (s_j^\pm)^2|\mu_i^n(a)|
			+s_j^\pm|\sigma_i^n(a)\eta_j|
			\leq B_{\mathcal C}s_j^\pm,
		\end{align*}
		where $s_j^\pm\leq\sqrt\dt\leq1$.  Thus every branch satisfies
		\[
		\min\{\sqrt\dt,d_i/B_{\mathcal C}\}
		\leq s_j^\pm\leq\sqrt\dt.
		\]
		Consequently
		\[
		m_{n,i}\leq\theta_j=s_j^+s_j^-\leq\dt
		\qquad\text{for every }j.
		\]
		Since $\sum_j\rho_j=1$,
		\[
		\frac1\dt
		\leq\sum_{j=1}^M\frac{\rho_j}{\theta_j}
		=\frac1{\tau_{n,i}(a)}
		\leq\frac1{m_{n,i}}.
		\]
		Taking reciprocals proves \eqref{eq:effective-time-lower-bound}.
	\end{proof}
	
	\subsection{From the stopped paths to the HJB update}
	
	The preceding subsections determine the branch locations, durations, and
	weights.  We now attach a continuation value to each branch and assemble the
	discrete operator.
	
	Fix $a\in A$ and consider the auxiliary linear terminal--boundary problem
	\begin{equation}
		\label{eq:fixed-control-linear-problem}
		\begin{cases}
			F^a[v^a]=\partial_t v^a
			+\dfrac12\operatorname{Tr}\!\left(A^aD^2v^a\right)
			+\langle\mu^a,Dv^a\rangle
			-r^av^a+f^a=0,
			&(t,x)\in(0,T)\times O,\\
			v^a(t,x)=u_\partial(t,x),
			&(t,x)\in[0,T)\times\partial O,\\
			v^a(T,x)=\Psi(x),
			&x\in\overline O.
		\end{cases}
	\end{equation}
	We use this linear problem to derive the one-step update for a fixed control.
	
	Let $Y^a$ solve
	\begin{equation}
		\label{eq:SDE}
		dY_s^a
		=\mu(s,Y_s^a,a)\,ds+\sigma(s,Y_s^a,a)\,dW_s,
		\qquad
		Y_{t_n}^a=x_i,
	\end{equation}
	and define
	\[
	\vartheta_\partial^a
	=\inf\{s\geq t_n:Y_s^a\notin O\},
	\qquad
	\vartheta^a=t_{n+1}\wedge\vartheta_\partial^a,
	\qquad
	R_s^a=\int_{t_n}^s r(q,Y_q^a,a)\,dq.
	\]
	For a classical solution of
	\eqref{eq:fixed-control-linear-problem}, the stopped Feynman--Kac formula is
	\begin{equation}
		\label{eq:Feynman-Kac}
		v^a(t_n,x_i)
		=\mathbb E\!\left[
		e^{-R_{\vartheta^a}^a}\mathcal V_{\vartheta^a}^a
		+\int_{t_n}^{\vartheta^a}
		e^{-R_s^a}f(s,Y_s^a,a)\,ds
		\right],
	\end{equation}
	where
	\[
	\mathcal V_{\vartheta^a}^a
	=\begin{cases}
		\Psi(Y_T^a),
		&\vartheta^a=T,\\
		u_\partial(
		\vartheta_\partial^a,
		Y_{\vartheta_\partial^a}^a),
		&\vartheta_\partial^a\leq t_{n+1},\quad \vartheta_\partial^a<T,\\
		v^a(t_{n+1},Y_{t_{n+1}}^a),
		&t_{n+1}<\min\{T,\vartheta_\partial^a\}.
	\end{cases}
	\]
	Thus a branch that survives a nonterminal time step reads the solution
	at the next time level, whereas a stopped branch reads the prescribed
	Dirichlet value at its first contact with $\partial O$.  A nonterminal tie
	\(\vartheta_\partial^a=t_{n+1}<T\) is included in the second line and reads
	\(u_\partial\); a terminal-time tie is resolved by the first line and reads
	\(\Psi\).
	
	To obtain the discrete fixed-control row, freeze the coefficients at
	$(t_n,x_i)$, replace the diffusion paths by
	\eqref{eq:branch-path}, and replace the expectation by the stopped weights
	\eqref{eq:universal-weights}.  Let $\Phi$ denote a grid function on the next
	time layer.  The continuation value on branch $(j,\pm)$ is
	\begin{equation}
		\label{eq:branch-continuation}
		C_{n,i}^{a,j,\pm}[\Phi]
		=\begin{cases}
			\Isp[\Phi](X_{n,i}^{*,a,j,\pm}),
			&\chi_{n,i}^{a,j,\pm}=0,\\
			\Psi\!\left(
			\Gamma_{n,i}^{a,j,\pm}
			(d_{n,i}^{a,j,\pm})
			\right),
			&\chi_{n,i}^{a,j,\pm}=1,\quad
			t_n+d_{n,i}^{a,j,\pm}=T,\\
			u_\partial\!\left(
			t_n+d_{n,i}^{a,j,\pm},
			\Gamma_{n,i}^{a,j,\pm}
			(d_{n,i}^{a,j,\pm})
			\right),
			&\chi_{n,i}^{a,j,\pm}=1,\quad
			t_n+d_{n,i}^{a,j,\pm}<T.
		\end{cases}
	\end{equation}
	Under Assumption~\ref{ass:quantitative-boundary}, \textup{(Q1)}--\textup{(Q3)}
	imply $u_\partial(T,\cdot)=\Psi$ on $\partial O$, so the terminal convention
	agrees with the lateral trace in the rate theorem.
	For a branch of duration $d$, we use
	\[
	e^{-r_i^n(a)d}
	=\frac{1}{1+r_i^n(a)d}+O(d^2),
	\qquad
	\int_0^d e^{-r_i^n(a)s}f_i^n(a)\,ds
	=f_i^n(a)d+O(d^2).
	\]
	The corresponding branch value is therefore
	\begin{equation}
		\label{eq:branch-operator}
		B_{n,i}^{a,j,\pm}[\Phi]
		=
		\frac{C_{n,i}^{a,j,\pm}[\Phi]}
		{1+r_i^n(a)d_{n,i}^{a,j,\pm}}
		+f_i^n(a)d_{n,i}^{a,j,\pm}.
	\end{equation}
	The two signs in pair $j$ are combined according to
	\begin{equation}
		\label{eq:pair-operator}
		P_{n,i}^{a,j}[\Phi]
		=
		\gamma_{n,i}^{a,j,+}
		B_{n,i}^{a,j,+}[\Phi]
		+
		\gamma_{n,i}^{a,j,-}
		B_{n,i}^{a,j,-}[\Phi].
	\end{equation}
	Averaging the pairs gives the fixed-control operator
	\begin{equation}
		\label{eq:Sa-boundary}
		\begin{aligned}
			S_{n,i}^a[\Phi]
			&=\sum_{j=1}^M
			\pi_{n,i}^{a,j}P_{n,i}^{a,j}[\Phi]\\
			&=\sum_{j=1}^M\sum_{\pm}
			\omega_{n,i}^{a,j,\pm}
			B_{n,i}^{a,j,\pm}[\Phi].
		\end{aligned}
	\end{equation}
	When the control set consists of a single element $a$, the interior update is
	\[
	v_{h,i}^{a,n}
	=S_{n,i}^a[v_h^{a,n+1}],
	\]
	which is the stopped scheme associated with
	\eqref{eq:fixed-control-linear-problem}.
	
	We now return to the original HJB terminal--boundary problem in the
	direct fixed-control notation of Section~\ref{sec:bounded}:
	\begin{equation}
		\label{eq:HJB-problem-scheme}
		\begin{cases}
			\displaystyle\inf_{a\in A}F^a[u](t,x)=0,
			&(t,x)\in(0,T)\times O,\\
			u(t,x)=u_\partial(t,x),
			&(t,x)\in[0,T)\times\partial O,\\
			u(T,x)=\Psi(x),
			&x\in\overline O.
		\end{cases}
	\end{equation}
	For every $a\in A$, the operator $S_{n,i}^a$ is the dynamic-programming row
	for the corresponding fixed-control equation.  The control is optimized only after
	the complete cubature average has been formed:
	\begin{equation}
		\label{eq:HJB-row-operator}
		S_{n,i}[\Phi]
		=\inf_{a\in A}S_{n,i}^a[\Phi].
	\end{equation}
	The fully discrete HJB scheme is
	\begin{equation}
		\label{eq:bounded-scheme}
		\begin{cases}
			u_{h,i}^{N_{\dt}}=\Psi(x_i),
			&x_i\in G_h,\\
			u_{h,i}^n=u_\partial(t_n,x_i),
			&x_i\in G_h\cap\partial O,\quad n<N_{\dt},\\
			u_{h,i}^n
			=\displaystyle\inf_{a\in A}
			S_{n,i}^a[u_h^{n+1}],
			&x_i\in G_h\cap O,\quad n<N_{\dt}.
		\end{cases}
	\end{equation}
	The infimum in \eqref{eq:bounded-scheme} is outside the cubature sum:
	one control is selected for the entire row.  By
	\eqref{eq:weight-moments}, the corresponding fixed-control average reproduces
	the drift and diffusion contributions of the generator over the effective time
	$\tau_{n,i}(a)$.

	\section{Monotonicity, stability, and consistency}
	\label{sec:BS}
	
	This section records the properties of the scheme that will be used in the
	error analysis.  Monotonicity and stability follow directly from the
	positivity of the branch weights and the interpolant.  The main point is the
	consistency estimate, for which the dependence of the remainder on the
	derivatives of the test function is stated explicitly.
	
	\subsection{Monotonicity and stability}
	
	\begin{proposition}
		\label{prop:bounded-monotone}
		For every interior node, $S_{n,i}$ is monotone and nonexpansive:
		\[
		\Phi\leq\widetilde\Phi
		\quad\Longrightarrow\quad
		S_{n,i}[\Phi]\leq S_{n,i}[\widetilde\Phi],
		\]
		and
		\[
		\abs{S_{n,i}[\Phi]-S_{n,i}[\widetilde\Phi]}
		\leq
		\norm{\Phi-\widetilde\Phi}_{\ell^\infty(G_h)}.
		\]
		Consequently, \eqref{eq:bounded-scheme} satisfies a discrete comparison
		principle: if grid functions \(U,Z\) satisfy
		\[
		U_i^n\le S_{n,i}[U^{n+1}],\qquad
		Z_i^n\ge S_{n,i}[Z^{n+1}]
		\]
		at every interior node, and \(U\le Z\) on the terminal and lateral
		nodes, then \(U\le Z\) on the whole grid.
	\end{proposition}
	
	\begin{proof}
By the positivity and constant-preserving property of
\eqref{eq:Isp}, the weight identities
\eqref{eq:weight-moments}, and \(r\ge0\) from
\eqref{eq:nonnegative-r-normalization}, each fixed-control
row \eqref{eq:Sa-boundary} is affine in the next-layer
nodal values, with nonnegative coefficients whose sum is
at most one. Hence it is monotone and nonexpansive.
Taking the infimum over \(a\in A\) preserves both properties,
which proves the first two assertions.

For comparison, if \(U^{n+1}\le Z^{n+1}\), monotonicity gives
\[
U_i^n
\le S_{n,i}[U^{n+1}]
\le S_{n,i}[Z^{n+1}]
\le Z_i^n
\]
at every interior node. Together with the prescribed ordering
on the lateral boundary, backward induction from the terminal
layer yields \(U\le Z\) on the whole grid.
\end{proof}
	
	\begin{proposition}[Uniform stability]
		\label{prop:bounded-stability}
		Let
		\[
		\norm{f}_\infty
		=
		\sup_{(t,x,a)\in[0,T]\times\overline O\times A}
		|f(t,x,a)|.
		\]
		Then
		\[
		\max_{0\leq n\leq N_{\dt}}
		\norm{u_h^n}_{\ell^\infty(G_h)}
		\leq
		\max\{\norm{\Psi}_\infty,\norm{u_\partial}_\infty\}
		+T\norm{f}_\infty.
		\]
		If $\Psi$, $u_\partial$, and $f$ are nonnegative, then $u_h$ is
		nonnegative.
	\end{proposition}
	
	\begin{proof}
		From \eqref{eq:Sa-boundary} and
		$\sum_{j,\pm}\omega_{n,i}^{a,j,\pm}d_{n,i}^{a,j,\pm}
		=\tau_{n,i}(a)\leq\dt$,
		\[
		\norm{u_h^n}_{\ell^\infty(G_h)}
		\leq
		\max\!\left\{
		\norm{u_h^{n+1}}_{\ell^\infty(G_h)},
		\norm{u_\partial}_\infty
		\right\}
		+\dt\norm{f}_\infty.
		\]
		Iteration from $u_h^{N_{\dt}}=\Psi$ proves the bound.  Positivity follows
		directly from the nonnegative coefficients of the scheme.
	\end{proof}
	
	\subsection{Consistency}
	
	For each fixed control, we estimate every remainder with a factor
	$\tau_{n,i}(a)$.  Dividing by this control-dependent effective time
	then gives a uniform consistency bound, even near the boundary.
	
	Recall the fixed-control expression $F^a$ defined in
	Section~\ref{sec:bounded}.  The HJB
	equation is \(\inf_{a\in A}F^a[u]=0\).
	
	For the consistency calculation, the test function is used on every branch,
	including a branch that meets the boundary.  Set
	\begin{equation}
		\label{eq:test-branch-value}
		\mathcal V_{n,i}^{a,j,\pm}[\varphi]
		=
		\begin{cases}
			\Isp[\varphi(t_{n+1},\cdot)|_{G_h}]
			(X_{n,i}^{*,a,j,\pm}),
			&\chi_{n,i}^{a,j,\pm}=0,\\[1mm]
			\varphi\!\left(
			t_n+d_{n,i}^{a,j,\pm},
			\Gamma_{n,i}^{a,j,\pm}
			(d_{n,i}^{a,j,\pm})
			\right),
			&\chi_{n,i}^{a,j,\pm}=1.
		\end{cases}
	\end{equation}
	The fixed-control test operator is
	\begin{equation}
		\label{eq:test-operator}
		\widehat S_{n,i}^a[\varphi]
		=
		\sum_{j=1}^M\sum_{\pm}
		\omega_{n,i}^{a,j,\pm}
		\left[
		\frac{\mathcal V_{n,i}^{a,j,\pm}[\varphi]}
		{1+r_i^n(a)d_{n,i}^{a,j,\pm}}
		+
		f_i^n(a)d_{n,i}^{a,j,\pm}
		\right].
	\end{equation}
	On a stopped branch, \(\widehat S_{n,i}^a[\varphi]\) uses the test-function
	value at first contact, while \(S_{n,i}^a\) uses the prescribed boundary
	datum.  The following test-row expansion therefore coincides with the
	scheme-row expansion whenever no branch stops or the test function agrees
	with the boundary datum at every stopped contact.
	
	\begin{proposition}[Test-row consistency]
		\label{prop:stopped-consistency}
		Let $\varphi$ be smooth on $Q_T^{\mathrm{cl}}$ and $0<\dt\leq1$.
		Then, uniformly in $a\in A$ and in the interior nodes,
		\begin{equation}
			\label{eq:stopped-consistency}
			\widehat S_{n,i}^a[\varphi]-\varphi(t_n,x_i)
			=
			\tau_{n,i}(a)F^a[\varphi](t_n,x_i)
			+
			R_{n,i}^a[\varphi],
		\end{equation}
		where
		\begin{align}
			\abs{R_{n,i}^a[\varphi]}
			\leq C\tau_{n,i}(a)\Bigg[
			&\sqrt{\dt}\left(
			\norm{D\varphi}_\infty
			+\norm{D^2\varphi}_\infty
			+\norm{D^3\varphi}_\infty
			+\norm{D\partial_t\varphi}_\infty
			\right)
			\notag\\
			&+\dt\left(
			\norm{\varphi}_\infty
			+\norm{\partial_t\varphi}_\infty
			+\norm{D^2\varphi}_\infty
			+\norm{\partial_{tt}\varphi}_\infty
			\right)
			+\frac{h^2}{\dt}
			\left(\norm{D\varphi}_\infty+\norm{D^2\varphi}_\infty\right)
			\Bigg].
			\label{eq:stopped-remainder}
		\end{align}
		The constant $C$ is independent of
		$a,n,i,h,\dt$, and $\varphi$.
		
		Consequently,
		\begin{equation}
			\label{eq:HJB-consistency}
			\inf_{a\in A}
			\frac{
				\widehat S_{n,i}^a[\varphi]-\varphi(t_n,x_i)}
			{\tau_{n,i}(a)}
			=
			\inf_{a\in A}F^a[\varphi](t_n,x_i)
			+
			O_\varphi\!\left(
			\sqrt{\dt}+\frac{h^2}{\dt}
			\right).
		\end{equation}
	\end{proposition}
	
	\begin{proof}
Fix \((n,i,a)\) and suppress these indices. Write
\[
\mu=\mu_i^n(a),\qquad \sigma=\sigma_i^n(a),\qquad
r=r_i^n(a),\qquad f=f_i^n(a),\qquad
\varphi_i^n=\varphi(t_n,x_i).
\]
Derivatives carrying the indices \(i,n\) are evaluated at \((t_n,x_i)\).

\smallskip
\noindent\emph{One-branch expansion.}
On branch \((j,\pm)\), set
\[
s=s_j^\pm=\sqrt{d_j^\pm},\qquad
z_j^\pm=\pm\sigma\eta_j,\qquad
y=sz_j^\pm+s^2\mu.
\]
The endpoint is \((t_n+s^2,x_i+y)\), and \(|y|\le Cs\)
because the coefficients and cubature points are bounded and \(s\le1\).
The curve
\(r\mapsto x_i+r z_j^\pm+r^2\mu\), \(0\le r\le s\), is the
stopped branch itself and hence remains in \(\overline O\).
Applying the branchwise Taylor formula of
Lemma~\ref{lem:bc-app-branch-remainder} with \(\gamma=1\), and absorbing
its explicit terms of orders \(s^3\) and \(s^4\) into the remainder, gives
\begin{align}
\varphi(t_n+s^2,x_i+y)
={}&\varphi_i^n+s\langle D\varphi_i^n,z_j^\pm\rangle
 +s^2\bigl(\partial_t\varphi_i^n
       +\langle\mu,D\varphi_i^n\rangle\bigr)\notag\\
&+\tfrac12s^2\langle D^2\varphi_i^n z_j^\pm,z_j^\pm\rangle
 +\mathcal E_j^\pm,
\label{eq:branch-taylor}
\end{align}
with
\begin{align}
|\mathcal E_j^\pm|\le C\Bigl(
&s^3\bigl(\|D^2\varphi\|_\infty+\|D^3\varphi\|_\infty
               +\|D\partial_t\varphi\|_\infty\bigr)\notag\\
&+s^4\bigl(\|D^2\varphi\|_\infty
               +\|\partial_{tt}\varphi\|_\infty\bigr)\Bigr).
\label{eq:branch-taylor-remainder}
\end{align}

For the discount, put \(q=(1+rs^2)^{-1}\). The identity
\[
q=1-rs^2+\frac{r^2s^4}{1+rs^2}
\]
and the mean-value estimate
\[
|\varphi(t_n+s^2,x_i+y)-\varphi_i^n|
\le C\bigl(s\|D\varphi\|_\infty
           +s^2\|\partial_t\varphi\|_\infty\bigr)
\]
give
\begin{align*}
&|q\varphi(t_n+s^2,x_i+y)-\varphi(t_n+s^2,x_i+y)
       +rs^2\varphi_i^n|\\
&\qquad\le C\Bigl(s^3\|D\varphi\|_\infty
       +s^4\bigl(\|\varphi\|_\infty
                   +\|\partial_t\varphi\|_\infty\bigr)\Bigr).
\end{align*}
Combining this bound with
\eqref{eq:branch-taylor}--\eqref{eq:branch-taylor-remainder}
and adding \(fs^2\), we find
\begin{align}
q\varphi(t_n+s^2,x_i+y)+fs^2
={}&\varphi_i^n+s\langle D\varphi_i^n,z_j^\pm\rangle\notag\\
&+s^2\bigl(\partial_t\varphi_i^n
 +\langle\mu,D\varphi_i^n\rangle-r\varphi_i^n+f\bigr)\notag\\
&+\tfrac12s^2\langle D^2\varphi_i^n z_j^\pm,z_j^\pm\rangle
 +\mathcal B_j^\pm,
\label{eq:discounted-branch-expansion}
\end{align}
where
\begin{align}
|\mathcal B_j^\pm|\le C\Bigl(
&s^3\bigl(\|D\varphi\|_\infty+\|D^2\varphi\|_\infty
 +\|D^3\varphi\|_\infty+\|D\partial_t\varphi\|_\infty\bigr)
 \notag\\
&+s^4\bigl(\|\varphi\|_\infty+\|\partial_t\varphi\|_\infty
 +\|D^2\varphi\|_\infty+\|\partial_{tt}\varphi\|_\infty\bigr)
\Bigr).
\label{eq:discounted-branch-remainder}
\end{align}

\smallskip
\noindent\emph{Weighted average.}
Multiply \eqref{eq:discounted-branch-expansion} by \(\omega_j^\pm\)
and sum over all branches. In the present notation,
\eqref{eq:weight-moments} gives
\[
\sum_{j,\pm}\omega_j^\pm=1,\qquad
\sum_{j,\pm}\omega_j^\pm(s_j^\pm)^2=\tau,
\]
\[
\sum_{j,\pm}\omega_j^\pm s_j^\pm z_j^\pm=0,\qquad
\sum_{j,\pm}\omega_j^\pm(s_j^\pm)^2
 z_j^\pm(z_j^\pm)^\top=\tau\sigma\sigma^\top.
\]
Thus the constant term gives \(\varphi_i^n\), the linear noise term
cancels, and the quadratic noise term satisfies
\[
\frac12\sum_{j,\pm}\omega_j^\pm(s_j^\pm)^2
 \langle D^2\varphi_i^n z_j^\pm,z_j^\pm\rangle
=\frac\tau2\operatorname{Tr}(\sigma\sigma^\top D^2\varphi_i^n).
\]
Together with the time, drift, discount, and source terms, this gives
the weighted principal part
\(\varphi_i^n+\tau F^a[\varphi](t_n,x_i)\).

For the remainders, \(s_j^\pm\le\sqrt\dt\) implies
\[
\sum_{j,\pm}\omega_j^\pm(s_j^\pm)^3
\le\sqrt\dt\sum_{j,\pm}\omega_j^\pm(s_j^\pm)^2
=\tau\sqrt\dt,\qquad
\sum_{j,\pm}\omega_j^\pm(s_j^\pm)^4\le\tau\dt.
\]
Consequently, \eqref{eq:discounted-branch-remainder} yields
\begin{align}
\left|\sum_{j,\pm}\omega_j^\pm\mathcal B_j^\pm\right|
\le C\tau\Bigl(
&\sqrt\dt\bigl(\|D\varphi\|_\infty+\|D^2\varphi\|_\infty
 +\|D^3\varphi\|_\infty+\|D\partial_t\varphi\|_\infty\bigr)
 \notag\\
&+\dt\bigl(\|\varphi\|_\infty+\|\partial_t\varphi\|_\infty
 +\|D^2\varphi\|_\infty+\|\partial_{tt}\varphi\|_\infty\bigr)
\Bigr).
\label{eq:summed-branch-remainder}
\end{align}

\smallskip
\noindent\emph{Interpolation.}
In \eqref{eq:test-branch-value}, stopped branches use the test function
at the contact point. Only unstopped branches require interpolation,
and on these branches \(s_j^\pm=\sqrt\dt\).
Equation~\eqref{eq:universal-weights} therefore gives
\begin{equation}
\omega_j^\pm
=\frac{\tau\rho_j}{s_j^\pm(s_j^++s_j^-)}
\le\frac{\tau\rho_j}{\dt},\qquad
\sum_{\chi_j^\pm=0}\omega_j^\pm\le\frac{2\tau}{\dt}.
\label{eq:interpolated-branch-mass}
\end{equation}
The second bound follows because each pair contains at most two
unstopped branches and \(\sum_j\rho_j=1\).
Write \(X_j^\pm=X_{n,i}^{*,a,j,\pm}\) and
\(q_j^\pm=(1+r(s_j^\pm)^2)^{-1}\). The interpolation contribution is
\[
E_{\rm int}
=\sum_{\substack{j,\pm\\\chi_j^\pm=0}}
\omega_j^\pm q_j^\pm
\Bigl(
\Isp[\varphi(t_{n+1},\cdot)|_{G_h}](X_j^\pm)
-\varphi(t_{n+1},X_j^\pm)
\Bigr).
\]
Using \(q_j^\pm\le1\), \eqref{eq:interpolated-branch-mass}, and
\eqref{eq:Isp-error}, including the displacement through \(p_h\),
we obtain
\[
|E_{\rm int}|\le C\tau\frac{h^2}{\dt}
\bigl(\|D\varphi\|_\infty+\|D^2\varphi\|_\infty\bigr).
\]
Thus the total remainder is
\[
R_{n,i}^a[\varphi]
=\sum_{j,\pm}\omega_j^\pm\mathcal B_j^\pm+E_{\rm int}.
\]
Combining the last bound with \eqref{eq:summed-branch-remainder}
proves \eqref{eq:stopped-consistency}--\eqref{eq:stopped-remainder}.
Finally, divide by \(\tau=\tau_{n,i}(a)\). Since the remainder bound
is uniform in \(a\) and \(\dt\le\sqrt\dt\), taking the control
infimum gives \eqref{eq:HJB-consistency}.
\end{proof}

	\section{Quantitative error bounds}
\label{sec:quantitative}

The error estimate combines three bounds: interior consistency for smooth
comparison functions, discrete barriers near the lateral boundary, and
Lemma~\ref{lem:discrete-terminal-modulus} near \(T\).
Localized comparison then gives the two one-sided errors, using scalar
and switching regularizations, respectively.
The auxiliary PDE estimates come from \cite{PicarelliReisingerRotaetxe2020}:
under \(s=T-t\), \(V(s,x)=u(T-s,x)\) solves
\(\partial_sV+H(T-s,x,V,DV,D^2V)=0\), with initial datum \(\Psi\)
and coefficient tuple \((A^a/2,\mu^a,-r^a,f^a)(T-s,x)\).

Shaking requires coefficient values outside the physical cylinder.
For a scalar component \(g_\ell\) of
\(g\in\{\sigma,\mu,r,f\}\), use McShane's extension
\cite{McShane1934} with the parabolic distance \(\varrho\) from
\eqref{eq:parabolic-distance}:
\begin{equation}
(Eg_\ell)(z,a)=\inf_{z'\in[0,T]\times\overline O}
 \{g_\ell(z',a)+K\varrho(z,z')\},
\qquad z\in\mathbb R\times\mathbb R^N.
\label{eq:quant-McShane-extension}
\end{equation}
The triangle inequality and Assumption~\ref{ass:bounded} give
\begin{equation}
Eg_\ell=g_\ell\ \text{on }[0,T]\times\overline O,
\qquad |Eg_\ell(z,a)-Eg_\ell(w,a)|\le K\varrho(z,w).
\label{eq:McShane-basic-properties}
\end{equation}
For the discount, replace its scalar extension \(\widetilde{Er}\) by
\begin{equation}
Er=\max\{0,\widetilde{Er}\}.
\label{eq:nonnegative-r-extension}
\end{equation}
The positive-part map preserves the Lipschitz bound and agrees with
\(r\ge0\) on the physical cylinder.

The control modulus is
\begin{equation}
\omega_A(\delta)=\max_{g\in\{\sigma,\mu,r,f\}}
 \sup_{\substack{(t,x)\in[0,T]\times\overline O,\ a,b\in A\\
 d_A(a,b)\le\delta}}|g(t,x,a)-g(t,x,b)|.
\label{eq:quant-control-modulus}
\end{equation}
Uniform parabolic Lipschitz continuity, control continuity, and
compactness imply \(\omega_A(\delta)\to0\).  Since the distance term
in \eqref{eq:quant-McShane-extension} is independent of the control,
the extensions satisfy
\begin{equation}
|g(z,a)-g(z,b)|\le C\omega_A(d_A(a,b)).
\label{eq:extended-control-modulus}
\end{equation}
Here and below the extended coefficients retain their original symbols,
and \(A^a=\sigma^a(\sigma^a)^\top\).  All evaluations occur in a fixed
bounded collar, on which the extensions have uniform bounds.

\begin{assumption}[Boundary assumptions for the error estimate]
		\label{ass:quantitative-boundary}
		In addition to Assumption~\ref{ass:bounded}, suppose that the following
		conditions hold.
		\begin{enumerate}[leftmargin=*]
			\item[\textup{(Q1)}]
			There are \(0<\gamma\le1\) and
			\[
			G,\zeta
			\in
			C^{2,\gamma}([0,T]\times\overline O)
			\]
			such that
			\[
			G=u_\partial
			\quad\hbox{on }[0,T]\times\partial O,
			\qquad
			\zeta=0
			\quad\hbox{on }[0,T]\times\partial O.
			\]
			Moreover,
			\[
			\zeta(t,x)>0,
			\qquad
			(t,x)\in[0,T)\times O,
			\]
			and
			\[
			\zeta(T,x)\ge0,
			\qquad
			x\in\overline O.
			\]
			
			\item[\textup{(Q2)}]
			There is \(\beta>0\) such that, for every \(a\in A\),
			\begin{equation}
				\partial_t\zeta
				+\frac12\tr(A^aD^2\zeta)
				+\mu^a\cdot D\zeta-r^a\zeta
				\le-\beta
				\quad\hbox{on }[0,T]\times O.
				\label{eq:quantitative-barrier-inequality}
			\end{equation}
			
			\item[\textup{(Q3)}]
			There is \(C_{\rm comp}>0\) such that
			\begin{equation}
				|\Psi(x)-G(T,x)|
				\le
				C_{\rm comp}\zeta(T,x),
				\qquad
				x\in\overline O.
				\label{eq:quantitative-terminal-compatibility}
			\end{equation}
		\end{enumerate}
	\end{assumption}
	
	Conditions \textup{(Q1)}--\textup{(Q3)} imply
\(u_\partial(T)=G(T)=\Psi\) on \(\partial O\).
After time reversal, \(\zeta/\beta\) and \(G\) satisfy the barrier
and lifting assumptions (A2)--(A3) of
\cite{PicarelliReisingerRotaetxe2020}; \textup{(Q2)} becomes its
inequality (3.3), and \textup{(Q3)} becomes its compatibility
condition (3.6).  The \(C^{2,\gamma}\) bounds supply its assumption (A5).

For the auxiliary boundary data before \(t=0\), define
\begin{equation}
\widehat G(t,x)=
\begin{cases}
3G(-t,x)-2G(-2t,x),&-T/2\le t<0,\\
G(t,x),&0\le t\le T.
\end{cases}
\label{eq:quantitative-G-extension}
\end{equation}
At \(t=0\), the values and spatial derivatives agree because
\(3-2=1\), and the time derivatives agree because \(-3+4=1\).
The fixed time rescalings preserve the H\"older bounds, giving
\begin{equation}
\|\widehat G\|_{2,\gamma;[-T/2,T]\times\overline O}
\le C\|G\|_{2,\gamma;[0,T]\times\overline O}.
\label{eq:quantitative-G-extension-bound}
\end{equation}
For \(\varepsilon_0\le\min\{1,\sqrt T/2\}\), this defines the data
on \([-2\varepsilon^2,T]\) for every \(\varepsilon\le\varepsilon_0\),
while
\begin{equation}
\widehat G=G\quad\text{on }[0,T]\times\overline O.
\label{eq:quantitative-extension-agreement}
\end{equation}
Lemma~\ref{lem:extended-cylinder-barrier} supplies a strict barrier on
the extended cylinders, with constants independent of \(\varepsilon\).

\begin{lemma}[Equivalence of weak and strong boundary conditions]
		\label{lem:relaxed-is-strong}
		Under Assumptions~\ref{ass:bounded} and
		\ref{ass:quantitative-boundary}, every bounded weak Dirichlet viscosity
		subsolution is a strong subsolution, and every bounded weak Dirichlet
		viscosity supersolution is a strong supersolution.  Conversely, every
		strong subsolution or supersolution is weak.
		
		More precisely, set
		\[
		C_G:=\sup_{(t,x,a)\in[0,T]\times\overline O\times A}
		|F^a[G](t,x)|
		\]
		and choose
		\[
		K\ge\max\left\{C_{\rm comp},\frac{C_G+1}{\beta}\right\}.
		\]
		Then every bounded weak subsolution \(v\) and every bounded weak
		supersolution \(w\) satisfy
		\begin{equation}
			v\le G+K\zeta,
			\qquad
			w\ge G-K\zeta
			\quad\hbox{on }[0,T]\times\overline O.
			\label{eq:relaxed-strong-sandwich}
		\end{equation}
		Since \textup{(Q1)} gives \(G=u_\partial\) and \(\zeta=0\) on the lateral
		boundary, these bounds reduce there to \(v\le u_\partial\) and
		\(w\ge u_\partial\).  These are exactly the pointwise boundary inequalities
		missing from the weak definition.  Hence the weak and strong subsolution
		classes, supersolution classes, and continuous solution classes coincide.
	\end{lemma}
	
	\begin{proof}
Set \(\Phi_+=G+K\zeta\) and \(\Phi_-=G-K\zeta\).  By
\textup{(Q2)} and the choice of \(K\), for every control,
\begin{equation}
F^a[\Phi_+]\le C_G-K\beta\le-1,
\qquad F^a[\Phi_-]\ge-C_G+K\beta\ge1.
\label{eq:relaxed-strong-strict-barriers}
\end{equation}
Also, \textup{(Q3)} gives \(\Phi_-(T)\le\Psi\le\Phi_+(T)\), while
\(\Phi_-=\Phi_+=u_\partial\) on the lateral boundary.

Suppose that a bounded upper semicontinuous weak subsolution \(v\)
satisfies \(m:=\max(v-\Phi_+)>0\).  At a maximizing point,
\(\Phi_++m\) touches \(v\) from above and, since \(r^a\ge0\),
\(F^a[\Phi_++m]\le-1\).  An interior maximum contradicts the
subsolution inequality \(\inf_aF^a\ge0\).  At a lateral maximum,
\(u_\partial-v=-m<0\), so both entries in the weak maximum condition
\eqref{eq:viscosity-lateral} are negative.  A terminal maximum, including at
a corner, is excluded by \eqref{eq:viscosity-weak-terminal} and
\textup{(Q3)}: \(v(T)\le\Psi\le\Phi_+(T)\) on \(\overline O\).
Hence \(v\le\Phi_+\).

For a bounded lower semicontinuous weak supersolution \(w\), a positive
maximum \(m:=\max(\Phi_--w)\) supplies the lower test function
\(\Phi_--m\), with \(F^a[\Phi_--m]\ge1\) for every control.
This contradicts the interior inequality \(\inf_aF^a\le0\), or the
weak lateral minimum condition \eqref{eq:viscosity-lateral-super}, whose
other entry is \(u_\partial-w=m>0\).  The terminal ordering again
excludes a terminal maximum.  Thus \(w\ge\Phi_-\).

Restricting these bounds to the lateral boundary proves the strong
inequalities.  Conversely, the strong boundary inequalities imply the
weak max/min conditions directly.
\end{proof}
	
	\begin{remark}[Examples of a common barrier]
\label{rem:quantitative-barrier-existence}
For time-independent coefficients in dimension \(N\ge2\), the
uniformly elliptic construction in \cite[Theorem~3.2]{Safonov1994}
gives, under its domain and coefficient hypotheses,
\(\zeta\in C^{2,\gamma_b}(\overline O)\) solving
\begin{equation}
\sup_{a\in A}\{\tfrac12\tr(A^aD^2\zeta)+\mu^a\cdot D\zeta-r^a\zeta\}
=-1\quad\text{in }O,\qquad \zeta=0\quad\text{on }\partial O.
\label{eq:stationary-barrier-problem}
\end{equation}
The maximum principle gives \(\zeta>0\), so \textup{(Q2)} holds
for every control with \(\beta=1\).  Degenerate examples are discussed
in \cite[Example~2.3]{DongKrylov2007}.  In
Appendix~\ref{app:numerical-mesh}, the explicit disk barrier
\(1-|x|^2\) verifies \textup{(Q2)} for controlled drift and rank-one
diffusion.
\end{remark}

Throughout the remainder of this section, \(C\) denotes a constant that may
	depend on the fixed data in Assumptions~\ref{ass:bounded} and
	\ref{ass:quantitative-boundary}, on the domain, and on the fixed cubature,
	but not on \(h,\dt,n,i,a\), or any regularization parameter.
	
	\subsection{Regularity of the original and auxiliary problems}
	\label{subsec:quantitative-regularity}
	
	Throughout this section, the solution \(u\) solves
	\begin{equation}
		\begin{cases}
			\displaystyle
			\inf_{a\in A}F^a[u](t,x)=0,
			&(t,x)\in Q_T,\\
			u(t,x)=u_\partial(t,x),
			&(t,x)\in[0,T)\times\partial O,\\
			u(T,x)=\Psi(x),
			&x\in\overline O.
		\end{cases}
		\label{eq:quantitative-physical-problem}
	\end{equation}
The change of variables is \(s=T-t\), with
\(U(s,x)=u(T-s,x)\).  In the forward convention of
\cite{PicarelliReisingerRotaetxe2020}, the spatial operator is
\[
\mathcal L^a(s,x,z,p,X)
=-\tfrac12\tr\bigl(A^a(T-s,x)X\bigr)
 -\mu^a(T-s,x)\cdot p+r^a(T-s,x)z-f^a(T-s,x).
\]
Thus \(U_s+\sup_a\mathcal L^a(s,x,U,DU,D^2U)=0\) is exactly
\eqref{eq:quantitative-physical-problem}.  The initial value at \(s=0\)
becomes the terminal value at \(t=T\), and forward lateral data
\(\widehat G(T-s,x)\) become \(\widehat G(t,x)\).

For the auxiliary equations in
\cite[equations~(4.6) and (4.11)]{PicarelliReisingerRotaetxe2020},
the same substitution sends a coefficient argument \(s+\theta\) to
\(T-(s+\theta)=t-\theta\).  It sends the scalar and switching
intervals to \([-\varepsilon^2,T]\) and \([-2\varepsilon^2,T]\),
respectively, with terminal datum \(\Psi\) at \(T\).

	We now define the two auxiliary problems.  For
	\(0<\varepsilon\le1\), let
	\begin{equation}
		\Theta_\varepsilon
		:=
		[0,\varepsilon^2]\times\overline B_\varepsilon(0).
		\label{eq:quantitative-shaking-set}
	\end{equation}
	For \(a\in A\), \((\theta,e)\in\Theta_\varepsilon\), and smooth \(w\), set
	\begin{align}
		F_{\theta,e}^a[w](t,x)
		:={}&
		\partial_t w(t,x)
		+\frac12\tr\!\left(
		A^a(t-\theta,x+e)D^2w(t,x)
		\right)\notag\\
		&+\mu^a(t-\theta,x+e)\cdot Dw(t,x)
		-r^a(t-\theta,x+e)w(t,x)
		+f^a(t-\theta,x+e).
		\label{eq:quantitative-shifted-expression}
	\end{align}
	
	The scalar shaken problem is
	\begin{equation}
		\begin{cases}
			\displaystyle
			\inf_{\substack{a\in A\\(\theta,e)\in\Theta_\varepsilon}}
			F_{\theta,e}^a[u^\varepsilon](t,x)=0,
			&(t,x)\in[-\varepsilon^2,T)\times O,\\[1mm]
			u^\varepsilon(t,x)=\widehat G(t,x),
			&(t,x)\in[-\varepsilon^2,T)\times\partial O,\\
			u^\varepsilon(T,x)=\Psi(x),
			&x\in\overline O.
		\end{cases}
		\label{eq:quantitative-scalar-shaken}
	\end{equation}
Under \(s=T-t\), its time interval becomes
\([0,T+\varepsilon^2]\).  More explicitly, if
\(U^\varepsilon(s,x)=u^\varepsilon(T-s,x)\), the forward equation is
\[
\partial_s U^\varepsilon+
\sup_{\substack{a\in A\\0\le\theta\le\varepsilon^2,\ |e|\le\varepsilon}}
\mathcal L^a(s+\theta,x+e,
 U^\varepsilon,DU^\varepsilon,D^2U^\varepsilon)=0.
\]
All values and derivatives of \(U^\varepsilon\) in this display are
evaluated at \((s,x)\); only the coefficients are shifted.

	For the switching construction, choose
	\(a_1,\ldots,a_J\in A\), \(J\ge2\), and \(k>0\).  The backward system used
	here is
	\begin{equation}
		\begin{cases}
			\displaystyle
			\max\left\{
			-\sup_{(\theta,e)\in\Theta_\varepsilon}
			F_{\theta,e}^{a_i}[v_i^{\varepsilon,k}](t,x),
			\;
			v_i^{\varepsilon,k}(t,x)
			-\min_{j\ne i}\bigl(v_j^{\varepsilon,k}(t,x)+k\bigr)
			\right\}=0,
			&
			\begin{array}{l}
				(t,x)\in[-2\varepsilon^2,T)\times O,\\
				i=1,\ldots,J,
			\end{array}\\[4mm]
			v_i^{\varepsilon,k}(t,x)=\widehat G(t,x),
			&
			\begin{array}{l}
				(t,x)\in[-2\varepsilon^2,T)\times\partial O,\\
				i=1,\ldots,J,
			\end{array}\\[3mm]
			v_i^{\varepsilon,k}(T,x)=\Psi(x),
			&
			\begin{array}{l}
				x\in\overline O,\\
				i=1,\ldots,J.
			\end{array}
		\end{cases}
		\label{eq:quantitative-shaken-switching}
	\end{equation}
For \(V_i(s,x)=v_i^{\varepsilon,k}(T-s,x)\), the forward system is
\[
\max\left\{
\partial_s V_i+
\inf_{\substack{0\le\theta\le\varepsilon^2\\|e|\le\varepsilon}}
\mathcal L^{a_i}(s+\theta,x+e,V_i,DV_i,D^2V_i),
\quad V_i-\min_{j\ne i}(V_j+k)
\right\}=0
\]
on \((0,T+2\varepsilon^2]\times O\), with initial datum \(\Psi\)
and lateral datum \(\widehat G(T-s,x)\).  The switching cost \(k\)
is unchanged.  The minus sign before the backward supremum in
\eqref{eq:quantitative-shaken-switching} follows from
\(-\sup F=\inf(-F)\).

The extra \(\varepsilon^2\) in the switching interval is inherited
from the forward construction in
\cite[Section~4.4]{PicarelliReisingerRotaetxe2020}.
Writing the shift in equation~(4.12) of that paper as
\(\overline V_i(s,x)=V_i(s+\varepsilon^2,x)\), this function is
defined on \([-\varepsilon^2,T+\varepsilon^2]\), because
\(V_i\) is defined on \([0,T+2\varepsilon^2]\).
We keep this larger auxiliary interval.  The convolution and the times
at which it is used for comparison are specified in
Section~\ref{subsec:quantitative-mollification}.

For every \(g\in\{\sigma,\mu,r,f\}\), the coefficient extensions give
\begin{equation}
\sup_{\substack{a\in A,\ (t,x)\in[0,T]\times\overline O\\
                 (\theta,e)\in\Theta_\varepsilon}}
|g^a(t-\theta,x+e)-g^a(t,x)|\le C\varepsilon.
\label{eq:auxiliary-coefficient-distance}
\end{equation}
Indeed, the parabolic distance between the two arguments is at most
\(\sqrt\theta+|e|\le2\varepsilon\).  This estimate also applies when
\(t-\theta<0\), by \eqref{eq:McShane-basic-properties}.
It is the coefficient bound used in the continuous-dependence argument
of \cite[Theorem~3.8]{PicarelliReisingerRotaetxe2020}.
Lemma~\ref{lem:extended-cylinder-barrier} supplies the uniform strict
barrier for both auxiliary cylinders.

	\begin{lemma}[Uniform regularity]
\label{lem:quantitative-regularity}
There exist \(L_{\rm reg}\ge1\), \(\varepsilon_0>0\), and \(k_0>0\),
depending only on the fixed data, such that the physical problem and
the auxiliary problems \eqref{eq:quantitative-scalar-shaken} and
\eqref{eq:quantitative-shaken-switching} have unique bounded continuous
strong Dirichlet solutions.  For
\(0<\varepsilon\le\varepsilon_0\) and \(0<k\le k_0\),
\begin{align}
\max\{\|u\|_\infty,[u]_{x;1},[u]_{t;1/2}\}&\le L_{\rm reg},
 &&\text{on }[0,T]\times\overline O,
\label{eq:quantitative-u-uniform}\\
\max\{\|u^\varepsilon\|_\infty,[u^\varepsilon]_{x;1},
 [u^\varepsilon]_{t;1/2}\}&\le L_{\rm reg},
 &&\text{on }[-\varepsilon^2,T]\times\overline O,
\label{eq:quantitative-shaken-uniform}\\
\max_i\max\{\|v_i^{\varepsilon,k}\|_\infty,[v_i^{\varepsilon,k}]_{x;1},
 [v_i^{\varepsilon,k}]_{t;1/2}\}&\le L_{\rm reg},
 &&\text{on }[-2\varepsilon^2,T]\times\overline O.
\label{eq:quantitative-switching-uniform}
\end{align}
The constant is independent of the number \(J\ge2\) and the choice of
controls in the switching system.  The physical solution is also the
unique bounded weak solution.
\end{lemma}

\begin{proof}
The time reversal and assumption correspondence above allow us to use
\cite[Theorems~3.3, 3.5, and~3.7]{PicarelliReisingerRotaetxe2020}:
they give comparison, existence, and the stated spatial Lipschitz and
temporal one-half-H\"older bounds for the physical solution.
Lemma~\ref{lem:relaxed-is-strong} identifies the weak and strong classes.

For the auxiliary problems, time reversal preserves the parabolic
distance.  Every shifted coefficient has the same spatial Lipschitz
and temporal one-half-H\"older bounds, uniformly in the shift.
The lifting bound \eqref{eq:quantitative-G-extension-bound} and
Lemma~\ref{lem:extended-cylinder-barrier} give the boundary estimates
on the longer cylinders.

We apply the comparison, existence and regularity arguments of
\cite[Section~3.1, Proposition~3.10, and Lemma~4.8]{PicarelliReisingerRotaetxe2020}
to these coefficient families.  Their estimates are uniform over the
shift parameters: taking a supremum or an infimum preserves a common
bound on the coefficients, and the strict barrier applies to each
fixed shift.  The forward equations have the coefficient arguments
\(s+\theta\) and cylinder lengths at most \(T+2\).  They yield the stated uniform
bounds for the scalar equation and the minimum-shaken switching
system.  The constants in the switching estimates are independent of
the number of modes.  Returning to \(t=T-s\) proves the result.
\end{proof}

\subsection{Interior residual and high-order consistency}
	
	The numerical update at an interior node is
	\[
	u_{h,i}^n=\inf_{a\in A}S_{n,i}^a[u_h^{n+1}].
	\]
	For the consistency and comparison arguments below, we measure the defect in
	this equality per unit effective time.  Thus, for two adjacent grid layers
	\(U^n\) and \(U^{n+1}\), define
	\begin{equation}
		\mathcal R_{h,n,i}^a[U_i^n,U^{n+1}]
		:=\frac{S_{n,i}^a[U^{n+1}]-U_i^n}{\tau_{n,i}(a)},
		\qquad
		\mathcal S_{h,n,i}[U_i^n,U^{n+1}]
		:=\inf_{a\in A}\mathcal R_{h,n,i}^a[U_i^n,U^{n+1}].
		\label{eq:bc-residual-main}
	\end{equation}
	The denominator must remain inside the infimum because the effective time
	\(\tau_{n,i}(a)\) depends on the control.  For the node under consideration,
	Lemma~\ref{lem:effective-time-lower-bound} gives
	\[
	0<m_{n,i}\le \tau_{n,i}(a)\le\dt
	\]
	for every control \(a\) at this interior node, and therefore
	\begin{equation}
		\mathcal S_{h,n,i}[u_{h,i}^n,u_h^{n+1}]=0
		\quad\Longleftrightarrow\quad
		u_{h,i}^n=\inf_{a\in A}S_{n,i}^a[u_h^{n+1}].
		\label{eq:bc-residual-equivalence}
	\end{equation}
	Indeed, if the update holds, the numerators are nonnegative and a
minimizing sequence makes them tend to zero; division by
\(\tau_{n,i}(a)\ge m_{n,i}\) preserves that limit.  Conversely, if the
residual infimum is zero, all numerators are nonnegative and a minimizing
sequence of residuals gives numerators tending to zero because
\(\tau_{n,i}(a)\le\dt\).

We shall use two constants to describe how far one complete numerical update
	can reach.  The mesh assumption \textup{(G4)} provides a constant
	\(C_{\rm patch}\), independent of \(h\), such that every point of the
	interpolation simplex associated with an endpoint \(y\) lies within
	\(C_{\rm patch}h\) of \(y\).  We also recall the branch-displacement constant
	from Lemma~\ref{lem:effective-time-lower-bound}:
	\begin{equation}
		B_{\mathcal C}
		:=1+
		\sup_{\substack{
				(t,x,a)\in[0,T]\times\overline O\times A\\
				1\le j\le M}}
		\left(
		|\mu(t,x,a)|+|\sigma(t,x,a)\eta_j|
		\right).
		\label{eq:main-branch-bound-constant}
	\end{equation}
	During one time step, every branch moves by
	at most \(B_{\mathcal C}\sqrt\dt\).
	
	\begin{theorem}[High-order consistency at nodes away from the boundary]
		\label{thm:main-high-order-consistency}
		Let \(0<\dt\le1\), \(0<\varepsilon\le1\), and let \(\phi\) be smooth on
		\([t_n,t_{n+1}]\times O_\varepsilon\).
		Assume that, on this region,
		\begin{equation}
			\|\phi\|_\infty\le C_{\rm reg},
			\qquad
			\|D_x^\alpha\partial_t^q\phi\|_\infty
			\le C_{\rm reg}\varepsilon^{1-|\alpha|-2q},
			\qquad 1\le |\alpha|+2q\le6.
			\label{eq:main-mollified-derivatives}
		\end{equation}
		Here all norms are taken over \([t_n,t_{n+1}]\times O_\varepsilon\), and
		\(C_{\rm reg}\) is independent of \(\varepsilon\), \(h\), \(\dt\), and the
		control.  
		
		Assume that the current node has the boundary clearance
		\begin{equation}
			\operatorname{dist}(x_i,\partial O)
			\ge
			2\varepsilon
			+B_{\mathcal C}\sqrt\dt
			+C_{\rm patch}h.
			\label{eq:main-high-order-collars}
		\end{equation}
		The last two terms cover the maximum branch displacement and the width of
		the interpolation simplex.  The remaining \(2\varepsilon\) margin leaves a
		full \(\varepsilon\)-neighborhood of every sampled point inside
		\(O_\varepsilon\).
		Then, for every control and every signed cubature direction,
		\begin{equation}
			d_{n,i}^{a,j,\pm}=\dt,
			\qquad
			\tau_{n,i}(a)=\dt,
			\qquad
			\omega_{n,i}^{a,j,\pm}=\frac{\rho_j}{2},
			\label{eq:main-no-stop-weights}
		\end{equation}
		and every point used for interpolation is at least \(2\varepsilon\) from
		\(\partial O\), and hence belongs to \(O_\varepsilon\).  Let
		\(\phi_h^{n+1}\) be any next-layer grid vector that agrees with
		\(\phi(t_{n+1},\cdot)\) at every interpolation vertex used by these
		full-step branches.  Its values at all other vertices do not enter this
		update.  Then, uniformly in \(a\),
		\begin{align}
			&\left|
			\mathcal R_{h,n,i}^a
			[\phi(t_n,x_i),\phi_h^{n+1}]
			-F^a[\phi](t_n,x_i)
			\right|\notag\\
			&\qquad\le C\left(
			\dt\varepsilon^{-3}+\dt^2\varepsilon^{-5}
			+\frac{h^2}{\dt}\varepsilon^{-1}
			\right).
			\label{eq:main-high-order-fixed}
		\end{align}
		Consequently,
		\begin{align}
			&\left|
			\mathcal S_{h,n,i}
			[\phi(t_n,x_i),\phi_h^{n+1}]
			-\inf_{a\in A}F^a[\phi](t_n,x_i)
			\right|\notag\\
			&\qquad\le C\left(
			\dt\varepsilon^{-3}+\dt^2\varepsilon^{-5}
			+\frac{h^2}{\dt}\varepsilon^{-1}
			\right).
			\label{eq:main-high-order-HJB}
		\end{align}
		The constant \(C\) is independent of \(h\), \(\dt\), \(\varepsilon\),
		\(n\), \(i\), and \(a\).
	\end{theorem}
	
	\begin{proof}
We first estimate the endpoint average, then add discount, source, and
interpolation.  Condition~\eqref{eq:main-high-order-collars} keeps
every branch at least \(2\varepsilon+C_{\rm patch}h\) from the boundary.
Thus \(d:=\dt\), the weights satisfy \eqref{eq:main-no-stop-weights},
and all active interpolation simplices and Taylor segments lie in
\(O_\varepsilon\).

Freeze the coefficients at \((t_n,x_i,a)\), and put
\[
\phi_0=\phi(t_n,x_i),\qquad
y_j^\pm=d\mu^a\pm\sqrt d\,\sigma^a\eta_j,\qquad
\overline\phi_d^{\,a}
=\sum_j\frac{\rho_j}{2}
 \bigl[\phi(t_n+d,x_i+y_j^+)+\phi(t_n+d,x_i+y_j^-)\bigr].
\]
In each pair, terms with an odd number of diffusion factors cancel.
The first two moments are
\[
\sum_j\frac{\rho_j}{2}(y_j^++y_j^-)=d\mu^a,\qquad
\sum_{j,\pm}\frac{\rho_j}{2}y_j^\pm(y_j^\pm)^\top
=dA^a+d^2\mu^a(\mu^a)^\top.
\]
Lemmas~\ref{lem:app-sixth-order-taylor} and
\ref{lem:app-mollified-truncation}, with
\eqref{eq:main-mollified-derivatives}, therefore give
\begin{equation}
\overline\phi_d^{\,a}=\phi_0+dL^a\phi+R_{\rm tay}^a,\qquad
|R_{\rm tay}^a|\le C(d^2\varepsilon^{-3}+d^3\varepsilon^{-5}),
\label{eq:main-average-expansion}
\end{equation}
where \(L^a\phi=\phi_t+\mu^a\cdot D\phi+\tfrac12\tr(A^aD^2\phi)\)
at the current point.

Set \(q_d=(1+r^ad)^{-1}\) and
\(S_{{\rm ex},d}^a[\phi]=q_d\overline\phi_d^{\,a}+df^a\).
Subtracting \(F^a[\phi]=L^a\phi-r^a\phi_0+f^a\), and using
\(q_d-1=-r^adq_d\), gives the exact identity
\begin{equation}
\begin{aligned}
\frac{S_{{\rm ex},d}^a[\phi]-\phi_0}{d}-F^a[\phi]
&=(q_d-1)L^a\phi+r^a(1-q_d)\phi_0+q_dR_{\rm tay}^a/d\\
&=q_dR_{\rm tay}^a/d+dq_d\bigl[(r^a)^2\phi_0-r^aL^a\phi\bigr].
\end{aligned}
\label{eq:main-discount-identity}
\end{equation}
The derivative bounds imply \(|L^a\phi|\le C\varepsilon^{-1}\).
Since \(q_d\le1\) and \(\varepsilon\le1\), it follows that
\begin{equation}
\left|\frac{S_{{\rm ex},d}^a[\phi]-\phi_0}{d}-F^a[\phi]\right|
\le C(d\varepsilon^{-3}+d^2\varepsilon^{-5}).
\label{eq:main-exact-consistency}
\end{equation}

At every active vertex, the supplied grid vector equals
\(\phi(t_n+d,\cdot)\).  Hence \eqref{eq:Isp-error} gives
\[
\left|\Isp[\phi_h^{n+1}](x_i+y_j^\pm)
       -\phi(t_n+d,x_i+y_j^\pm)\right|
\le Ch^2(\|D\phi\|_\infty+\|D^2\phi\|_\infty)
\le Ch^2\varepsilon^{-1}.
\]
The positive weights sum to one, so the same bound controls the
difference between the numerical row and \(S_{{\rm ex},d}^a\).
Dividing it by \(\tau_{n,i}(a)=d\) and adding
\eqref{eq:main-exact-consistency} proves
\eqref{eq:main-high-order-fixed}.
Finally, \(|\inf_a X_a-\inf_a Y_a|\le\sup_a|X_a-Y_a|\)
gives \eqref{eq:main-high-order-HJB}.
\end{proof}

	\subsection{Stopped barriers}
	
	Fix once and for all the universal value \(\kappa:=1\), and define
	\[
	C_G
	:=
	\sup_{(t,x,a)\in[0,T]\times\overline O\times A}
	|F^a[G](t,x)|.
	\]
	Choose
	\begin{equation}
		K_{\rm bar}
		\ge
		\max\left\{
		C_{\rm comp},
		\frac{C_G+2\kappa}{\beta}
		\right\},
		\qquad
		\Phi_+:=G+K_{\rm bar}\zeta,
		\qquad
		\Phi_-:=G-K_{\rm bar}\zeta.
		\label{eq:bc-main-barriers}
	\end{equation}
	We show that these barriers remain discrete upper and lower bounds:
	the strict differential inequalities must dominate the stopped-row
	remainder after division by \(\tau\).
	
	\begin{proposition}[Stopped barriers]
		\label{prop:bc-main-barrier}
		Assume \(0<\dt\le1\).  There is a constant \(C_{\rm bar}>0\),
		independent of \(h\), \(\dt\), the row, and the control, such that every
		interior row, including a row on which one or more branches stop before
		\(t_{n+1}\), satisfies
		\begin{align}
			\mathcal R_{h,n,i}^a
			\bigl[\Phi_+(t_n,x_i),\Phi_+(t_{n+1},\cdot)\bigr]
			&\le
			-\kappa
			+C_{\rm bar}
			\left(
			\dt^{\gamma/2}+\frac{h^2}{\dt}
			\right),
			\label{eq:bc-main-bar-plus}\\
			\mathcal R_{h,n,i}^a
			\bigl[\Phi_-(t_n,x_i),\Phi_-(t_{n+1},\cdot)\bigr]
			&\ge
			\kappa
			-C_{\rm bar}
			\left(
			\dt^{\gamma/2}+\frac{h^2}{\dt}
			\right).
			\label{eq:bc-main-bar-minus}
		\end{align}
		If
		\begin{equation}
			C_{\rm bar}
			\left(
			\dt^{\gamma/2}+\frac{h^2}{\dt}
			\right)
			\le\frac{\kappa}{2},
			\label{eq:bc-main-bar-small}
		\end{equation}
		then
		\begin{equation}
			G-K_{\rm bar}\zeta
			\le u_h\le
			G+K_{\rm bar}\zeta,
			\qquad
			G-K_{\rm bar}\zeta
			\le u\le
			G+K_{\rm bar}\zeta.
			\label{eq:bc-main-sandwich}
		\end{equation}
	\end{proposition}
	
	\begin{proof}
Fix a row \((n,i)\) and a control \(a\), and freeze all coefficients
there.  Write \(\tau=\tau_{n,i}(a)\), \(s_j^\pm=\sqrt{d_j^\pm}\),
\(z_j^\pm=\pm\sigma\eta_j\),
\(y_j^\pm=(s_j^\pm)^2\mu+s_j^\pm z_j^\pm\), and
\(q_j^\pm=(1+r(s_j^\pm)^2)^{-1}\).
The two durations in a pair remain separate throughout the calculation.

\smallskip
\noindent\emph{Stopped endpoint values and the remainder.}
For \(\phi=\Phi_+\) or \(\Phi_-\), each stopped continuation equals
\(\phi(t_n+d_j^\pm,x_i+y_j^\pm)\).  Indeed, \(\chi_j^\pm=1\)
means that the endpoint \(x_i+y_j^\pm\) lies on \(\partial O\),
by \eqref{eq:branch-duration}.  Before \(T\), equality follows from
\(G=u_\partial\) and \(\zeta=0\).  At \(T\), \textup{(Q1)} gives
\(\zeta(T,x_i+y_j^\pm)=0\), so \textup{(Q3)} gives
\(\Psi=G(T)=\Phi_+(T)=\Phi_-(T)\) at this endpoint.
Unstopped branches read the supplied next-layer vector
\(\phi(t_{n+1},\cdot)\), also when \(t_{n+1}=T\).
Denote by \(S_{n,i}^{a,{\rm ex}}[\phi]\) the row with point evaluation
instead of interpolation on these unstopped branches.

The curve \(x_i+\lambda^2\mu+\lambda z_j^\pm\),
\(0\le\lambda\le s_j^\pm\), stays in \(\overline O\) up to first
contact.  Lemma~\ref{lem:bc-app-branch-remainder} therefore gives, with
derivatives at \((t_n,x_i)\) and \(\phi_0=\phi(t_n,x_i)\),
\begin{align}
\phi(t_n+d_j^\pm,x_i+y_j^\pm)
={}&\phi_0+(s_j^\pm)^2(\phi_t+\mu\cdot D\phi)
 +s_j^\pm D\phi\cdot z_j^\pm\notag\\
&+\tfrac12(s_j^\pm)^2(z_j^\pm)^\top D^2\phi\,z_j^\pm
 +(s_j^\pm)^3\mu^\top D^2\phi\,z_j^\pm\notag\\
&+\tfrac12(s_j^\pm)^4\mu^\top D^2\phi\,\mu+\mathfrak r_j^\pm,
\label{eq:bc-main-branch-taylor}
\end{align}
where \(|\mathfrak r_j^\pm|\le C(s_j^\pm)^{2+\gamma}\).
The moment identities \eqref{eq:weight-moments} imply, for every
\(\ell\ge2\),
\begin{equation}
\sum_{j,\pm}\omega_j^\pm(s_j^\pm)^\ell
\le\dt^{(\ell-2)/2}\sum_{j,\pm}\omega_j^\pm(s_j^\pm)^2
=\tau\dt^{(\ell-2)/2}.
\label{eq:bc-main-stopped-powers}
\end{equation}
This factor \(\tau\) is retained in every error estimate below.

Since \(r\ge0\) is bounded,
\[
0<q_j^\pm\le1,\qquad
q_j^\pm=1-r(s_j^\pm)^2+O((s_j^\pm)^4).
\]
Substituting \eqref{eq:bc-main-branch-taylor} in
\[
S_{n,i}^{a,{\rm ex}}[\phi]-\phi_0
=\sum_{j,\pm}\omega_j^\pm q_j^\pm
  [\phi(t_n+d_j^\pm,x_i+y_j^\pm)-\phi_0]
 +\phi_0\sum_{j,\pm}\omega_j^\pm(q_j^\pm-1)+f\tau
\]
and using \eqref{eq:weight-moments} cancels the first noise moment and
gives the diffusion term \(\tfrac12\tau\tr(A^aD^2\phi)\).  Thus
\begin{equation}
S_{n,i}^{a,{\rm ex}}[\phi]-\phi_0
=\tau F^a[\phi](t_n,x_i)+R_{\rm ex}[\phi],
\qquad |R_{\rm ex}[\phi]|\le C\tau\dt^{\gamma/2}.
\label{eq:bc-main-exact-consistency}
\end{equation}
To estimate \(R_{\rm ex}\), first note that the undiscounted noise
terms cancel by \eqref{eq:weight-moments}:
\[
\sum_{j,\pm}\omega_j^\pm s_j^\pm D\phi\cdot z_j^\pm=0.
\]
The discount leaves only the correction with the additional factor
\(q_j^\pm-1=O((s_j^\pm)^2)\).  This correction and the
\((s_j^\pm)^3\mu^\top D^2\phi\,z_j^\pm\) terms in
\eqref{eq:bc-main-branch-taylor} are bounded by
\(C\sum_{j,\pm}\omega_j^\pm(s_j^\pm)^3\).
The quadratic drift terms and the other discount corrections are bounded
by \(C\sum_{j,\pm}\omega_j^\pm(s_j^\pm)^4\).
Finally, Lemma~\ref{lem:bc-app-branch-remainder} bounds the Taylor
remainders by \(C\sum_{j,\pm}\omega_j^\pm(s_j^\pm)^{2+\gamma}\).
Applying \eqref{eq:bc-main-stopped-powers} to these three sums gives
\[
|R_{\rm ex}[\phi]|
\le C\tau\left(\Delta t^{1/2}+\Delta t+\Delta t^{\gamma/2}\right).
\]
Since \(0<\gamma\le1\) and \(\dt\le1\), both \(\dt^{1/2}\)
and \(\dt\) are at most \(\dt^{\gamma/2}\), proving the remainder
bound in \eqref{eq:bc-main-exact-consistency}.

By \eqref{eq:interpolated-branch-mass}, the total interpolated weight
is at most \(2\tau/\dt\).  Applying \eqref{eq:Isp-error}, including
the evaluation map \(p_h\), gives

\begin{equation}
\big|S_{n,i}^a[\phi(t_{n+1},\cdot)]-S_{n,i}^{a,{\rm ex}}[\phi]\big|
\le C\tau\frac{h^2}{\dt}.
\label{eq:bc-main-bar-interpolation}
\end{equation}
Together with \eqref{eq:bc-main-exact-consistency}, this proves
\begin{equation}
S_{n,i}^a[\phi(t_{n+1},\cdot)]-\phi_0
=\tau F^a[\phi](t_n,x_i)+R_{n,i}^a[\phi],
\quad |R_{n,i}^a[\phi]|
\le C_{\rm bar}\tau\left(\dt^{\gamma/2}+\frac{h^2}{\dt}\right).
\label{eq:bc-main-bar-master}
\end{equation}

\smallskip
\noindent\emph{Barrier comparison.}
Condition \textup{(Q2)} and the choice of \(K_{\rm bar}\) give
\begin{equation}
F^a[\Phi_+]\le C_G-K_{\rm bar}\beta\le-2\kappa,
\qquad
F^a[\Phi_-]\ge-C_G+K_{\rm bar}\beta\ge2\kappa.
\label{eq:bc-main-cont-barriers}
\end{equation}
These inequalities make \(\Phi_-\) a subsolution and \(\Phi_+\) a
supersolution of \eqref{eq:quantitative-physical-problem}.
On the lateral boundary, both equal \(u_\partial\), since \(G=u_\partial\)
and \(\zeta=0\).  At the terminal time, \textup{(Q3)} and
\(K_{\rm bar}\ge C_{\rm comp}\) give
\[
\Phi_-(T,\cdot)\le\Psi\le\Phi_+(T,\cdot).
\]
The solution \(u\) attains these data by
Lemma~\ref{lem:quantitative-regularity}.  Applying continuous comparison
\cite[Theorem~3.3]{PicarelliReisingerRotaetxe2020} after time reversal
therefore yields \(\Phi_-\le u\le\Phi_+\).

For the discrete solution, divide \eqref{eq:bc-main-bar-master} by
\(\tau>0\) and use \eqref{eq:bc-main-cont-barriers}.  Weakening the
fixed margin to \(\kappa\) gives
\eqref{eq:bc-main-bar-plus}--\eqref{eq:bc-main-bar-minus}.
Under \eqref{eq:bc-main-bar-small}, every control row satisfies
\[
S_{n,i}^a[\Phi_+(t_{n+1},\cdot)]\le\Phi_+(t_n,x_i),
\qquad
S_{n,i}^a[\Phi_-(t_{n+1},\cdot)]\ge\Phi_-(t_n,x_i).
\]
Taking the infimum over \(a\) gives the discrete supersolution and
subsolution inequalities.  Together with the lateral and terminal
ordering above, Proposition~\ref{prop:bounded-monotone} gives
\(\Phi_-\le u_h\le\Phi_+\).  This proves
\eqref{eq:bc-main-sandwich}; in particular, at every grid point,
\[
|u-u_h|\le\Phi_+-\Phi_-=2K_{\rm bar}\zeta.
\]
\end{proof}
	
	\subsection{Mollified scalar and switching comparison functions}
	\label{subsec:quantitative-mollification}
	
	The scalar regularization will bound \(u-u_h\); the switching
	regularization will bound \(u_h-u\).  Choose a nonnegative function
	\(\rho\in C_c^\infty((0,1)\times B_1(0))\) satisfying
	\[
	\int_0^1\int_{B_1(0)}\rho(s,y)\,dy\,ds=1,
	\]
	and set
	\begin{equation}
		\rho_\varepsilon(\vartheta,e)
		:=
		\varepsilon^{-N-2}
		\rho\!\left(
		\frac{\vartheta}{\varepsilon^2},
		\frac{e}{\varepsilon}
		\right).
		\label{eq:parabolic-mollifier}
	\end{equation}
Thus \(\rho_\varepsilon\) is nonnegative, has unit mass, and is
compactly supported in \(\Theta_\varepsilon\).
For \(V(s,x)=v(T-s,x)\), the shifted forward convolution in
\cite[equation~(4.12)]{PicarelliReisingerRotaetxe2020} becomes
\[
\int_{\Theta_\varepsilon}V(T-t+\varepsilon^2-\vartheta,x-e)
 \rho_\varepsilon(\vartheta,e)\,d\vartheta\,de
=\int_{\Theta_\varepsilon}v(t-\varepsilon^2+\vartheta,x-e)
 \rho_\varepsilon(\vartheta,e)\,d\vartheta\,de.
\]
The integration variables and the kernel are unchanged.

	For every sufficiently small \(\varepsilon>0\), let
	\(u^\varepsilon\) solve \eqref{eq:quantitative-scalar-shaken} and set
	\begin{equation}
		\phi_{\rm sc}^\varepsilon(t,x)
		:=
		\int_{\Theta_\varepsilon}
		u^\varepsilon(t-\varepsilon^2+\vartheta,x-e)
		\rho_\varepsilon(\vartheta,e)
		\,d\vartheta\,de,
		\qquad
		(t,x)\in[0,T]\times O_\varepsilon.
		\label{eq:scalar-backward-convolution}
	\end{equation}
	Choose \(\nu=\nu(\varepsilon)>0\) such that
	\begin{equation}
		\omega_A(\nu)\le\varepsilon^{1/3},
		\label{eq:quant-control-net-choice}
	\end{equation}
	and choose an indexed list \(a_1,\ldots,a_J\in A\), \(J\ge2\), whose
	range
	\[
	A_\nu:=\{a_m:1\le m\le J\}
	\]
	is a finite \(\nu\)-net of \(A\).  If this net consists of one point,
	repeat that point in the indexed list.  Set
	\begin{equation}
		k=16L_{\rm reg}\varepsilon,
		\label{eq:quant-switching-cost}
	\end{equation}
	and let \(v_m^{\varepsilon,k}\), \(m=1,\ldots,J\), be the components of
	the solution of \eqref{eq:quantitative-shaken-switching}.  For
	\((t,x)\in[0,T]\times O_\varepsilon\), set
	\begin{align}
		g_m^\varepsilon(t,x)
		&:=
		\int_{\Theta_\varepsilon}
		v_m^{\varepsilon,k}(t-\varepsilon^2+\vartheta,x-e)
		\rho_\varepsilon(\vartheta,e)
		\,d\vartheta\,de,
		\qquad m=1,\ldots,J,
		\label{eq:switching-backward-convolution}\\
		\phi_{\rm sw}^\varepsilon(t,x)
		&:=
		\min_{1\le m\le J}g_m^\varepsilon(t,x).
		\label{eq:quant-switching-envelope}
	\end{align}
	
To obtain a differential inequality with coefficients at \((t,x)\),
we use the convolution at \(t+\varepsilon^2\):
\[
\phi_{\rm sc}^\varepsilon(t+\varepsilon^2,x)
=\int_{\Theta_\varepsilon}u^\varepsilon(t+\vartheta,x-e)
 \rho_\varepsilon(\vartheta,e)\,d\vartheta\,de,\qquad
0\le t\le T-\varepsilon^2.
\]
The same identity holds for each switching component.  At a sampled point,
choosing the shake \((\theta,e)=(\vartheta,e)\) restores the coefficient
arguments \((t+\vartheta-\vartheta,x-e+e)=(t,x)\).
The terminal data stay at \(T\); Lemma~\ref{lem:discrete-terminal-modulus}
handles the remaining interval \([T-\varepsilon^2,T]\).

	\begin{proposition}[Mollified comparison functions]
		\label{prop:quantitative-comparison-functions}
		The functions \(\phi_{\rm sc}^\varepsilon\) and
		\(g_m^\varepsilon\), \(m=1,\ldots,J\), defined above belong to
		\(C^\infty([0,T]\times O_\varepsilon)\), and
		
		\begin{align}
			\|\phi_{\rm sc}^\varepsilon-u\|_
			{L^\infty([0,T]\times O_\varepsilon)}
			&\le C\varepsilon,
			\label{eq:quant-scalar-close}\\
			F^a[\phi_{\rm sc}^\varepsilon(\cdot+\varepsilon^2,\cdot)](t,x)
			&\ge0,
			\qquad
			a\in A,\quad
			(t,x)\in[0,T-\varepsilon^2]\times O_\varepsilon,
			\label{eq:quant-scalar-pde}\\
			\|\phi_{\rm sc}^\varepsilon(T)-\Psi\|_
			{L^\infty(O_\varepsilon)}
			&\le C\varepsilon.
			\label{eq:quant-scalar-terminal}
		\end{align}
		Although \(\phi_{\rm sw}^\varepsilon\) need not be smooth, it satisfies
		\begin{align}
			\|\phi_{\rm sw}^\varepsilon-u\|_
			{L^\infty([0,T]\times O_\varepsilon)}
			&\le C\varepsilon^{1/3},
			\label{eq:quant-switching-close}\\
			\|\phi_{\rm sw}^\varepsilon(T)-\Psi\|_
			{L^\infty(O_\varepsilon)}
			&\le C\varepsilon.
			\label{eq:quant-switching-terminal}
		\end{align}
		Moreover, for every
		\((t,x)\in[0,T-\varepsilon^2]\times O_\varepsilon\) and every index \(m_*\)
		satisfying
		\[
		m_*\in
		\operatorname*{argmin}_{1\le m\le J}
		g_m^\varepsilon(t+\varepsilon^2,x),
		\]
		the corresponding smooth component satisfies
		\begin{equation}
			F^{a_{m_*}}[g_{m_*}^\varepsilon(\cdot+\varepsilon^2,\cdot)](t,x)\le0.
			\label{eq:quant-switching-active}
		\end{equation}
		Finally, for \(1\le|\alpha|+2q\le6\),
		\begin{align}
			&\|D_x^\alpha\partial_t^q
			\phi_{\rm sc}^\varepsilon\|_
			{L^\infty([0,T]\times O_\varepsilon)}
			\notag\\
			&\qquad+
			\max_{1\le m\le J}
			\|D_x^\alpha\partial_t^q
			g_m^\varepsilon\|_
			{L^\infty([0,T]\times O_\varepsilon)}
			\le
			C_{\alpha,q}\varepsilon^{1-|\alpha|-2q}.
			\label{eq:quant-mollified-derivatives}
		\end{align}
		Here \(\alpha\) is a spatial multi-index, \(q\in\mathbb N_0\), and
		the constants are independent of \(\varepsilon\), \(\nu\), and \(J\).
	\end{proposition}
	
	\begin{proof}
The compactly supported kernel samples
\(-\varepsilon^2<t-\varepsilon^2+\vartheta<T\) and \(x-e\in O\);
hence the convolutions are smooth on the stated domain.
We verify their differential signs and approximation errors, then the
derivative and terminal bounds needed for consistency and comparison.

\smallskip
\noindent\emph{Scalar function.}
For fixed \(a\) and \((\vartheta,e)\in\operatorname{supp}\rho_\varepsilon\),
the shaken equation gives, in the viscosity sense,
\begin{equation}
F^a[u^\varepsilon(\,\cdot+\vartheta,\cdot-e)](t,x)
=F_{\vartheta,e}^a[u^\varepsilon](t+\vartheta,x-e)\ge0,
\qquad 0\le t\le T-\varepsilon^2.
\label{eq:scalar-translate-inequality}
\end{equation}
Every translated function thus satisfies the same linear inequality.
Finite convex combinations and local uniform limits preserve this
viscosity inequality.  Positive Riemann sums for the convolution give
\eqref{eq:quant-scalar-pde}, classically by smoothness, as in
\cite[Theorem~4.7]{PicarelliReisingerRotaetxe2020}.

For approximation, index the original and shaken equations by
\(A\times\Theta_\varepsilon\).  The original coefficients are independent
of the shift; the two coefficient tuples differ by \(C\varepsilon\),
by \eqref{eq:auxiliary-coefficient-distance}, and their lateral and
terminal data agree on the physical cylinder.  Continuous dependence
\cite[Theorem~3.8 and Proposition~3.10]{PicarelliReisingerRotaetxe2020}
therefore gives
\begin{equation}
\|u^\varepsilon-u\|_{L^\infty([0,T]\times\overline O)}
\le C\varepsilon.
\label{eq:quant-scalar-auxiliary-close}
\end{equation}
When \(t-\varepsilon^2+\vartheta\ge0\), this bound and the regularity
of \(u\) control each sampled difference by \(C\varepsilon\).
When \(t-\varepsilon^2+\vartheta<0\), compare through time zero:
\[
\begin{aligned}
&|u^\varepsilon(t-\varepsilon^2+\vartheta,x-e)-u(t,x)|\\
&\le |u^\varepsilon(t-\varepsilon^2+\vartheta,x-e)-u^\varepsilon(0,x-e)|
 +|u^\varepsilon(0,x-e)-u(0,x-e)|\\
&\quad+|u(0,x-e)-u(t,x)|\le C\varepsilon.
\end{aligned}
\]
Here \(-\varepsilon^2\le t-\varepsilon^2+\vartheta<0\) and
\(0\le t<\varepsilon^2\), so both time increments are controlled by
Lemma~\ref{lem:quantitative-regularity}.
Averaging with the unit-mass kernel proves \eqref{eq:quant-scalar-close}.

\smallskip
\noindent\emph{Switching function.}
Let \(u_\nu\) solve the HJB equation with the finite control set \(A_\nu\).
Index its coefficients and those of \(u\) by
\[
\widehat A_\nu=\{(a,b)\in A\times A_\nu:d_A(a,b)\le\nu\}.
\]
Both projections are surjective, so the operators are unchanged.
By \eqref{eq:extended-control-modulus}, the paired coefficients differ
by at most \(C\omega_A(\nu)\); continuous dependence gives
\(\|u_\nu-u\|_\infty\le C\omega_A(\nu)\).
Comparison with the unshaken switching system for the same controls and
cost gives an error \(C\varepsilon\)
\cite[Lemma~4.8]{PicarelliReisingerRotaetxe2020}, and its difference
from \(u_\nu\) is bounded by \(Ck^{1/3}\)
\cite[Theorem~3.11]{PicarelliReisingerRotaetxe2020}.
Together with the obstacle inequalities, these estimates yield
\begin{align}
\max_{i,j}\|v_i^{\varepsilon,k}-v_j^{\varepsilon,k}\|_
 {L^\infty([-2\varepsilon^2,T]\times\overline O)}
&\le k,
\label{eq:quant-switching-component-gap}\\
\max_m\|v_m^{\varepsilon,k}-u\|_
 {L^\infty([0,T]\times\overline O)}
&\le C\bigl(\varepsilon+k^{1/3}+\omega_A(\nu)\bigr).
\label{eq:quant-switching-components-close}
\end{align}
The sampled points are within parabolic distance \(2\varepsilon\)
of \((t,x)\), so
\begin{equation}
|g_m^\varepsilon(t,x)-v_m^{\varepsilon,k}(t,x)|
\le2L_{\rm reg}\varepsilon.
\label{eq:quant-switching-mollification-gap}
\end{equation}
Using \eqref{eq:quant-control-net-choice} and
\eqref{eq:quant-switching-cost} gives
\(\max_m\|g_m^\varepsilon-u\|_\infty\le C\varepsilon^{1/3}\).
Taking the pointwise minimum proves \eqref{eq:quant-switching-close}.

To obtain the opposite differential sign, fix
\((t,x)\in[0,T-\varepsilon^2]\times O_\varepsilon\) and choose
\(m_*\in\operatorname*{argmin}_m g_m^\varepsilon(t+\varepsilon^2,x)\).
We keep this one component throughout the convolution.
The regularity estimate gives
\[
|g_m^\varepsilon(t+\varepsilon^2,x)-v_m^{\varepsilon,k}(t,x)|
\le2L_{\rm reg}\varepsilon,
\qquad
v_{m_*}^{\varepsilon,k}(t,x)-v_j^{\varepsilon,k}(t,x)
\le4L_{\rm reg}\varepsilon\quad(j\ne m_*).
\]
Applying the same regularity bound to both components at each sampled
point, and using \(k=16L_{\rm reg}\varepsilon\), gives
\[
v_{m_*}^{\varepsilon,k}(t+\vartheta,x-e)
-v_j^{\varepsilon,k}(t+\vartheta,x-e)-k
\le8L_{\rm reg}\varepsilon-k=-8L_{\rm reg}\varepsilon<0.
\]
The obstacle is therefore inactive on a neighborhood of the compact
set of sampled points.  The differential branch gives there
\(F_{\theta,e'}^{a_{m_*}}[v_{m_*}^{\varepsilon,k}]\le0\)
for every \((\theta,e')\in\Theta_\varepsilon\).
Choose \((\theta,e')=(\vartheta,e)\), translate, and average as in the
scalar argument.  This proves
\eqref{eq:quant-switching-active} for the smooth component
\(g_{m_*}^\varepsilon\), without differentiating the minimum.

\smallskip
\noindent\emph{Derivative and terminal estimates.}
For \(1\le|\alpha|+2q\le6\), differentiate the kernel, or equivalently
integrate by parts in the distributional sense:
\begin{equation}
D_x^\alpha\partial_t^qg_m^\varepsilon(t,x)
=(-1)^q\int_{\Theta_\varepsilon}
 v_m^{\varepsilon,k}(t-\varepsilon^2+\vartheta,x-e)
 D_e^\alpha\partial_\vartheta^q\rho_\varepsilon\,d\vartheta\,de.
\label{eq:quant-convolution-derivative}
\end{equation}
Each time derivative contributes a minus sign; the spatial signs
cancel.  The differentiated kernel has zero integral, so
\begin{equation}
\begin{aligned}
D_x^\alpha\partial_t^qg_m^\varepsilon(t,x)
=(-1)^q\int_{\Theta_\varepsilon}
 &[v_m^{\varepsilon,k}(t-\varepsilon^2+\vartheta,x-e)
      -v_m^{\varepsilon,k}(t,x)]\\
&\hspace{8mm}\cdot D_e^\alpha\partial_\vartheta^q
 \rho_\varepsilon\,d\vartheta\,de.
\end{aligned}
\label{eq:quant-convolution-cancellation}
\end{equation}
Regularity and kernel scaling give, respectively,
\begin{align}
|v_m^{\varepsilon,k}(t-\varepsilon^2+\vartheta,x-e)
      -v_m^{\varepsilon,k}(t,x)|
&\le2L_{\rm reg}\varepsilon,
\label{eq:quant-convolution-increment}\\
\|D_e^\alpha\partial_\vartheta^q\rho_\varepsilon\|_{L^1(\Theta_\varepsilon)}
&\le C_{\alpha,q}\varepsilon^{-|\alpha|-2q}.
\label{eq:quant-convolution-kernel}
\end{align}
Their product is \(C_{\alpha,q}\varepsilon^{1-|\alpha|-2q}\).
The identical calculation for \(u^\varepsilon\) proves
\eqref{eq:quant-mollified-derivatives}.
At \(t=T\), both auxiliary problems equal \(\Psi\), so the same
regularity estimate gives
\[
|\phi_{\rm sc}^\varepsilon(T,x)-\Psi(x)|
+\max_m|g_m^\varepsilon(T,x)-\Psi(x)|
\le C\int_{\Theta_\varepsilon}
 (\sqrt{\varepsilon^2-\vartheta}+|e|)\rho_\varepsilon\,d\vartheta\,de
\le C\varepsilon.
\]
Taking the minimum preserves this bound, proving both terminal estimates.
\end{proof}

	\subsection{Localized comparison and the global estimate}
	
	The next lemma turns an interior residual bound into a global bound
	on the comparison interval.  The discrepancies on its final layer and
	in the boundary strip are included in \(B_\delta^\pm\).  For \(\delta>0\), set
	\[
	I_\delta
	:=\{i:x_i\in O,\ \operatorname{dist}(x_i,\partial O)\ge\delta\},
	\qquad
	\mathcal C_\delta
	:=\{i:\operatorname{dist}(x_i,\partial O)<\delta\}.
	\]
	Fix an integer \(N_*\) with \(1\le N_*\le N_\dt\).
For a grid function \(W\) on the layers \(0,\ldots,N_*\), define
	\begin{align}
		B_\delta^+(W)
		&:=\max\left\{
		\max_i(u_{h,i}^{N_*}-W_i^{N_*})^+,
		\max_{\substack{0\le n<N_*\\i\in\mathcal C_\delta}}
		(u_{h,i}^n-W_i^n)^+\right\},
		\label{eq:bc-local-B-plus}\\
		B_\delta^-(W)
		&:=\max\left\{
		\max_i(W_i^{N_*}-u_{h,i}^{N_*})^+,
		\max_{\substack{0\le n<N_*\\i\in\mathcal C_\delta}}
		(W_i^n-u_{h,i}^n)^+\right\}.
		\label{eq:bc-local-B-minus}
	\end{align}
	
	\begin{lemma}[Localized discrete comparison]
		\label{lem:bc-local-comparison}
		Let \(E\ge0\).  If
		\begin{equation}
			\mathcal S_{h,n,i}[W_i^n,W^{n+1}]\le E
			\qquad (0\le n<N_*,\ i\in I_\delta),
			\label{eq:bc-local-super}
		\end{equation}
		then
		\begin{equation}
			\max_{\substack{0\le n\le N_*\\i\in I_\delta}}
			(u_{h,i}^n-W_i^n)^+
			\le B_\delta^+(W)+TE.
			\label{eq:bc-local-upper}
		\end{equation}
		If
		\begin{equation}
			\mathcal S_{h,n,i}[W_i^n,W^{n+1}]\ge-E
			\qquad (0\le n<N_*,\ i\in I_\delta),
			\label{eq:bc-local-sub}
		\end{equation}
		then
		\begin{equation}
			\max_{\substack{0\le n\le N_*\\i\in I_\delta}}
			(W_i^n-u_{h,i}^n)^+
			\le B_\delta^-(W)+TE.
			\label{eq:bc-local-lower}
		\end{equation}
	\end{lemma}
	
	\begin{proof}
Let \(\mathcal V_{n,i}\) contain all next-layer vertices used with
positive weight by this row for some control.  For fixed \(a\), the
source and prescribed boundary values cancel in
\(S_{n,i}^a[U]-S_{n,i}^a[V]\).  The remaining interpolation coefficients
are nonnegative and have total mass at most one.  Thus
\begin{equation}
S_{n,i}[U]-S_{n,i}[V]
\le\max_{\ell\in\mathcal V_{n,i}}(U_\ell-V_\ell)^+,
\label{eq:bc-local-nonexpansive}
\end{equation}
where an empty maximum is zero.  Here passage to the infimum uses
\(\inf_aX_a-\inf_aY_a\le\sup_a(X_a-Y_a)\).

Under \eqref{eq:bc-local-super}, choose \(a_m\) such that
\[
\mathcal R^{a_m}_{h,n,i}[W_i^n,W^{n+1}]\le E+m^{-1}.
\]
Since \(\tau_{n,i}(a_m)\le\dt\),
\[
S_{n,i}[W^{n+1}]
\le S_{n,i}^{a_m}[W^{n+1}]
\le W_i^n+\dt(E+m^{-1}).
\]
Letting \(m\to\infty\) gives
\(S_{n,i}[W^{n+1}]\le W_i^n+\dt E\).
Under \eqref{eq:bc-local-sub}, every control satisfies
\[
S_{n,i}^a[W^{n+1}]
\ge W_i^n-\tau_{n,i}(a)E\ge W_i^n-\dt E,
\]
so \(S_{n,i}[W^{n+1}]\ge W_i^n-\dt E\).
Combining these bounds with \(u_{h,i}^n=S_{n,i}[u_h^{n+1}]\)
and \eqref{eq:bc-local-nonexpansive} gives, respectively,
\begin{align*}
(u_{h,i}^n-W_i^n)^+
&\le\max_{\ell\in\mathcal V_{n,i}}
 (u_{h,\ell}^{n+1}-W_\ell^{n+1})^++\dt E,\\
(W_i^n-u_{h,i}^n)^+
&\le\max_{\ell\in\mathcal V_{n,i}}
 (W_\ell^{n+1}-u_{h,\ell}^{n+1})^++\dt E.
\end{align*}
At collar vertices and on layer \(N_*\), the corresponding discrepancy
is already bounded by \(B_\delta^\pm(W)\).
Backward induction therefore bounds it on layer \(n\) by
\(B_\delta^\pm(W)+(N_*-n)\dt E\).
Since \((N_*-n)\dt\le T\), both conclusions follow.
\end{proof}
	Fix \(C_{\rm con}\) to be a uniform upper bound for the constant in
	Theorem~\ref{thm:main-high-order-consistency}, for both families of
	regularized functions constructed in
	Proposition~\ref{prop:quantitative-comparison-functions}.  Define the
	interior consistency error per unit effective time
	\begin{equation}
		\mathfrak E_h(\varepsilon)
		:=C_{\rm con}\left(
		\dt\varepsilon^{-3}+\dt^2\varepsilon^{-5}
		+\frac{h^2}{\dt}\varepsilon^{-1}\right).
		\label{eq:bc-consistency-error}
	\end{equation}
	
	\begin{theorem}[Error bound for the stopped scheme]
		\label{thm:bc-rate}
		Assume Assumptions~\ref{ass:bounded} and
		\ref{ass:quantitative-boundary}, the mesh hypotheses
		\textup{(G1)}--\textup{(G4)} together with
		\eqref{eq:Isp}--\eqref{eq:Isp-error}, and one fixed centrally
		symmetric positive cubature.  Let \(u\) be the unique bounded
		continuous Dirichlet solution from
		Lemma~\ref{lem:quantitative-regularity}, and let \(u_h\) be defined by
		the stopped scheme \eqref{eq:bounded-scheme}.  Suppose that
		\[
		h\to0,\qquad
		\dt\to0,\qquad
		q_h:=\frac{h^2}{\dt}\to0.
		\]
		Then, for all sufficiently small meshes,
		\begin{align}
			\|(u-u_h)^+\|_{\ell^\infty(G_{h,\dt})}
			&\le
			C\left(
			\dt^{1/4}+\frac{h}{\sqrt{\dt}}
			\right),
			\label{eq:bc-main-upper-rate}\\
			\|(u_h-u)^+\|_{\ell^\infty(G_{h,\dt})}
			&\le
			C\left(
			\dt^{1/10}
			+h^{1/2}\dt^{-1/4}
			\right).
			\label{eq:bc-main-lower-rate}
		\end{align}
		Here \(C\) is independent of \(h\) and \(\dt\).
		In particular, if \(\dt\asymp h^{10/7}\), then
		\begin{equation}
			\|(u-u_h)^+\|_{\ell^\infty(G_{h,\dt})}
			=O(h^{2/7}),
			\qquad
			\|(u_h-u)^+\|_{\ell^\infty(G_{h,\dt})}
			=O(h^{1/7}),
			\qquad
			\|u_h-u\|_{\ell^\infty(G_{h,\dt})}
			=O(h^{1/7}).
			\label{eq:bc-main-final-rate}
		\end{equation}
	\end{theorem}
	
	\begin{proof}
The scalar and switching comparisons have approximation errors
\(C\varepsilon\) and \(C\varepsilon^{1/3}\), respectively.
We balance these with \(\mathfrak E_h(\varepsilon)\) and reserve a
boundary strip containing every branch and interpolation vertex.
Fix \(C_{\rm sc},C_{\rm sw}\ge1\),
\(C_{\rm col}\ge2+B_{\mathcal C}+C_{\rm patch}\), and set
\begin{align}
\varepsilon_{\rm sc}&=C_{\rm sc}\max\{\dt^{1/4},q_h^{1/2}\},
&\delta_{\rm sc}&=C_{\rm col}\varepsilon_{\rm sc},
\label{eq:bc-scalar-scale}\\
\varepsilon_{\rm sw}&=C_{\rm sw}\max\{\dt^{3/10},q_h^{3/4}\},
&\delta_{\rm sw}&=C_{\rm col}\varepsilon_{\rm sw}.
\label{eq:bc-switching-scale}
\end{align}
Both smoothing scales tend to zero.  Their definitions give
\begin{equation}
\dt\le\varepsilon_{\rm sc}^4,\quad q_h\le\varepsilon_{\rm sc}^2,
\qquad
\dt\le\varepsilon_{\rm sw}^{10/3},\quad
q_h\le\varepsilon_{\rm sw}^{4/3}.
\label{eq:bc-scale-powers}
\end{equation}
For sufficiently small meshes the scales are at most one.  Consequently,
\(\sqrt\dt\le\varepsilon_{\rm sc}^2\le\varepsilon_{\rm sc}\) and
\(h=\sqrt{q_h\dt}\le\varepsilon_{\rm sc}^3\le\varepsilon_{\rm sc}\);
likewise \(\sqrt\dt\le\varepsilon_{\rm sw}^{5/3}\le\varepsilon_{\rm sw}\)
and \(h\le\varepsilon_{\rm sw}^{7/3}\le\varepsilon_{\rm sw}\).
Our choice of \(C_{\rm col}\) therefore ensures
\begin{align}
\delta_{\rm sc}-2\varepsilon_{\rm sc}
&\ge B_{\mathcal C}\sqrt\dt+C_{\rm patch}h,
\label{eq:bc-upper-collar}\\
\delta_{\rm sw}-2\varepsilon_{\rm sw}
&\ge B_{\mathcal C}\sqrt\dt+C_{\rm patch}h.
\label{eq:bc-lower-collar}
\end{align}
The barrier smallness condition \eqref{eq:bc-main-bar-small} holds
because \(\dt^{\gamma/2}+q_h\to0\), and
Proposition~\ref{prop:quantitative-comparison-functions} applies at both
smoothing scales.

\smallskip
\noindent\emph{The final time layers.}
The shifted differential inequalities apply only up to \(T-\varepsilon^2\).
We first control the remaining layers and the endpoint of the comparison
interval.  For either smoothing scale, define
\[
N_\varepsilon:=\max\{n:t_n\le T-\varepsilon^2\}.
\]
For small meshes, \(\varepsilon^2<T/2\) and
\(\dt\le\varepsilon^2\), by \eqref{eq:bc-scale-powers}.  Therefore
\[
\varepsilon^2\le T-t_{N_\varepsilon}
<\varepsilon^2+\dt\le2\varepsilon^2.
\]
Lemma~\ref{lem:discrete-terminal-modulus} and the temporal
H\"older estimate \eqref{eq:quantitative-u-uniform} give
\[
\max_{N_\varepsilon\le n\le N_\dt}
\|u_h^n-u(t_n,\cdot)\|_{\ell^\infty(G_h)}\le C\varepsilon.
\]
For \(v=u^\varepsilon\) or \(v=v_m^{\varepsilon,k}\),
the terminal datum \(v(T)=\Psi\) and uniform regularity give
\[
\left|\int_{\Theta_\varepsilon}
 v(t_{N_\varepsilon}+\vartheta,x-e)\rho_\varepsilon
 \,d\vartheta\,de-\Psi(x)\right|\le C\varepsilon.
\]
Indeed, all sampled times lie between \(T-2\varepsilon^2\) and \(T\).
Combining this bound with Lemma~\ref{lem:discrete-terminal-modulus}
controls the comparison-function mismatch at the layer
\(N_\varepsilon\) by \(C\varepsilon\).

The estimates \eqref{eq:quant-scalar-close} and
\eqref{eq:quant-switching-close} remain valid when the mollified
function is evaluated at \(t+\varepsilon^2\) and compared with
\(u(t,x)\): the additional difference
\(|u(t+\varepsilon^2,x)-u(t,x)|\le C\varepsilon\)
is absorbed in their right-hand sides.
Write \(N_{\rm sc}=N_{\varepsilon_{\rm sc}}\) and
\(N_{\rm sw}=N_{\varepsilon_{\rm sw}}\).

\smallskip
\noindent\emph{The scalar estimate.}
For \(0\le n\le N_{\rm sc}\), define
\[
W_{{\rm sc},i}^n=
\begin{cases}
u_{h,i}^n,&\operatorname{dist}(x_i,\partial O)\le\varepsilon_{\rm sc},\\
\phi_{\rm sc}^{\varepsilon_{\rm sc}}(t_n+\varepsilon_{\rm sc}^2,x_i),
 &\operatorname{dist}(x_i,\partial O)>\varepsilon_{\rm sc}.
\end{cases}
\]
For \(0\le n<N_{\rm sc}\) and \(i\in I_{\delta_{\rm sc}}\), \eqref{eq:bc-upper-collar} keeps
the complete branch and all active interpolation vertices at least
\(2\varepsilon_{\rm sc}\) from the boundary.  Thus every value of
\(W_{\rm sc}\) used by this row is a value of
\(\phi_{\rm sc}^{\varepsilon_{\rm sc}}\).
Theorem~\ref{thm:main-high-order-consistency} and
\eqref{eq:quant-scalar-pde} give, for every control,
\[
\mathcal R_{h,n,i}^a[W_{{\rm sc},i}^n,W_{\rm sc}^{n+1}]
\ge F^a[\phi_{\rm sc}^{\varepsilon_{\rm sc}}(\cdot+\varepsilon_{\rm sc}^2,\cdot)](t_n,x_i)
 -\mathfrak E_h(\varepsilon_{\rm sc})
\ge-\mathfrak E_h(\varepsilon_{\rm sc}).
\]
Taking the infimum yields
\begin{equation}
\mathcal S_{h,n,i}[W_{{\rm sc},i}^n,W_{\rm sc}^{n+1}]
\ge-\mathfrak E_h(\varepsilon_{\rm sc}).
\label{eq:bc-upper-residual}
\end{equation}

On the collar, the mismatch vanishes when
\(\operatorname{dist}(x_i,\partial O)\le\varepsilon_{\rm sc}\).
At the remaining collar nodes,
\eqref{eq:bc-main-sandwich} and the time-shifted version of
\eqref{eq:quant-scalar-close} established above give
\[
W_{{\rm sc},i}^n-u_{h,i}^n
\le u+C\varepsilon_{\rm sc}-(G-K_{\rm bar}\zeta)
\le2K_{\rm bar}\zeta+C\varepsilon_{\rm sc}
\le C(\delta_{\rm sc}+\varepsilon_{\rm sc}).
\]
Here \(\zeta\le C\operatorname{dist}(x_i,\partial O)\) follows
from its spatial Lipschitz continuity and zero boundary trace.
The mismatch at the final comparison layer \(N_{\rm sc}\) is at most
\(C\varepsilon_{\rm sc}\), as proved above; hence
\(B_{\delta_{\rm sc}}^-(W_{\rm sc})\le
C(\delta_{\rm sc}+\varepsilon_{\rm sc})\).
Lemma~\ref{lem:bc-local-comparison}, with \(N_*=N_{\rm sc}\), gives
\[
(W_{{\rm sc},i}^n-u_{h,i}^n)^+
\le C(\delta_{\rm sc}+\varepsilon_{\rm sc})
 +T\mathfrak E_h(\varepsilon_{\rm sc}),\qquad i\in I_{\delta_{\rm sc}}.
\]
At these nodes, \eqref{eq:quant-scalar-close} adds at most
\(C\varepsilon_{\rm sc}\) when \(W_{\rm sc}\) is replaced by \(u\).
On the collar, \eqref{eq:bc-main-sandwich} directly gives
\((u-u_h)^+\le2K_{\rm bar}\zeta\le C\delta_{\rm sc}\).
Combining the two spatial regions for \(n\le N_{\rm sc}\) with
the estimate on the final time layers proves
\begin{equation}
\|(u-u_h)^+\|_{\ell^\infty(G_{h,\dt})}
\le C\left(\delta_{\rm sc}+\varepsilon_{\rm sc}
 +\dt\varepsilon_{\rm sc}^{-3}+\dt^2\varepsilon_{\rm sc}^{-5}
 +q_h\varepsilon_{\rm sc}^{-1}\right).
\label{eq:bc-unoptimized-upper}
\end{equation}

\smallskip
\noindent\emph{The switching estimate.}
For \(0\le n\le N_{\rm sw}\), define
\[
W_{{\rm sw},i}^n=
\begin{cases}
u_{h,i}^n,&\operatorname{dist}(x_i,\partial O)\le\varepsilon_{\rm sw},\\
\phi_{\rm sw}^{\varepsilon_{\rm sw}}(t_n+\varepsilon_{\rm sw}^2,x_i),
 &\operatorname{dist}(x_i,\partial O)>\varepsilon_{\rm sw}.
\end{cases}
\]
For \(0\le n<N_{\rm sw}\) and \(i\in I_{\delta_{\rm sw}}\), \eqref{eq:bc-lower-collar}
again makes the row unstopped and identifies all its sampled values
with \(\phi_{\rm sw}^{\varepsilon_{\rm sw}}\).
Choose \(m_*\in\operatorname*{argmin}_m
g_m^{\varepsilon_{\rm sw}}(t_n+\varepsilon_{\rm sw}^2,x_i)\).
By \eqref{eq:quant-switching-envelope}, the current value equals
\(g_{m_*}^{\varepsilon_{\rm sw}}\), while each next-layer value is
at most that of this component.  Positivity and
\eqref{eq:quant-switching-active} therefore imply
\begin{align*}
\mathcal S_{h,n,i}[W_{{\rm sw},i}^n,W_{\rm sw}^{n+1}]
&\le\mathcal R_{h,n,i}^{a_{m_*}}
 [g_{m_*}^{\varepsilon_{\rm sw}}(t_n+\varepsilon_{\rm sw}^2,x_i),
  g_{m_*}^{\varepsilon_{\rm sw}}(t_{n+1}+\varepsilon_{\rm sw}^2,\cdot)]\\
&\le F^{a_{m_*}}[g_{m_*}^{\varepsilon_{\rm sw}}(\cdot+\varepsilon_{\rm sw}^2,\cdot)](t_n,x_i)
 +\mathfrak E_h(\varepsilon_{\rm sw})
\le\mathfrak E_h(\varepsilon_{\rm sw}).
\end{align*}
Thus
\begin{equation}
\mathcal S_{h,n,i}[W_{{\rm sw},i}^n,W_{\rm sw}^{n+1}]
\le\mathfrak E_h(\varepsilon_{\rm sw}).
\label{eq:bc-lower-residual}
\end{equation}

The collar mismatch is zero where \(W_{\rm sw}=u_h\).
Elsewhere on the collar, \eqref{eq:bc-main-sandwich} and the
time-shifted version of \eqref{eq:quant-switching-close} established above give
\[
u_{h,i}^n-W_{{\rm sw},i}^n
\le G+K_{\rm bar}\zeta-(u-C\varepsilon_{\rm sw}^{1/3})
\le2K_{\rm bar}\zeta+C\varepsilon_{\rm sw}^{1/3}
\le C(\delta_{\rm sw}+\varepsilon_{\rm sw}^{1/3}).
\]
At the final comparison layer \(N_{\rm sw}\), the mismatch is
\(C\varepsilon_{\rm sw}\le C\varepsilon_{\rm sw}^{1/3}\), so
\(B_{\delta_{\rm sw}}^+(W_{\rm sw})\le
C(\delta_{\rm sw}+\varepsilon_{\rm sw}^{1/3})\).
Lemma~\ref{lem:bc-local-comparison}, with \(N_*=N_{\rm sw}\), and
\eqref{eq:quant-switching-close} yield
\[
(u_{h,i}^n-u(t_n,x_i))^+
\le C(\delta_{\rm sw}+\varepsilon_{\rm sw}^{1/3})
 +T\mathfrak E_h(\varepsilon_{\rm sw}),\qquad i\in I_{\delta_{\rm sw}}.
\]
On the collar, the barrier sandwich gives
\((u_h-u)^+\le C\delta_{\rm sw}\).  Combining this with the
\(C\varepsilon_{\rm sw}\) estimate on the final time layers gives
\begin{equation}
\|(u_h-u)^+\|_{\ell^\infty(G_{h,\dt})}
\le C\left(\delta_{\rm sw}+\varepsilon_{\rm sw}^{1/3}
 +\dt\varepsilon_{\rm sw}^{-3}+\dt^2\varepsilon_{\rm sw}^{-5}
 +q_h\varepsilon_{\rm sw}^{-1}\right).
\label{eq:bc-unoptimized-lower}
\end{equation}

\smallskip
\noindent\emph{Choice of scales.}
Using \eqref{eq:bc-scale-powers} term by term gives
\begin{align*}
\dt\varepsilon_{\rm sc}^{-3}
 +\dt^2\varepsilon_{\rm sc}^{-5}+q_h\varepsilon_{\rm sc}^{-1}
&\le2\varepsilon_{\rm sc}+\varepsilon_{\rm sc}^3
\le3\varepsilon_{\rm sc},\\
\dt\varepsilon_{\rm sw}^{-3}
 +\dt^2\varepsilon_{\rm sw}^{-5}+q_h\varepsilon_{\rm sw}^{-1}
&\le2\varepsilon_{\rm sw}^{1/3}+\varepsilon_{\rm sw}^{5/3}
\le3\varepsilon_{\rm sw}^{1/3}.
\end{align*}
Together with \(\delta_{\rm sc}=C_{\rm col}\varepsilon_{\rm sc}\)
and \(\delta_{\rm sw}=C_{\rm col}\varepsilon_{\rm sw}
\le C\varepsilon_{\rm sw}^{1/3}\), these estimates prove
\eqref{eq:bc-main-upper-rate}--\eqref{eq:bc-main-lower-rate}.
Finally, \(\dt\asymp h^{10/7}\) gives \(q_h\asymp h^{4/7}\).
Substitution yields the two one-sided orders in
\eqref{eq:bc-main-final-rate}; their maximum is the full error.
\end{proof}

	\section{Numerical experiments}
\label{sec:numerics}

We consider cold-wall heat loss and three controlled diffusion problems.
The disk problems use \eqref{eq:bounded-HJB} with its infimum convention.

For all disk computations, $O=\{x\in\mathbb R^2:|x|<1\}$ and $T=1/4$.
The mesh consists of the origin and the points
\[
x_{k,j}=\frac{k}{N_r}\left(\cos\frac{2\pi j}{N_\theta},
\sin\frac{2\pi j}{N_\theta}\right),\qquad
1\le k\le N_r,\quad 0\le j<N_\theta,\quad
N_\theta=\lceil2\pi N_r\rceil.
\]
Each annular cell is split from its inner vertex at angle $\theta_j$ to its
outer vertex at angle $\theta_{j+1}$; the first ring is joined to the origin.
We let $h$ be the maximum triangle diameter; the radial resolutions are specified below.
Interpolation is linear on each triangle. Between the boundary polygon and
the circle, the two adjacent boundary values are interpolated linearly in
the polar angle. Appendix~\ref{app:numerical-mesh} verifies the mesh hypotheses.
In Examples 2 and 4, the cubature has
$\eta_1=(1,1)^\top$, $\eta_2=(1,-1)^\top$, and
$\rho_1=\rho_2=1/2$.
For each signed direction, the branch is stopped at the first positive root of
\[
|x_i\pm s\sigma^a(x_i)\eta_j+s^2\mu^a(x_i)|^2=1,
\qquad 0<s\le\sqrt{\Delta t};
\]
its duration is $s^2$, or $\Delta t$ if there is no contact. The resulting
weights and update are those of \eqref{eq:bounded-scheme}.

We set $N_t=N_{\Delta t}$ and $\Delta t=T/N_t$. Errors include all spatial
nodes, with time levels specified below and consecutive order
$p_\infty=\log(E_\infty^{\rm coarse}/E_\infty^{\rm fine})/
\log(h_{\rm coarse}/h_{\rm fine})$.
Examples 1--3 were computed in MATLAB R2022a on 64-bit Windows;
Example 4 was computed in Python.

\subsection{Example 1: parallel heat loss to a cold wall}
\label{sec:num-disk}

We model cooling of a hot spot to a zero-temperature wall, using the
constant-field specialization of the magnetized heat model in
\cite{NarskiOttaviani2014}. Perpendicular conduction is neglected on
the faster parallel diffusion time scale. The nondimensional data are
\[
A=\{0\},\qquad
\sigma^0=\binom{\sqrt3/2}{1/2},\qquad
A^0=\begin{pmatrix}3/4&\sqrt3/4\\\sqrt3/4&1/4\end{pmatrix},
\qquad\mu^0=\binom00,\qquad r^0=f^0=0,
\]
\[
u_\partial=0,\qquad\Psi(x)=1-|x|^2.
\]
This gives the linear, rank-one equation
$u_t+\tfrac38u_{x_1x_1}+\tfrac{\sqrt3}{4}u_{x_1x_2}
+\tfrac18u_{x_2x_2}=0$.
The temperature after physical time $s$ is $u(T-s,\cdot)$.

Solving the heat equation on each field-line chord gives an independent
reference. Set $b=\sigma^0$,
$\xi=b\cdot x$, $z=(-1/2,\sqrt3/2)\cdot x$,
$\ell=\sqrt{1-z^2}$, and $\alpha_k=(k+1/2)\pi$. For $t<T$ and $x\in O$,
\begin{equation}
u(t,x)=4\ell^2\sum_{k=0}^{\infty}\frac{(-1)^k}{\alpha_k^3}
\cos\!\left(\frac{\alpha_k\xi}{\ell}\right)
\exp\!\left[-\frac{\alpha_k^2(T-t)}{2\ell^2}\right].
\label{eq:num-heat-exact}
\end{equation}
The series is truncated with a uniform tail bound below $10^{-14}$.
We report $E_\infty=\max_i|u_{h,i}^0-u(0,x_i)|$, including all grid nodes.

The two branches are $x_i\pm\sqrt{\Delta t}\,b$, with truncated lengths
$s_i^\pm$. On the same mesh, the truncated LISL generator
\cite{ReisingerRotaetxe2017} is
\[
(L_hv)_i=\frac{S_{n,i}^0[v]-v_i}{s_i^+s_i^-}.
\]
Here the zero boundary value makes $S_{n,i}^0$ linear and time independent.
Forward Euler uses a time step at
most $0.95\min_i(-1/(L_h)_{ii})$. Crank--Nicolson uses time step
$\Delta t$ and one sparse LU factorization, reused throughout the computation.
Endpoint monitoring retains the full branches with weights $1/2$,
assigning zero to endpoints outside the disk.

We use $N_r=24,48,96,192$ and
$N_t=\lceil T/(ch)\rceil$, $c\in\{1/4,1/2,1\}$.
Each $(N_r,c)$ uses common interpolation and branch scales. Timings are
medians of three alternating runs, including assembly, factorization and
marching, but excluding reference evaluation and file output.

\begin{table}[H]
\caption{Example 1. Full mesh sequence, $N_t=\lceil2T/h\rceil$ and $\Delta t=T/N_t$.
Time-update counts: $9,18,35,68$; explicit LISL: $38,108,299,834$.}
\label{tab:num-heat-comparison}
\centering\small
\setlength{\tabcolsep}{1.75pt}
\begin{tabular}{@{}rr*{4}{rrr}@{}}
\toprule
& & \multicolumn{3}{c}{Stopped scheme} & \multicolumn{3}{c}{Explicit LISL} & \multicolumn{3}{c}{CN-LISL, direct LU} & \multicolumn{3}{c}{Endpoint monitoring}\\
\cmidrule(lr){3-5}\cmidrule(lr){6-8}\cmidrule(lr){9-11}\cmidrule(l){12-14}
$N_r$ & $h$ & $E_\infty$ & $p_\infty$ & Time (s) & $E_\infty$ & $p_\infty$ & Time (s) & $E_\infty$ & $p_\infty$ & Time (s) & $E_\infty$ & $p_\infty$ & Time (s)\\
24 & $5.827\mathrm e{-2}$ & $6.388\mathrm e{-3}$ & -- & 0.00593 & $5.497\mathrm e{-3}$ & -- & 0.00718 & $6.350\mathrm e{-3}$ & -- & 0.036 & $1.227\mathrm e{-1}$ & -- & 0.00605\\
48 & $2.929\mathrm e{-2}$ & $3.183\mathrm e{-3}$ & 1.013 & 0.0256 & $3.680\mathrm e{-3}$ & 0.583 & 0.0354 & $3.903\mathrm e{-3}$ & 0.708 & 0.25 & $8.381\mathrm e{-2}$ & 0.554 & 0.0247\\
96 & $1.468\mathrm e{-2}$ & $1.642\mathrm e{-3}$ & 0.959 & 0.0894 & $1.927\mathrm e{-3}$ & 0.937 & 0.251 & $2.015\mathrm e{-3}$ & 0.957 & 2.57 & $7.564\mathrm e{-2}$ & 0.148 & 0.0943\\
192 & $7.354\mathrm e{-3}$ & $8.083\mathrm e{-4}$ & 1.025 & 0.556 & $1.008\mathrm e{-3}$ & 0.938 & 2.95 & $1.037\mathrm e{-3}$ & 0.961 & 31.4 & $5.195\mathrm e{-2}$ & 0.544 & 0.564\\
\bottomrule
\end{tabular}
\end{table}

The stopped scheme has observed orders $0.959$--$1.025$ and, from $N_r=48$
onward, smaller errors than both LISL schemes. The LISL orders approach one;
endpoint monitoring has larger errors and lower, irregular orders.
On this sequence, the explicit-LISL/stopped time ratio rises from 1.21 to 5.30.

\subsection{Example 2: control-dependent anisotropic diffusion}
\label{sec:num-finite-hjb}

We now solve \eqref{eq:bounded-HJB} with $A=\{-1,0,1\}$,
$\mu^a=0$, $r^a=0.35$, and
\[
\sigma^a=
\begin{pmatrix}
\sqrt{0.08}\cos\theta_a&-\sin\theta_a\\
\sqrt{0.08}\sin\theta_a&\cos\theta_a
\end{pmatrix},\qquad
\theta_a=\frac{(30+20a)\pi}{180},\qquad A^a=\sigma^a(\sigma^a)^\top.
\]
The control therefore selects one of three orientations of the diffusion.
Let $k=\pi/2$ and prescribe the exact solution
\begin{equation}
\begin{aligned}
u_{\rm ex}(t,x)=e^{-t}\bigl[2
&+0.30\sin(kx_1)\cos(kx_2)
+0.20\cos(2kx_1)\sin(kx_2)\\
&+0.15\sin(2kx_1)\sin(2kx_2)\bigr].
\end{aligned}
\label{eq:num-disk-exact}
\end{equation}
Set
$g(x)=0.9\sin(\pi x_1/2)\cos(\pi x_2/2)$, and prescribe
\begin{equation}
\begin{aligned}
f^a(t,x)={}&1.35u_{\rm ex}(t,x)
-\tfrac12\tr\bigl(A^aD^2u_{\rm ex}(t,x)\bigr)\\
&+20\left[(a-g(x))^2-\inf_{b\in A}(b-g(x))^2\right].
\end{aligned}
\label{eq:num-hjb-source}
\end{equation}
All derivatives here are those of the explicit function
\eqref{eq:num-disk-exact}. The terminal and lateral data are
$\Psi=u_{\rm ex}(T,\cdot)$ and $u_\partial=u_{\rm ex}|_{[0,T)\times\partial O}$.
Substitution shows that $F^a[u_{\rm ex}]$ equals the nonnegative last term
of \eqref{eq:num-hjb-source}; its infimum is zero. An optimal control for
this exact solution is the element of $A$ nearest to $g(x)$, with ties at
$g(x)=\pm1/2$. This construction makes the minimizing control vary in space.

We use $N_r=12,24,48,96,192,384$ and
$N_t=4,11,28,76,203,546$, respectively, so $\Delta t\asymp h^{10/7}$.
Here $E_\infty=\max_{n,i}|u_{h,i}^n-u_{\rm ex}(t_n,x_i)|$ includes
all time layers as well as all spatial nodes.

\begin{table}[H]
\caption{Example 2. Maximum errors for the three-control problem,
using $\Delta t=T/N_t\asymp h^{10/7}$.}
\label{tab:num-finite-hjb}
\centering\small
\begin{tabular}{@{}rrrr@{}}
\toprule
$h$ & $N_t$ & $E_\infty$ & $p_\infty$\\
\midrule
$1.149\mathrm e{-1}$ & 4 & $4.028\mathrm e{-2}$ & --\\
$5.827\mathrm e{-2}$ & 11 & $1.450\mathrm e{-2}$ & 1.505\\
$2.929\mathrm e{-2}$ & 28 & $5.878\mathrm e{-3}$ & 1.312\\
$1.468\mathrm e{-2}$ & 76 & $2.455\mathrm e{-3}$ & 1.264\\
$7.354\mathrm e{-3}$ & 203 & $1.055\mathrm e{-3}$ & 1.222\\
$3.680\mathrm e{-3}$ & 546 & $5.146\mathrm e{-4}$ & 1.036\\
\bottomrule
\end{tabular}
\end{table}

The maximum error decreases at every refinement. The last consecutive
order is $1.036$. Errors are measured on the full grid, including the nodes
whose branches meet the boundary.

\subsection{Example 3: a finite-volatility \texorpdfstring{$G$}{G}-heat equation}
\label{sec:num-peng-gheat}

On $\mathbb T^2=[0,1)^2$, take $T=1$, $A=\{1,2,3,4\}$,
$\mu^a=0$, $r^a=0$, and the one-column diffusion factors
\[
\sigma^1=\binom{0.40}{0},\quad
\sigma^2=\binom{0}{0.75},\quad
\sigma^3=\frac{0.90}{\sqrt2}\binom11,\quad
\sigma^4=\frac{0.60}{\sqrt2}\binom1{-1}.
\]
With $G(M)=\tfrac12\sup_{a\in A}\tr(A^aM)$ and
$A^a=\sigma^a(\sigma^a)^\top$, the terminal-value problem is
\[
\begin{cases}
u_t+G(D^2u)+f=0,&(t,x)\in(0,1)\times\mathbb T^2,\\
u(1,x)=\Psi(x),&x\in\mathbb T^2,
\end{cases}
\]
with periodicity in both spatial coordinates. This example uses
$\sup_{a\in A}F^a[u]=0$, so the update takes a supremum of the complete
fixed-control rows. Define
\[
\Psi(x)=2+0.20\cos(2\pi x_1)+0.15\sin(2\pi x_2)
+0.10\cos(2\pi(x_1+x_2)),
\]
\[
D^2\Psi=-4\pi^2
\begin{pmatrix}
0.20\cos(2\pi x_1)+0.10\cos(2\pi(x_1+x_2))&0.10\cos(2\pi(x_1+x_2))\\
0.10\cos(2\pi(x_1+x_2))&0.15\sin(2\pi x_2)+0.10\cos(2\pi(x_1+x_2))
\end{pmatrix}.
\]
The common source $f^a=f=e^{1-t}[\Psi-G(D^2\Psi)]$ gives
$u_{\rm ex}(t,x)=e^{1-t}\Psi(x)$.
We use a uniform periodic $N\times N$ grid with spacing $h=1/N$,
$N_t=N/4$, and positive bilinear interpolation. Each fixed control has
the two equally weighted branches $x\pm\sqrt{\Delta t}\sigma^a$.
Here $E_\infty=\max_{i,j}|u_{h,ij}^0-u_{\rm ex}(0,ih,jh)|$ is measured at $t=0$.

\begin{table}[H]
\caption{Example 3. Maximum errors at $t=0$ for the periodic problem.}
\label{tab:num-peng-gheat}
\centering\small
\begin{tabular}{@{}rrrr@{}}
\toprule
$N$ & $N_t$ & $E_\infty$ & $p_\infty$\\
\midrule
20 & 5 & $1.508\mathrm e{0}$ & --\\
40 & 10 & $4.866\mathrm e{-1}$ & 1.632\\
80 & 20 & $2.131\mathrm e{-1}$ & 1.192\\
160 & 40 & $9.728\mathrm e{-2}$ & 1.131\\
320 & 80 & $4.564\mathrm e{-2}$ & 1.092\\
640 & 160 & $2.186\mathrm e{-2}$ & 1.062\\
\bottomrule
\end{tabular}
\end{table}

The maximum error decreases on all six meshes, and the consecutive order
approaches one. This periodic computation tests the nonlinear diffusion
update without a lateral boundary.

\subsection{Example 4: controlled drift and rank-one diffusion}
\label{sec:num-drifted-boundary}

We return to the unit disk with $T=1/4$, $A=\{-1,0,1\}$, and $r^a=0.35$.
Consider \eqref{eq:bounded-HJB} with
\[
\mu^a(x)=\begin{pmatrix}0.18x_1-0.10x_2+0.08a\\0.10x_1+0.18x_2+0.04a\end{pmatrix},
\qquad
\sigma^a=\begin{pmatrix}
\nu\cos\theta_a&-\sin\theta_a\\
\nu\sin\theta_a&\cos\theta_a
\end{pmatrix},\qquad \theta_a=\frac{(30+20a)\pi}{180}.
\]
We compare $\nu=\sqrt{0.08}$ and $\nu=0$, keeping the drift, control set,
data, and mesh sequence fixed. In the second case $A^a$ has rank one.
With $u_{\rm ex}$ as in \eqref{eq:num-disk-exact} and
$g(x)=0.9\sin(\pi x_1/2)\cos(\pi x_2/2)$, set
\[
\begin{aligned}
f^a={}&1.35u_{\rm ex}-\mu^a\cdot Du_{\rm ex}
-\tfrac12\tr\bigl(\sigma^a(\sigma^a)^\top D^2u_{\rm ex}\bigr)\\
&+20\left[(a-g(x))^2-\inf_{b\in A}(b-g(x))^2\right],\\
\Psi(x)={}&u_{\rm ex}(T,x),\qquad
u_\partial(t,x)=u_{\rm ex}(t,x)\quad (x\in\partial O).
\end{aligned}
\]
Thus $u_{\rm ex}$ solves the problem for both values of $\nu$.
We use the $\Delta t\asymp h^{10/7}$ sequence from Example 2.

\begin{table}[H]
\caption{Example 4. Maximum errors with controlled drift. Both cases use
the same nodes and time steps; all nodes and time layers enter the error.}
\label{tab:num-drifted-boundary}
\centering\small
\begin{tabular}{@{}rrrrrr@{}}
\toprule
& & \multicolumn{2}{c}{$\nu=\sqrt{0.08}$} & \multicolumn{2}{c}{$\nu=0$}\\
\cmidrule(lr){3-4}\cmidrule(l){5-6}
$h$ & $N_t$ & $E_\infty$ & $p_\infty$ & $E_\infty$ & $p_\infty$\\
\midrule
$1.149\mathrm e{-1}$ & 4 & $4.008\mathrm e{-2}$ & -- & $4.129\mathrm e{-2}$ & --\\
$5.827\mathrm e{-2}$ & 11 & $1.468\mathrm e{-2}$ & 1.479 & $1.540\mathrm e{-2}$ & 1.452\\
$2.929\mathrm e{-2}$ & 28 & $6.012\mathrm e{-3}$ & 1.298 & $6.343\mathrm e{-3}$ & 1.290\\
$1.468\mathrm e{-2}$ & 76 & $2.507\mathrm e{-3}$ & 1.267 & $2.655\mathrm e{-3}$ & 1.261\\
$7.354\mathrm e{-3}$ & 203 & $1.079\mathrm e{-3}$ & 1.219 & $1.210\mathrm e{-3}$ & 1.137\\
$3.680\mathrm e{-3}$ & 546 & $5.232\mathrm e{-4}$ & 1.045 & $5.626\mathrm e{-4}$ & 1.106\\
\bottomrule
\end{tabular}
\end{table}

The two error sequences are close. Removing the smaller diffusion
eigenvalue does not prevent convergence in this example: the final errors
are $5.232\times10^{-4}$ and $5.626\times10^{-4}$, with consecutive orders
$1.045$ and $1.106$, respectively.

\section{Conclusions}
\label{sec:conclusions}

We extend first-contact stopping and pair balancing to general fixed centrally
symmetric positive degree-two cubatures, allowing unequal pair masses.
For degenerate parabolic Dirichlet HJB equations on bounded domains,
the stopped weights
preserve generator moments despite unequal branch durations.
Positive interpolation gives monotonicity, discrete comparison, and uniform
stability without diagonal dominance or uniform ellipticity.

The main technical contribution is to retain the control-dependent effective
time \(\tau\) in both the stopped Taylor remainder and the weighted
interpolation error. In particular, the bound \(2\tau/\Delta t\) on
the interpolated branch mass prevents consistency from deteriorating as
\(\tau/\Delta t\) vanishes. This gives the uniform discrete barrier
estimates needed near the Dirichlet boundary.

Combined with interior moment cancellations, a discrete terminal estimate,
and the bounded-domain shaking and switching framework, these bounds yield
the global error estimates of Theorem~\ref{thm:bc-rate}. Under the stated
coefficient, boundary, and mesh assumptions, including a common strict
boundary barrier, they apply to nonsmooth viscosity solutions and give an
\(O(h^{1/7})\) maximum-norm bound along \(\Delta t\asymp h^{10/7}\).

The cold-wall refinement study, evaluated against an independent series
solution, shows near-first-order errors, smaller maximum errors than
endpoint monitoring, and a growing time advantage over explicit truncated
LISL. The controlled anisotropic and Peng G-heat examples show error decay
for nonlinear diffusion; the drifted example retains comparable accuracy
after one diffusion eigenvalue is set to zero.

\appendix
	
	\section{Paired-branch Taylor estimate in the contracted interior}
	\label{app:high-order-interior}
	
	At the interior nodes considered in
	Theorem~\ref{thm:main-high-order-consistency}, every branch has duration
	\(\dt\).  Hence the effective time and the branch weights reduce to
	\begin{equation}
		\tau_{n,i}(a)=\dt,
		\qquad
		\omega_{n,i}^{a,j,\pm}=\frac{\rho_j}{2}.
		\label{eq:app-no-stop-weights}
	\end{equation}
	This appendix proves the Taylor estimate used for such complete branches.
	The stopped-row calculation, where the two branches in a pair may have
	different durations, is treated separately in
	Proposition~\ref{prop:stopped-consistency}.
	
	Fix \((t,x,a)\), freeze the coefficients at this point, and let
	\(0<d\le1\).  For the positive and negative branches in pair \(j\), write
	\begin{equation}
		y_j^\pm
		:=d\mu^a\pm\sqrt d\,\sigma^a\eta_j,
		\qquad 1\le j\le M.
		\label{eq:app-paired-increments}
	\end{equation}
	Recall that
	\[
	\sum_{j=1}^M\rho_j=1,
	\qquad
	\sum_{j=1}^M\rho_j\eta_j\eta_j^\top=I,
	\qquad
	A^a=\sigma^a(\sigma^a)^\top.
	\]
	
All derivatives in Lemma~\ref{lem:app-cubature-moments} are evaluated at
\((t,x)\).  To write the higher Taylor terms, use the directional derivative
\begin{equation}
(v\cdot D_x)^q\phi(t,x)
:=\left.\frac{d^q}{d\lambda^q}\phi(t,x+\lambda v)\right|_{\lambda=0}.
\label{eq:app-directional-derivative-notation}
\end{equation}
Here \(v\) is held fixed when differentiating.  In particular,
\((v\cdot D_x)\phi=D\phi\cdot v\) and
\((v\cdot D_x)^2\phi=v^\top D^2\phi\,v\).
We use the derivative norm
\begin{equation}
\|D_x^q\phi(t,x)\|
:=\sup_{|v_1|,\ldots,|v_q|\le1}
\big|(v_1\cdot D_x)\cdots(v_q\cdot D_x)\phi(t,x)\big|.
\label{eq:app-derivative-operator-norm}
\end{equation}
Thus \(|(v\cdot D_x)^q\phi|\le\|D_x^q\phi\||v|^q\).
A subscript \(\infty\) also takes the supremum over the region in question.

\begin{lemma}[Cancellation between paired branches]
		\label{lem:app-cubature-moments}
		For every smooth \(\phi\),
		\begin{align}
			\sum_{j=1}^M\frac{\rho_j}{2}
			\left(
			D\phi\cdot y_j^+
			+D\phi\cdot y_j^-
			\right)
			&=d\mu^a\cdot D\phi,
			\label{eq:app-paired-first-moment}\\
			\sum_{j=1}^M\frac{\rho_j}{2}
			\left(
			(y_j^+)^\top D_x^2\phi\,y_j^+
			+(y_j^-)^\top D_x^2\phi\,y_j^-
			\right)
			&=d\tr(A^aD_x^2\phi)
			+d^2(\mu^a)^\top D_x^2\phi\,\mu^a.
			\label{eq:app-paired-second-moment}
		\end{align}
		For \(q=3,4,5,6\),
		\begin{equation}
\left|\sum_{j=1}^M\frac{\rho_j}{2}
 \left((y_j^+\cdot D_x)^q\phi+(y_j^-\cdot D_x)^q\phi\right)\right|
\le
\begin{cases}
Cd^2\|D_x^q\phi\|, & q=3,4,\\
Cd^3\|D_x^q\phi\|, & q=5,6.
\end{cases}
\label{eq:app-paired-higher-moments}
\end{equation}
		and
		\begin{equation}
			\sum_{j=1}^M\frac{\rho_j}{2}
			\left(|y_j^+|^q+|y_j^-|^q\right)
			\le Cd^{q/2},
			\qquad q=2,4,6.
			\label{eq:app-paired-absolute-moments}
		\end{equation}
		The constants are uniform in \(a,t,x,d\).
	\end{lemma}
	
\begin{proof}
The two increments in \eqref{eq:app-paired-increments} satisfy
\(y_j^++y_j^-=2d\mu^a\).  Multiplying by \(D\phi\), averaging with
weights \(\rho_j/2\), and using \(\sum_j\rho_j=1\) proves
\eqref{eq:app-paired-first-moment}.
For the quadratic term, direct expansion gives
\[
\frac12\left((y_j^+)^\top D^2\phi\,y_j^+
             +(y_j^-)^\top D^2\phi\,y_j^-\right)
=d^2(\mu^a)^\top D^2\phi\,\mu^a
 +d(\sigma^a\eta_j)^\top D^2\phi\,\sigma^a\eta_j.
\]
The two mixed terms have opposite signs and cancel.  Multiplication by
\(\rho_j\) and summation give \eqref{eq:app-paired-second-moment}, since
\(\sum_j\rho_j(\sigma^a\eta_j)(\sigma^a\eta_j)^\top=A^a\).

For the higher terms, the frozen directional derivatives commute, so
\[
\begin{aligned}
&\frac12\left((y_j^+\cdot D_x)^q\phi
                 +(y_j^-\cdot D_x)^q\phi\right)\\
&\qquad=\sum_{\substack{0\le k\le q\\k\ {\rm even}}}
\binom qk d^{q-k/2}
 (\mu^a\cdot D_x)^{q-k}(\sigma^a\eta_j\cdot D_x)^k\phi.
\end{aligned}
\]
Indeed, a term with \(k\) diffusion factors changes by \((-1)^k\)
between the two signs; only even \(k\) remain.  The smallest power of
\(d\) is obtained at the largest even \(k\):
\[
\begin{array}{c|rrrr}
q&3&4&5&6\\ \hline
\text{largest even }k&2&4&4&6\\
\text{smallest power}&d^2&d^2&d^3&d^3
\end{array}
\]
The coefficient bounds, \eqref{eq:app-derivative-operator-norm}, and
\(d\le1\) therefore give \eqref{eq:app-paired-higher-moments}
after averaging in \(j\).
Finally,
\[
|y_j^\pm|\le d|\mu^a|+\sqrt d\,|\sigma^a\eta_j|\le C\sqrt d.
\]
Taking the \(q\)-th power and using \(\sum_j\rho_j=1\) proves
\eqref{eq:app-paired-absolute-moments}.
\end{proof}
	
	\begin{lemma}[Sixth-order Taylor estimate for complete branches]
		\label{lem:app-sixth-order-taylor}
		Suppose that all derivatives
		\(D_x^\alpha\partial_t^q\phi\) with
		\(|\alpha|+2q\le6\) are continuous and bounded on the space--time region
		containing \([t,t+d]\times\bigcup_{j,\pm}[x,x+y_j^\pm]\).  Then
		\begin{align}
			&\sum_{j=1}^M\frac{\rho_j}{2}
			\left[
			\phi(t+d,x+y_j^+)
			+\phi(t+d,x+y_j^-)
			\right]
			\notag\\
			&\quad=
			\phi(t,x)
			+d\left(
			\partial_t\phi
			+\mu^a\cdot D\phi
			+\frac12\tr(A^aD^2\phi)
			\right)(t,x)
			+R_d^a[\phi].
			\label{eq:app-paired-taylor-expansion}
		\end{align}
		Moreover,
		\begin{align}
			|R_d^a[\phi]|
			\le{}&Cd^2\Bigl(
			\|D_x^2\phi\|_\infty
			+\|D_x^3\phi\|_\infty
			+\|D_x^4\phi\|_\infty
			\Bigr)
			\notag\\
			&+Cd^2\Bigl(
			\|D_x\partial_t\phi\|_\infty
			+\|D_x^2\partial_t\phi\|_\infty
			+\|\partial_{tt}\phi\|_\infty
			\Bigr)
			\notag\\
			&+Cd^3\Bigl(
			\|D_x^5\phi\|_\infty
			+\|D_x^3\partial_t\phi\|_\infty
			+\|D_x\partial_{tt}\phi\|_\infty
			\Bigr)
			\notag\\
			&+Cd^3\Bigl(
			\|D_x^6\phi\|_\infty
			+\|D_x^4\partial_t\phi\|_\infty
			+\|D_x^2\partial_{tt}\phi\|_\infty
			+\|\partial_{ttt}\phi\|_\infty
			\Bigr).
			\label{eq:app-sixth-order-remainder}
		\end{align}
	\end{lemma}
	
\begin{proof}
We first expand one branch, then average the two signs to identify the
terms of order \(d\).  Fix \(y=y_j^\pm\).  Taylor expansion in time gives
\[
\phi(t+d,x+y)=\phi(t,x+y)+d\partial_t\phi(t,x+y)
 +\frac{d^2}{2}\partial_{tt}\phi(t,x+y)+r_t,
\qquad |r_t|\le Cd^3\|\partial_{ttt}\phi\|_\infty.
\]
Expand the three displayed functions in space to orders five, three,
and one, respectively.  Since \(|y|=O(\sqrt d)\), each omitted term
then has size \(O(d^3)\) times its derivative norm.  This gives
\begin{equation}
\begin{aligned}
\phi(t+d,x+y_j^\pm)
={}&\sum_{m=0}^{5}\frac{(y_j^\pm\cdot D_x)^m\phi}{m!}\\
&+d\sum_{m=0}^{3}\frac{(y_j^\pm\cdot D_x)^m\partial_t\phi}{m!}\\
&+\frac{d^2}{2}\left(\partial_{tt}\phi
                  +D\partial_{tt}\phi\cdot y_j^\pm\right)
+r_{j,\pm},
\end{aligned}
\label{eq:app-sequential-taylor}
\end{equation}
where all displayed derivatives are at \((t,x)\), and
\begin{align}
|r_{j,\pm}|\le C\Bigl(&
 \|D_x^6\phi\|_\infty|y_j^\pm|^6
 +d\|D_x^4\partial_t\phi\|_\infty|y_j^\pm|^4\notag\\
&+d^2\|D_x^2\partial_{tt}\phi\|_\infty|y_j^\pm|^2
 +d^3\|\partial_{ttt}\phi\|_\infty\Bigr).
\label{eq:app-integral-remainder-bound}
\end{align}
The four terms come, in order, from the three spatial remainders and
the time remainder.

Multiply \eqref{eq:app-sequential-taylor} by \(\rho_j/2\) and sum
over \(j\) and both signs.  We treat its three lines in turn.

The first line has constant, linear, and quadratic contributions
\[
\phi+d\mu^a\cdot D\phi
 +\frac d2\tr(A^aD^2\phi)
 +\frac{d^2}{2}(\mu^a)^\top D^2\phi\,\mu^a,
\]
by \eqref{eq:app-paired-first-moment}--\eqref{eq:app-paired-second-moment}.
Its cubic and quartic terms are bounded by
\(Cd^2(\|D_x^3\phi\|_\infty+\|D_x^4\phi\|_\infty)\),
and its fifth-order term by \(Cd^3\|D_x^5\phi\|_\infty\),
by \eqref{eq:app-paired-higher-moments}.

For the second line, apply the same identities to \(\partial_t\phi\)
and multiply by \(d\).  The constant term is \(d\partial_t\phi\),
while the linear and quadratic terms are
\[
d^2\mu^a\cdot D\partial_t\phi
 +\frac{d^2}{2}\tr(A^aD^2\partial_t\phi)
 +\frac{d^3}{2}(\mu^a)^\top D^2\partial_t\phi\,\mu^a.
\]
These are bounded by
\(Cd^2(\|D_x\partial_t\phi\|_\infty+
\|D_x^2\partial_t\phi\|_\infty)\), using \(d^3\le d^2\).
The cubic term is bounded by \(Cd^3\|D_x^3\partial_t\phi\|_\infty\).
The third line averages exactly to
\[
\frac{d^2}{2}\partial_{tt}\phi
 +\frac{d^3}{2}\mu^a\cdot D\partial_{tt}\phi.
\]

Thus the constant and order-\(d\) terms give
\eqref{eq:app-paired-taylor-expansion}; all other displayed terms are
bounded by the first three lines of \eqref{eq:app-sixth-order-remainder}.
Finally, \eqref{eq:app-paired-absolute-moments} gives
\[
\sum_{j,\pm}\frac{\rho_j}{2}|r_{j,\pm}|
\le Cd^3\left(
\|D_x^6\phi\|_\infty+\|D_x^4\partial_t\phi\|_\infty
+\|D_x^2\partial_{tt}\phi\|_\infty+\|\partial_{ttt}\phi\|_\infty
\right).
\]
This is the last line of \eqref{eq:app-sixth-order-remainder} and
completes the proof.
\end{proof}
	
	\begin{lemma}[Truncation bound at the smoothing scale]
		\label{lem:app-mollified-truncation}
		Let \(0<\varepsilon\le1\), and suppose
		\begin{equation}
			\|\phi\|_\infty\le C_{\rm reg},
			\qquad
			\|D_x^\alpha\partial_t^q\phi\|_\infty
			\le C_{\rm reg}\varepsilon^{1-|\alpha|-2q},
			\qquad
			1\le|\alpha|+2q\le6.
			\label{eq:app-mollified-derivatives}
		\end{equation}
		Then the remainder in \eqref{eq:app-paired-taylor-expansion} satisfies
		\begin{equation}
			|R_d^a[\phi]|
			\le
			C\left(
			d^2\varepsilon^{-3}
			+d^3\varepsilon^{-5}
			\right),
			\label{eq:app-average-truncation}
		\end{equation}
		uniformly in the control.
	\end{lemma}
	
	\begin{proof}
		Every derivative multiplied by \(d^2\) in
		\eqref{eq:app-sixth-order-remainder} has parabolic order between two and
		four.  Since \(0<\varepsilon\le1\), all of these derivatives are bounded by
		\(C\varepsilon^{-3}\).  Every derivative multiplied by \(d^3\) has
		parabolic order five or six and is therefore bounded by
		\(C\varepsilon^{-5}\).  Substitution into
		\eqref{eq:app-sixth-order-remainder} proves
		\eqref{eq:app-average-truncation}.
	\end{proof}
	
	Together, Lemmas~\ref{lem:app-sixth-order-taylor} and
	\ref{lem:app-mollified-truncation} give exactly the complete-branch expansion
	used in Theorem~\ref{thm:main-high-order-consistency}.

	\section{Auxiliary barriers and estimates near the terminal time}
	\label{app:auxiliary-barrier}
	
	Lemma~\ref{lem:extended-cylinder-barrier} supplies the common strict
	barrier required for auxiliary regularity in
	Lemma~\ref{lem:quantitative-regularity}.  The scalar and switching
	cylinders extend to \(-\varepsilon^2\) and \(-2\varepsilon^2\);
	we construct the barrier on the larger cylinder, with coefficients
	evaluated at \((t-\theta,x+e)\).
	Lemma~\ref{lem:discrete-terminal-modulus} separately controls the
	numerical solution near \(T\), as needed in Theorem~\ref{thm:bc-rate}.
	
	\begin{lemma}[Barrier on the extended auxiliary cylinders]
		\label{lem:extended-cylinder-barrier}
		Let
		\begin{equation}
			H_\varepsilon
			:=
			T+2\varepsilon^2,
			\qquad
			\lambda_\varepsilon
			:=
			\frac{T}{H_\varepsilon},
			\qquad
			\chi_\varepsilon(t)
			:=
			\lambda_\varepsilon
			\bigl(t+2\varepsilon^2\bigr),
			\label{eq:extended-time-compression}
		\end{equation}
		and define
		\begin{equation}
			\zeta_{\rm ext}^{\varepsilon}(t,x)
			:=
			\frac{2}{\beta}\,
			\zeta\bigl(\chi_\varepsilon(t),x\bigr),
			\qquad
			(t,x)\in
			[-2\varepsilon^2,T]\times\overline O.
			\label{eq:extended-cylinder-barrier}
		\end{equation}
		For all sufficiently small \(\varepsilon>0\),
		\(\zeta_{\rm ext}^{\varepsilon}\) is positive on
		\([ -2\varepsilon^2,T)\times O\), nonnegative on
		\(\{T\}\times\overline O\), and vanishes on
		\([ -2\varepsilon^2,T]\times\partial O\).  Moreover, there is a
		constant \(C\), independent of \(\varepsilon\), such that
		\begin{equation}
			\left\|
			\zeta_{\rm ext}^{\varepsilon}
			\right\|_{
				2,\gamma;
				[-2\varepsilon^2,T]\times\overline O
			}
			\leq C.
			\label{eq:extended-cylinder-barrier-norm}
		\end{equation}
		
		For every \(a\in A\),
		\((\theta,e)\in\Theta_\varepsilon\), and
		\((t,x)\in[-2\varepsilon^2,T]\times O\),
		\begin{align}
			&\partial_t
			\zeta_{\rm ext}^{\varepsilon}(t,x)
			+\frac12\tr\!\left(
			A^a(t-\theta,x+e)
			D^2\zeta_{\rm ext}^{\varepsilon}(t,x)
			\right)
			\notag\\
			&\qquad
			+\mu^a(t-\theta,x+e)
			\cdot D\zeta_{\rm ext}^{\varepsilon}(t,x)
			\notag\\
			&\qquad
			-r^a(t-\theta,x+e)
			\zeta_{\rm ext}^{\varepsilon}(t,x)
			\leq-1.
			\label{eq:extended-cylinder-barrier-residual}
		\end{align}
		
		The lateral lifting \(\widehat G\) satisfies
		\begin{equation}
			\left\|
			\widehat G
			\right\|_{
				2,\gamma;
				[-2\varepsilon^2,T]\times\overline O
			}
			\leq C,
			\label{eq:extended-lifting-norm}
		\end{equation}
		and the terminal--lateral compatibility estimate
		\begin{equation}
			\left|
			\Psi(x)-\widehat G(T,x)
			\right|
			\leq
			\frac{\beta C_{\rm comp}}{2}\,
			\zeta_{\rm ext}^{\varepsilon}(T,x),
			\qquad
			x\in\overline O,
			\label{eq:extended-cylinder-compatibility}
		\end{equation}
		holds with a constant independent of \(\varepsilon\).
		
		The same conclusions hold on the scalar cylinder
		\([-\varepsilon^2,T]\times\overline O\) by restriction.
	\end{lemma}
	
	\begin{proof}
The increasing affine map \(\chi_\varepsilon\) sends
\([-2\varepsilon^2,T]\) onto \([0,T]\), with
\(\chi_\varepsilon(T)=T\) and \(0<\lambda_\varepsilon\le1\).
Only the barrier uses this map; the auxiliary terminal remains \(T\).
Thus positivity, the zero lateral trace, and the uniform norm bound
\eqref{eq:extended-cylinder-barrier-norm} follow from \textup{(Q1)}.
For \((\theta,e)\in\Theta_\varepsilon\),
\[
|t-\theta-\chi_\varepsilon(t)|
\le\varepsilon^2+(1-\lambda_\varepsilon)(T+2\varepsilon^2)
=3\varepsilon^2,\qquad |e|\le\varepsilon.
\]
Consequently the coefficient extensions satisfy, for
\(b\in\{A,\mu,r\}\),
\begin{equation}
|b^a(t-\theta,x+e)-b^a(\chi_\varepsilon(t),x)|\le C\varepsilon.
\label{eq:extended-coefficient-distance}
\end{equation}
For \(\chi=\chi_\varepsilon(t)\), \eqref{eq:extended-coefficient-distance} gives
\begin{align*}
&\tfrac12\left|\tr\!\left(
 (A^a(t-\theta,x+e)-A^a(\chi,x))D^2\zeta(\chi,x)\right)\right|\\
&\quad+\left|(\mu^a(t-\theta,x+e)-\mu^a(\chi,x))\cdot D\zeta(\chi,x)\right|
 +\left|(r^a(t-\theta,x+e)-r^a(\chi,x))\zeta(\chi,x)\right|\\
&\qquad\le C\varepsilon
 (\|D^2\zeta\|_\infty+\|D\zeta\|_\infty+\|\zeta\|_\infty),
\end{align*}
and the time correction satisfies
\[
|(\lambda_\varepsilon-1)\zeta_t(\chi,x)|
\le \frac{2\varepsilon^2}{T}\|\zeta_t\|_\infty.
\]
Together with \textup{(Q2)}, these bounds place the left-hand side of
\eqref{eq:extended-cylinder-barrier-residual} below
\((2/\beta)(-\beta+C_{\rm sh}\varepsilon)\), with fixed \(C_{\rm sh}\).
For \(C_{\rm sh}\varepsilon_0\le\beta/2\), this is at most \(-1\).

The lifting bound follows from
\eqref{eq:quantitative-G-extension-bound} and
\([-2\varepsilon^2,T]\subset[-T/2,T]\).
Finally, \(\zeta_{\rm ext}^\varepsilon(T)=2\zeta(T)/\beta\) and
\(\widehat G(T)=G(T)\), so \textup{(Q3)} gives
\eqref{eq:extended-cylinder-compatibility}.  Restriction to the smaller
scalar cylinder preserves all these properties.
\end{proof}
	
\begin{lemma}[Estimate near the terminal time]
\label{lem:discrete-terminal-modulus}
Assume Assumption~\ref{ass:bounded}, conditions \textup{(Q1)} and
\textup{(Q3)} of Assumption~\ref{ass:quantitative-boundary}, and the
mesh and interpolation hypotheses of Theorem~\ref{thm:bc-rate}.
For \(0<\dt\le1\) and \(q_h=h^2/\dt\le1\), the solution of
\eqref{eq:bounded-scheme} satisfies
\begin{equation}
\|u_h^n-\Psi|_{G_h}\|_{\ell^\infty(G_h)}
\le C\sqrt{T-t_n},\qquad 0\le n\le N_\dt,
\label{eq:discrete-terminal-modulus}
\end{equation}
where \(C\) is independent of the mesh and of \(n\).
\end{lemma}

\begin{proof}
The case \(n=N_\dt\) is the terminal condition.  Fix an arbitrary
\(n<N_\dt\) and \(x_i\in G_h\), and set
\[
\eta=\sqrt{T-t_n},\qquad P(x)=\sqrt{\eta^2+|x-x_i|^2},\qquad
\Phi_\pm(t,x)=\Psi(x_i)\pm L_\Psi P(x)\pm K\frac{T-t}{\eta}.
\]
Here \(L_\Psi\) is a Lipschitz constant of \(\Psi\),
\(\sqrt\dt\le\eta\le\sqrt T\), \(\|DP\|_\infty\le1\), and
\(\|D^2P\|_\infty\le\eta^{-1}\).
We keep \(n,x_i,\eta\) fixed, verify the discrete upper and lower
inequalities on all layers \(t_n\le t_m\le T\), and finally evaluate
the resulting bounds at \((t_n,x_i)\).  The time term will dominate
the increase from spatial averaging and interpolation.

Since \(P(x)\ge|x-x_i|\), Lipschitz continuity gives
\(\Phi_-(T,x)\le\Psi(x)\le\Phi_+(T,x)\).
On \(\partial O\), conditions \textup{(Q1)} and \textup{(Q3)} imply
\(u_\partial=G\) and \(G(T,x)=\Psi(x)\); hence
\[
|u_\partial(t,x)-\Psi(x_i)|
\le |G(t,x)-G(T,x)|+|\Psi(x)-\Psi(x_i)|
\le \|\partial_tG\|_\infty(T-t)+L_\Psi|x-x_i|.
\]
Thus \(K\ge\sqrt T\,\|\partial_tG\|_\infty\) ensures the boundary
ordering at every contact time, including those between grid layers.

Fix any interior row \((m,\ell,a)\), \(n\le m<N_\dt\), and suppress
its indices.  With coefficients frozen there, put
\(s_j^\pm=\sqrt{d_j^\pm}\) and
\(y_j^\pm=d_j^\pm\mu\pm s_j^\pm\sigma\eta_j\).
Taylor's inequality gives
\[
P(x_\ell+y_j^\pm)
\le P(x_\ell)+DP(x_\ell)\cdot y_j^\pm+\frac{|y_j^\pm|^2}{2\eta}.
\]
By \eqref{eq:weight-moments},
\(\sum\omega_j^\pm y_j^\pm=\tau\mu\) and
\(\sum\omega_j^\pm|y_j^\pm|^2\le C\tau\).
Using \(\eta\le\sqrt T\), we obtain
\[
\sum_{j,\pm}\omega_j^\pm P(x_\ell+y_j^\pm)
\le P(x_\ell)+C\tau/\eta.
\]
Let \(\mathcal P_j^\pm\) be \(P\) at a stopped contact, or
\(\Isp[P|_{G_h}]\) at an unstopped full-step endpoint.
These values are nonnegative.  Equations~\eqref{eq:Isp-error} and
\eqref{eq:interpolated-branch-mass} give
\[
\sum_{j,\pm}\omega_j^\pm\mathcal P_j^\pm
\le P(x_\ell)+\frac{C\tau}{\eta}
   +C\frac{\tau h^2}{\dt}(1+\eta^{-1})
\le P(x_\ell)+\frac{C\tau}{\eta},
\]
where the last bound uses \(q_h\le1\) and \(\eta\le\sqrt T\).

Set \(r_{\max}=\|r\|_\infty\) and \(q_j^\pm=(1+rd_j^\pm)^{-1}\).
Since \(r\ge0\) and \(d_j^\pm\le\dt\),
\[
\sum_{j,\pm}\omega_j^\pm(1-q_j^\pm)\le r_{\max}\tau,\qquad
\sum_{j,\pm}\omega_j^\pm q_j^\pm d_j^\pm
\ge\frac{\tau}{1+r_{\max}\dt}.
\]
Also \(\sum\omega_j^\pm q_j^\pm\mathcal P_j^\pm
\le P(x_\ell)+C\tau/\eta\).
To offset this spatial increase, write \(H_m=T-t_m\) and compute the
change in the remaining time:
\[
\begin{aligned}
\sum_{j,\pm}\omega_j^\pm q_j^\pm(H_m-d_j^\pm)-H_m
&=-H_m\sum_{j,\pm}\omega_j^\pm(1-q_j^\pm)
  -\sum_{j,\pm}\omega_j^\pm q_j^\pm d_j^\pm\\
&\le-\frac{\tau}{1+r_{\max}\dt}.
\end{aligned}
\]
The interpolant preserves constants.  At stopped contacts, the prescribed
value lies between \(\Phi_-\) and \(\Phi_+\), as proved above.
Expanding the three terms of \(\Phi_\pm\) in the row therefore yields
\[
\begin{aligned}
S_{m,\ell}^a[\Phi_+(t_{m+1},\cdot)|_{G_h}]-\Phi_+(t_m,x_\ell)
&\le\tau\left(C+\frac{CL_\Psi}{\eta}
             -\frac{K}{(1+r_{\max}\dt)\eta}\right),\\
S_{m,\ell}^a[\Phi_-(t_{m+1},\cdot)|_{G_h}]-\Phi_-(t_m,x_\ell)
&\ge-\tau\left(C+\frac{CL_\Psi}{\eta}
             -\frac{K}{(1+r_{\max}\dt)\eta}\right).
\end{aligned}
\]
The constant term is bounded by
\((r_{\max}\|\Psi\|_\infty+\|f\|_\infty)\tau\);
the source contributes exactly \(f\tau\).
In addition to the boundary requirement, choose
\(K\ge(1+r_{\max})(C\sqrt T+CL_\Psi)\).
Since \(\dt\le1\) and \(\eta\le\sqrt T\), the parentheses are
nonpositive in every row, uniformly in \(n,i,h,\dt\).
Taking the infimum over \(a\) gives the discrete upper and lower
inequalities.  Proposition~\ref{prop:bounded-monotone}, applied
backwards from the terminal layer, now gives
\[
\Phi_-(t_m,x_\ell)\le u_{h,\ell}^m\le\Phi_+(t_m,x_\ell),
\qquad n\le m\le N_\dt.
\]
At the fixed target point, \(P(x_i)=\eta=(T-t_n)/\eta\), so
\[
|u_{h,i}^n-\Psi(x_i)|\le(L_\Psi+K)\sqrt{T-t_n}.
\]
The same \(K\) works for every choice of \(n,i\).
Taking the maximum over \(i\) proves \eqref{eq:discrete-terminal-modulus}.
\end{proof}

	\section{Taylor remainder along stopped branches}
\label{app:bc-barrier-comparison}

The stopped-barrier proof uses a Taylor estimate along the actual
branch curve, so all sampled points remain in \(\overline O\).
The following lemma supplies its \(s^{2+\gamma}\) remainder.

\begin{lemma}[Branchwise parabolic Taylor remainder]
		\label{lem:bc-app-branch-remainder}
		Let \(0<\gamma\le1\) and \(\phi\in C^{2,\gamma}(Q_T^{\rm cl})\), let
		\(0<s\le\sqrt\dt\le1\), and put
		\[
		z=\pm\sigma_i^n(a)\eta_j,\qquad
		\mu=\mu_i^n(a),\qquad
		p(r)=rz+r^2\mu,\qquad y=p(s).
		\]
		Assume
		\[
		x_i+p(r)\in\overline O,\qquad0\le r\le s,
		\qquad
		x_i+p(r)\in O,\qquad0\le r<s,
		\qquad t_n+s^2\le T.
		\]
		Then, with all derivatives on the right evaluated at \((t_n,x_i)\),
		\begin{align}
			\phi(t_n+s^2,x_i+y)
			={}&\phi(t_n,x_i)+s^2\partial_t\phi
			+s^2D\phi\cdot\mu+sD\phi\cdot z\notag\\
			&+\frac12s^2z^\top D^2\phi\,z
			+s^3\mu^\top D^2\phi\,z
			+\frac12s^4\mu^\top D^2\phi\,\mu
			+\mathfrak r_j^\pm,
			\label{eq:bc-app-branch-taylor}
		\end{align}
		where
		\begin{equation}
			\abs{\mathfrak r_j^\pm}
			\le C\norm{\phi}_{2,\gamma;Q_T^{\rm cl}}s^{2+\gamma}.
			\label{eq:bc-app-one-branch-rem}
		\end{equation}
		The constant depends only on the common coefficient bounds and the fixed
		cubature.
	\end{lemma}
	
\begin{proof}
Write \(t=t_n\), \(x=x_i\), and \(y=p(s)\).  We prove the expansion
by separating its time and space remainders.  The coefficient bounds give
\[
|p(r)|\le Cr,\qquad |p'(r)|=|z+2r\mu|\le C,
\qquad 0\le r\le s,
\]
so in particular \(|y|\le Cs\).

First, integrate in time at the endpoint \(x+y\), adding and subtracting
\(\partial_t\phi(t,x)\):
\begin{align}
\phi(t+s^2,x+y)
&=\phi(t,x+y)+s^2\partial_t\phi(t,x)+R_t,\notag\\
R_t&=\int_0^{s^2}
 \bigl(\partial_t\phi(t+q,x+y)-\partial_t\phi(t,x)\bigr)\,dq.
\label{eq:bc-app-time-remainder}
\end{align}
The parabolic \(\gamma\)-H\"older continuity of \(\partial_t\phi\) gives
\begin{equation}
|R_t|\le[\partial_t\phi]_{\gamma;Q_T^{\rm cl}}
 \int_0^{s^2}(\sqrt q+|y|)^\gamma\,dq
 \le C[\partial_t\phi]_{\gamma;Q_T^{\rm cl}}s^{2+\gamma}.
\label{eq:bc-app-time-remainder-bound}
\end{equation}
Here the integration length is \(s^2\), and the variation of
\(\partial_t\phi\) is bounded by \(Cs^\gamma\).

For the spatial expansion, integrate along \(x+p(r)\), which stays in
\(\overline O\).  Subtracting the linear approximation of the gradient
gives the exact identity
\begin{equation}
\begin{aligned}
&D\phi(t,x+p(r))-D\phi(t,x)-D^2\phi(t,x)p(r)\\
&\qquad=\int_0^r
 \bigl(D^2\phi(t,x+p(q))-D^2\phi(t,x)\bigr)p'(q)\,dq.
\end{aligned}
\label{eq:bc-app-space-remainder}
\end{equation}
Multiply this identity by \(p'(r)\) using the inner product and integrate
from \(0\) to \(s\).  The three terms on the left integrate to
\(\phi(t,x+y)-\phi(t,x)\), \(-D\phi(t,x)\cdot y\), and
\(-\tfrac12 y^\top D^2\phi(t,x)y\), respectively.  Consequently,
\begin{align}
\phi(t,x+y)
={}&\phi(t,x)+D\phi(t,x)\cdot y
 +\tfrac12 y^\top D^2\phi(t,x)y+R_x,\notag\\
|R_x|
\le{}&C[D^2\phi]_{\gamma;Q_T^{\rm cl}}
       \int_0^s\int_0^r q^\gamma\,dq\,dr
 \le C[D^2\phi]_{\gamma;Q_T^{\rm cl}}s^{2+\gamma}.
\label{eq:bc-app-space-remainder-bound}
\end{align}
Indeed, the Hessian difference in \eqref{eq:bc-app-space-remainder}
is bounded by \([D^2\phi]_{\gamma;Q_T^{\rm cl}}|p(q)|^\gamma\le Cq^\gamma\),
and both factors \(p'\) are bounded.

Combining the two expansions leaves the remainder
\(\mathfrak r_j^\pm=R_t+R_x\), which satisfies
\eqref{eq:bc-app-one-branch-rem}.  Finally, substitute \(y=sz+s^2\mu\):
\[
D\phi\cdot y=sD\phi\cdot z+s^2D\phi\cdot\mu,
\]
\[
\tfrac12 y^\top D^2\phi\,y
=\tfrac12s^2z^\top D^2\phi\,z
 +s^3\mu^\top D^2\phi\,z
 +\tfrac12s^4\mu^\top D^2\phi\,\mu.
\]
These are exact identities, giving all the terms in
\eqref{eq:bc-app-branch-taylor}.
\end{proof}

	\section{Mesh and boundary hypotheses for the disk examples}
\label{app:numerical-mesh}

This appendix checks that the disk discretization used in
Section~\ref{sec:numerics} falls within the assumptions of the error analysis.
These checks do not enter the numerical update.

\begin{lemma}[Polar disk mesh]
\label{lem:num-polar-mesh}
The polar triangulations in Section~\ref{sec:numerics}, with positive
angular interpolation between the boundary polygon and the circle,
satisfy \textup{(G1)}--\textup{(G4)} and
\eqref{eq:Isp}--\eqref{eq:Isp-error}, uniformly in $N_r$.
\end{lemma}
\begin{proof}
The triangulation is conforming, every triangle is nondegenerate, and
each boundary triangle has at most one boundary edge. The small angles
at the origin are permitted by \textup{(G3)}.
Put $z(\theta)=(\cos\theta,\sin\theta)$ and
$\Delta\theta=2\pi/N_\theta$. Define $p_h(x)=x$ in the polygon and,
for $x=rz(\theta)$ in a boundary crescent, set
\[
p_h(x)=(1-\alpha)z(\theta_j)+\alpha z(\theta_{j+1}),\qquad
\alpha=(\theta-\theta_j)/\Delta\theta.
\]
Since $|z''|=1$, linear interpolation gives
$|p_h(z(\theta))-z(\theta)|\le(\Delta\theta)^2/8$.
The chord radius is
$\cos(\Delta\theta/2)/\cos(\theta-\theta_j-\Delta\theta/2)$,
so a crescent point also satisfies
$0\le1-r\le(\Delta\theta)^2/8$. Hence
$|p_h(x)-x|\le(\Delta\theta)^2/4\le Ch^2$, which gives the
Hausdorff and evaluation-map estimates.

The active vertices lie within $Ch$ of $x$, and their weights are
nonnegative and sum to one. Strict convexity of the disk gives the
segment condition in \textup{(G4)}. At $y=p_h(x)$ the active barycentric
weights, including the two chord weights, satisfy
$\sum_i\lambda_i x_i=y$. Taylor expansion about $y$ therefore yields
\[
\left|\varphi(x)-\sum_i\lambda_i\varphi(x_i)\right|
\le |x-y|\|D\varphi\|_\infty
+\tfrac12h^2\|D^2\varphi\|_\infty
\le Ch^2\|\varphi\|_{C^2}.
\]
This proves the interpolation estimate without a minimum-angle bound.
\end{proof}

For the disk examples, take $\zeta=1-|x|^2$.
In Example 1, $G=0$ gives
$\tfrac12\tr(A^0D^2\zeta)=-1$ and $|\Psi-G(T)|=\zeta$,
so \textup{(Q1)}--\textup{(Q3)} hold. In Examples 2 and 4,
take $G=u_{\rm ex}$; \textup{(Q1)} and \textup{(Q3)} then hold
with zero terminal discrepancy. In Example 2, the left side of
\textup{(Q2)} is at most $-1.08$.
In Example 4, let
$M_\mu=\sqrt{0.18^2+0.10^2}+\sqrt{0.08^2+0.04^2}$. Since
$|\mu^a|\le M_\mu$, both diffusion ranks satisfy
\begin{equation}
\tfrac12\tr(A^aD^2\zeta)+\mu^a\cdot D\zeta-r^a\zeta
\le-(1+\nu^2)+2M_\mu=-\beta_\nu<0.
\label{eq:num-drifted-barrier}
\end{equation}
Here $\beta_{\sqrt{0.08}}\simeq0.4893$ and $\beta_0\simeq0.4093$.

\end{document}